\documentclass[11pt]{article} 
\usepackage[left=1in,top=1in,right=1in,bottom=1in,letterpaper]{geometry}

\usepackage{algorithm,algorithmic}
\usepackage{lipsum}
\usepackage{graphicx,amsmath,amsthm}
\graphicspath{{Plots/}}
\usepackage{epstopdf}
\usepackage{algorithmic}
\ifpdf
  \DeclareGraphicsExtensions{.eps,.pdf,.png,.jpg}
\else
  \DeclareGraphicsExtensions{.eps}
\fi
\usepackage{microtype}
\usepackage{subfigure}
\usepackage{booktabs} 
\usepackage{amsfonts}
\usepackage{makecell}
\usepackage{comment}
\usepackage{enumitem}
\usepackage{tikz}

\usepackage[capitalise]{cleveref}

\Crefname{ALC@unique}{Line}{Lines}

\newcommand{\M}{\mathcal{M}}
\newcommand{\R}{\mathbb{R}}
\newcommand{\E}{\mathbb{E}}
\newcommand{\St}{\mathrm{St}}
\newcommand{\grad}{\mathrm{grad}}
\newcommand{\hess}{\mathrm{Hess}}

\newcommand{\proj}{\operatorname{Proj}}
\newcommand{\argmin}{\operatorname{argmin}}
\newcommand{\Exp}{\operatorname{Exp}}
\newcommand{\N}{\mathcal{N}}
\newcommand{\X}{\mathcal{X}}
\newcommand{\Tr}{\mathrm{Tr}}
\newcommand{\st}{\mathrm{ s.t. }}

\makeatletter
\@addtoreset{equation}{section}
\makeatother

\allowdisplaybreaks

\newtheorem{theorem}{Theorem}[section]
\newtheorem{proposition}{Proposition}[section]
\newtheorem{lemma}{Lemma}[section]

\newtheorem{remark}{Remark}[section]
\newtheorem{corollary}{Corollary}[section]

\newtheorem{definition}{Definition}[section]
\newtheorem{assumption}{Assumption}[section]
\crefname{hypothesis}{Hypothesis}{Hypotheses}

\title{Stochastic Zeroth-order Riemannian Derivative Estimation and Optimization}
\author{Jiaxiang Li\thanks{Department of Mathematics, University of California, Davis.\texttt{jxjli@ucdavis.edu}.} 
\and Krishnakumar Balasubramanian\thanks{Department of Statistics, University of California, Davis. \texttt{kbala@ucdavis.edu}.}
\and Shiqian Ma\thanks{Department of Mathematics, University of California, Davis. \texttt{sqma@ucdavis.edu}.}
}

\begin{document}
\maketitle

\begin{abstract}
We consider stochastic zeroth-order optimization over Riemannian submanifolds embedded in Euclidean space, where the task is to solve Riemannian optimization problem with only noisy objective function evaluations. Towards this, our main contribution is to propose estimators of the Riemannian gradient and Hessian from noisy objective function evaluations, based on a Riemannian version of the Gaussian smoothing technique. The proposed estimators overcome the difficulty of the non-linearity of the manifold constraint and the issues that arise in using Euclidean Gaussian smoothing techniques when the function is defined only over the manifold. We use the proposed estimators to solve Riemannian optimization problems in the following settings for the objective function: (i) stochastic and gradient-Lipschitz (in both nonconvex and geodesic convex settings), (ii) sum of gradient-Lipschitz and non-smooth functions, and (iii) Hessian-Lipschitz. For these settings, we analyze the oracle complexity of our algorithms to obtain appropriately defined notions of $\epsilon$-stationary point or $\epsilon$-approximate local minimizer. Notably, our complexities are independent of the dimension of the ambient Euclidean space and depend only on the intrinsic dimension of the manifold under consideration. We demonstrate the applicability of our algorithms by simulation results and real-world applications on black-box stiffness control for robotics and black-box attacks to neural networks.
\end{abstract}
\section{Introduction}
Consider the following Riemannian optimization problem:
\begin{equation}\label{problem}
    \min \ f(x)+h(x), \ \st, \ x\in\M,
\end{equation}
where $\M$ is a Riemannian submanifold embedded in $\R^n$, $f:\M\rightarrow\R$ is a smooth and possibly nonconvex function, and $h:\mathbb{R}^n \to \mathbb{R}$ is a convex and nonsmooth function. Here, convexity and smoothness are interpreted as the function is being considered in the ambient Euclidean space. Iterative algorithms for solving \eqref{problem} usually require the gradient and Hessian information of the objective function. However, in many applications, the analytical form of the function $f$ (or $h$) and its gradient are not available, and we can only obtain noisy function evaluations via a zeroth-order oracle. This setting, termed as \emph{stochastic zeroth-order Riemannian optimization}, generalizes stochastic zeroth-order Euclidean optimization (i.e., when $\M\equiv\R^n$ in \eqref{problem}), a topic which goes back to the early works of~\cite{matyas1965random, nelder1965simplex, nemirovsky1983problem} in the 1960's; see also \cite{conn2009introduction, audet2017derivative, larson2019derivative} for recent books and surveys. 

In the Euclidean setting, two popular techniques for estimating the gradient from (noisy) function queries include the finite-differences method~\cite{spall2005introduction} and the Gaussian smoothing techniques~\cite{nemirovsky1983problem}. Earlier works in this setting focused on using the estimated gradient to obtain asymptotic convergence rates of iterative optimization algorithms. Recently, obtaining non-asymptotic guarantees on the oracle complexity of stochastic zeroth-order optimization has been of great interest. Towards that, \cite{nesterov2011random,nesterov2017random} analyzed the Gaussian smoothing technique for estimating the Euclidean gradient from noisy function evaluations and proved that for unconstrained convex minimization, one needs $O(n^2/\epsilon^2)$ noisy function evaluations to obtain an $\epsilon$-optimal solution. 
{This complexity was improved by~\cite{ghadimi2013stochastic} to $O(n/\epsilon^2)$ when the objective function is further assumed to be gradient-smooth}. Note that this oracle complexity depends linearly on the problem dimension $n$ and it was proved that the linear dependency on $n$ is unavoidable \cite{jamieson2012query, duchi2015optimal}. Nonconvex and smooth setting was also considered in \cite{ghadimi2013stochastic}. In particular, now assuming $h\equiv 0$  and $\M\equiv \R^n$ in~\eqref{problem}, it was shown that the number of function evaluations for obtaining an $\epsilon$-stationary point $\bar{x}$ (i.e., $\E\| \nabla f(\bar{x})\| \leq \epsilon$), is $O(n/\epsilon^4)$.


In the Riemannian manifold setting, however, a main challenge in designing and analyzing zeroth-order algorithms is the lack of availability of theoretically sound methods to estimate the Riemannian gradients and Hessians from (noisy) function evaluations.
To this end, our main contribution in this work is to construct estimators of the Riemannian gradient and Hessian from noisy function evaluations, based on a modified Gaussian smoothing technique from~\cite{nesterov2017random} and~\cite{balasubramanian2019zeroth}. The main difficulty addressed here is that the gradient and Hessian estimator in~\cite{nesterov2017random} and~\cite{balasubramanian2019zeroth} respectively, require computing $f(x+\nu u)$, for some parameter $\nu>0$ and an $n$-dimensional standard Gaussian vector $u$. However, the point $x+\nu u$ may not necessarily lie on the manifold $\mathcal{M}$. To resolve this issue, we propose an estimator based on the smoothing technique and sampling Gaussian random vectors on the tangent space of the manifold $\mathcal{M}$. 
We then use the developed methodology to design stochastic zeroth-order algorithms for solving \eqref{problem} with oracle complexities that depend only on the manifold dimension $d$, and independent of the ambient Euclidean dimension $n$.

\subsection{Related Works}

As mentioned previously, zeroth-order optimization has a long history; we refer the reader to \cite{conn2009introduction, audet2017derivative, larson2019derivative} for more details. The  oracle complexity of  methods from the above works are at least linear in terms of their dependence on dimensionality. Recent works in this field have been focusing on stochastic zeroth-order optimization in high-dimensions \cite{wang2018stochastic,golovin2019gradientless, balasubramanian2019zeroth, cai2020zeroth}. Assuming a sparse structure (for example, the function being optimized depends only on $s$ of the $n$ coordinates), the above works have shown that the oracle complexity of zeroth-order optimization depends only poly-logarithmically on the dimension $n$, and it has a linear dependency only on the sparsity parameter $s$, which is typically small compared to $n$ in several applications. Compared to these works, we assume a manifold structure on the function being optimized and obtain oracle complexities that depend only on the manifold dimension and independent of the ambient Euclidean dimension.

Apart from the above, \emph{Bayesian optimization} is yet another popular class of methods for optimizing functions based on noisy function values. This approach aims at finding the global minimizer by enforcing a Gaussian process prior on the space of function being optimized and using Bayesian sampling techniques. 
We refer the reader to~\cite{mockus1994application, mockus2012bayesian, shahriari2015taking, frazier2018tutorial} for an overview of such techniques in the Euclidean settings and their applications to a variety of fields including robotics, recommender systems, preference learning and hyperparameter tuning. A common limitation of the above algorithms is that they are usually not scalable well to solve high-dimensional problems. 
Recent developments on Bayesian optimization for high-dimensional problems include 
\cite{li2016high, wang2016bayesian, mutny2018efficient, rolland2018high, wang2020learning}
where people considered zeroth-order optimization with structured functions (for example, sparse or additive functions), and developed Bayesian optimization algorithms and related analysis. Very recently,~\cite{oh2018bock, jaquier2020bayesian, jaquier2020high} considered heuristic Bayesian optimization algorithms for function defined over non-Euclidean domains, including Riemannian domains, without any theoretical analysis.

Riemannian optimization in the first or second-order setting has drawn a lot of attention recently due to its applications in various fields, including low-rank matrix completion \cite{boumal2011rtrmc,vandereycken2013low}, phase retrieval \cite{bendory2017non,sun2018geometric}, dictionary learning \cite{cherian2016riemannian,sun2016complete}, dimensionality reduction \cite{harandi2017dimensionality,tripuraneni2018averaging,mishra2019riemannian} and manifold regression~\cite{lin2017extrinsic,lin2020robust}. For smooth Riemannian optimization, i.e., $h\equiv 0$ in \eqref{problem}, it was shown that Riemannian gradient descent method require $\mathcal{O}(1/{\epsilon^{2}})$ iterations to converge to an $\epsilon$-stationary point \cite{boumal2018global}. Stochastic algorithms were also studied for smooth Riemannian optimization  \cite{bonnabel2013stochastic,zhou2019faster,weber2019nonconvex,zhang2016fast,kasai2018riemannian,zhou2019faster,weber2019nonconvex}. In particular, using the SPIDER variance reduction technique, \cite{zhou2019faster} proved that $\mathcal{O}(1/{\epsilon^{3}})$ oracle calls are required to obtain an $\epsilon$-stationary point in expectation. When the function $f$ takes a finite-sum structure, the Riemannian SVRG \cite{zhang2016fast} achieves $\epsilon$-stationary solution with $\mathcal{O}({k^{2/3}}/{\epsilon^{2}})$ oracle calls where $k$ is number of summands. When the nonsmooth function $h$ presents in \eqref{problem},  Riemannian sub-gradient methods (RSGM)  are widely used  \cite{borckmans2014riemannian,li2019nonsmooth} and they require $\mathcal{O}(1/{\epsilon^{4}})$ iterations. ADMM for solving \eqref{problem} has also been studied \cite{kovnatsky2016madmm,Lai2014}, but they usually lack convergence guarantee, while the analysis presented in \cite{Zhang-Ma-Zhang-2017} requires some strong assumptions. The recently proposed manifold proximal gradient method (ManPG) \cite{chen2018proximal} for solving \eqref{problem} requires $\mathcal{O}(1/{\epsilon^{2}})$ number of iterations to find an $\epsilon$-stationary solution. Variants of ManPG such as ManPPA \cite{chen2019manifold}, ManPL \cite{wang-manpl-2020} and stochastic ManPG \cite{Wang-stochastic-ManPG} have also been studied. Note that none of these works considers the zeroth-order setting. 
Recently, there are some attempts on stochastic zeroth-order Riemannian optimization \cite{chattopadhyay2015derivative, fong2019stochastic}, but they are mostly heuristics and do not have any rigorous convergence guarantees. 

\subsection{Motivating Applications}\label{sec:motiveapp}
Our motivation for developing a theoretical framework for stochastic zeroth-order Riemannian optimization is due to several important emerging applications; see, e.g., \cite{chattopadhyay2015derivative, marco2017design, yuan2019bayesian,  jaquier2020bayesian, kachan2020persistent}. Below, we discuss two concrete examples, which we will revisit in Section~\ref{sec:revisitrealworld}, to illustrate the applicability of the methods developed in this work. We also briefly discuss a third application in topological data analysis, and numerical experiments on this application will be conducted in a future work, as it is more involved and beyond the scope of this paper. 

\subsubsection{Black-box Stiffness Control for Robotics}\label{sec:roboticsapp}

Our first motivating application is from the field of robotics. It has become increasingly common to use zeroth-order optimization techniques to optimize control parameter and policies in robotics~\cite{marco2016automatic, driess2017constrained, yuan2019bayesian}. This is because that the cost functions being optimized in robotics are not available in a closed form as a function of the control parameter. Invariably for a given choice of control parameter, the cost function needs to be evaluated through a real-world experiment on a given robot or through simulation. Recently, domain knowledge has been used as constraints on the control parameter space, among which a common choice is the geometry-aware constraint. For example, control parameters like stiffness, inertia and manipulability lie on the positive semidefinite manifold, orthogonal group and unit sphere, respectively. Hence, there is a need to develop zeroth-order optimization methods over the manifolds to optimize the above mentioned control parameters~\cite{jaquier2020bayesian}.

\subsubsection{Zeroth-order Attacks on Deep Neural Networks (DNNs)}\label{sec:zoattackapp}
Our second motivating application is based on developing black-box attacks to DNNs. Despite the recent success of DNNs, studies have shown that they are vulnerable to adversarial attacks: even a well-trained DNN could completely misclassify a slightly perturbed version of the original image (which is undetectable to the human eyes); see, e.g., \cite{szegedy2013intriguing,goodfellow2014explaining}. As a result, it is extremely important on the one hand to come up with methods to train DNNs that are robust to adversarial attacks, and on the other hand to develop efficient attacks on DNNs with the goal being to make them misclassify. In practice, as the architecture of the DNN is not known to the attacker, several works, for example,~\cite{chen2017zoo,tu2019autozoom,cheng2018query}, use zeroth-order optimization algorithms for designing adversarial attacks. However, existing works have an inherent drawback-- the perturbed training example designed to fool the DNN is not in the same domain as the original training data. For example, despite the fact that natural images typically lie on a manifold \cite{Weinberger-image-manifold}, the perturbations are not constrained to lie on the same manifold. This naturally motivates us to use zeroth-order Riemannian optimization methods to design adversarial examples to fool DNNs, which at the same time preserves the manifold structures in the training data.

\subsubsection{Black-box Methods for Topological Dimension  Reduction}

The third motivating example is from the field of dimension reduction, a popular class of techniques for reducing the dimension of high-dimensional unstructured data for feature extraction and visualization. There exists a variety of methods for this task; we refer the interested reader to~\cite{lee2007nonlinear, burges2010dimension} for more details. However, a majority of the existing techniques are based on \emph{geometric} motivations. Recently, there has been a growing literature on using \emph{topological} information for performing data analysis~\cite{chazal2017introduction, mcinnes2018umap, rabadan2019topological}. One such method is a dimension reduction technique called Persistent Homology-Based Projection Pursuit~\cite{kachan2020persistent}. Roughly speaking, given a point-cloud data set with cardinality $n$ and dimension $m$ (i.e., a matrix $X \in \mathbb{R}^{m \times n}$), persistence homology refers to developing a multi-scale characterization of topologically invariant features available in the data. Such information is summarized in terms of the so-called persistence diagram, $D(X)$, which is a multiset of points in a two-dimensional plane. The idea in~\cite{kachan2020persistent} is to obtain a transformation $P^\top \in \mathbb{R}^{p \times m}$, with $p \ll m$, such that the topological summaries of the original dataset $X$ and the reduced dimensional dataset $P^\top X$ are close to each other; that is, the persistence diagram $D(X)$ and $D(P^\top X)$ are close in the 2-Wasserstein distance. The problem is then formulated as (informally speaking),
\begin{align*}
\min_{\{P \in \mathbb{R}^{m \times p}: P^\top P = I  \}} W_2(D(X), D(P^\top X)),
\end{align*} 
which is an optimization problem over the Stiefel manifold (see Section~\ref{sec:basic} for more details). It turns out that calculating the gradient of the above objective function is highly non-trivial and computationally expensive~\cite{leygonie2019framework}. However, evaluating the objecting function for various value of the matrix $P$ is relatively less expensive. Hence, this serves as yet another problem in which the methodology developed in this work could be applied naturally.
   
\subsection{Main Contributions} 
We now summarize our main contributions. 

\begin{enumerate}[leftmargin=15pt,noitemsep]
    \item In Section~\ref{sec:prelim}, we propose the (stochastic) zeroth-order Riemannian gradient (\ref{Gauss_oracle}) and Hessian (\ref{Gauss_oracle_second_order}) estimators, which addresses the infeasibility issue of the sampling for the case of derivative-free optimization over manifolds.
    \item In Section~\ref{sec:algorithms}, we demonstrate the applicability of the developed estimators for stochastic zeroth-order Riemannian optimization, as listed below. A summary of these results is given in Table \ref{table0}. To the best of our knowledge, our results are the first complexity results for stochastic zeroth-order Riemannian optimization.
    \begin{itemize}
    \item When $h(x)\equiv 0$ and the exact function evaluations of $f(x)$ are obtainable, we propose a zeroth-order Riemannian gradient descent method (\texttt{ZO-RGD}) and provide its oracle complexity for obtaining an $\epsilon$-stationary point of \eqref{problem} (see Theorem~\ref{thm1}).
    \item When $h(x)\equiv 0$ and $f(x) = \E_\xi [F(x,\xi)]$, we propose a zeroth-order Riemannian stochastic gradient descent method (\texttt{ZO-RSGD}). We analyze its oracle complexity under two different settings (see Theorems \ref{thm2} and \ref{thm:geo_cvx}). 
    \item When $h(x)$ is convex and nonsmooth, we propose a zeroth-order stochastic Riemannian proximal gradient method (\texttt{ZO-SManPG}) and provide its oracle complexity for obtaining an $\epsilon$-stationary point of \eqref{problem} (see  Theorem~\ref{thm3}).
    \item When $h(x)\equiv 0$ and $f(x) = \E_\xi [F(x,\xi)]$, where $F(x, \xi)$ satisfies certain Lipschitz Riemannian Hessian property, we propose a zeroth-order Riemannian stochastic cubic regularized Newton's method (\texttt{ZO-RSCRN}) that provably converges to an $\epsilon$-approximate local minimizer (see Theorem~\ref{thm:scrn}).
    \end{itemize}
    \item In Section~\ref{experiments}, we provide experimental results on simulated data to quantify the performance of our methods. We then demonstrate the applicability of our methods to the problem of black-box attacks to deep neural networks and robotics.
\end{enumerate}

\begin{table}[t]
\begin{center}
\begin{small}
\begin{sc}
\begin{tabular}{|c|c|c|c|}
\hline
Algorithm & Structure & Iteration Complexity & Oracle Complexity \\
\hline
\texttt{ZO-RGD} & smooth & $\mathcal{O}\left({d}/{\epsilon^{2}}\right)$ & $\mathcal{O}\left({d}/{\epsilon^{2}}\right)$ \\ \hline
 \texttt{ZO-RSGD} & \thead{smooth,\\ stochastic} & $\mathcal{O}\left(1/{\epsilon^{2}}\right)$ & $\mathcal{O}({d}/{\epsilon^{4}})$ \\ \hline
 \texttt{ZO-RSGD} & \thead{smooth, stochastic,\\ Geo-convex} & $\mathcal{O}\left( {1}/{\epsilon} \right)$ & $\mathcal{O}\left( {d}/{\epsilon^{2}} \right)$ \\ \hline
 \texttt{ZO-SManPG} & \thead{nonsmooth \\ Stochastic} & $\mathcal{O}\left(1/{\epsilon^{2}}\right)$ & $\mathcal{O}\left({d}/{\epsilon^{4}}\right)$\\ \hline
 \texttt{ZO-RSCRN} & \thead{Lipschitz Hessian\\ Stochastic} & $\mathcal{O}(1/{\epsilon^{1.5}})$ & $\mathcal{O}\left( {d}/{\epsilon^{3.5}}+{d^4}/{\epsilon^{2.5}} \right)$\\
\hline
\end{tabular}
\end{sc}
\end{small}
\end{center}
\caption{Summary of the convergence results proved in this paper. For all but the  \texttt{ZO-RSCRN} algorithm, the reported complexities correspond to $\epsilon$-stationary solution; for the  \texttt{ZO-RSCRN} algorithm the complexities correspond to $\epsilon$-local minimizers. Here, $d$ is the intrinsic dimension of the manifold $\mathcal{M}$. Furthermore, Iteration complexity refers to the number of iterations and oracle complexity refers to the number of calls to the (stochastic) zeroth-order oracle.}
\label{table0}
\end{table}

\section{Preliminaries and Methodology}
\label{sec:prelim}

We start this section with a brief review of basics of Riemannian optimization. We then introduce our stochastic zeroth-order Riemannian gradient and Hessian estimators, and provide bias and moment bounds quantifying the accuracy of the proposed estimators, which will be useful for the convergence analysis later.

\subsection{Basics of Riemannian Optimization}\label{sec:basic}

Let $\M\subset\R^n$ be a differentiable embedded submanifold. We have the following definition for the tangent space.
\begin{definition}[Tangent space]
    Consider a manifold $\M$ embedded in a Euclidean space. For any $x\in\M$, the tangent space $T_x\M$ at $x$ is a linear subspace that consists of the derivatives of all differentiable curves on $\M$ passing through $x$:
    \begin{equation}\label{eq_tangent_space}
        T_x\M=\{\gamma^{\prime}(0): \gamma(0)=x, \gamma([-\delta, \delta]) \subset \mathcal{M}\text { for some } \delta>0, \gamma \text { is differentiable}\}.
    \end{equation}
\end{definition}

The manifold $\M$ is a Riemannian manifold if it is equipped with an inner product on the tangent space, $\langle \cdot, \cdot \rangle _x : T_x\M \times T_x\M \rightarrow \R$, that varies smoothly on $\M$. We also introduce the concept of the dimension of a manifold.
\begin{definition}[Dimension of a manifold~\cite{absil2009optimization}] The dimension of the manifold $\mathcal{M}$, denoted as $d$, is the dimension of the Euclidean space that the manifold is locally homeomorphic to. 
In particular, the dimension of the tangent space is always equal to the dimension of the manifold.
\end{definition}
As an example, consider the Stiefel manifold $\M= \St(n, p):=\{X\in\R^{n\times p}: X^\top X=I_p\}$. The tangent space of $\St(n, p)$ is given by $T_X\M=\{Y\in\R^{n\times p}: X^\top Y+Y^\top X=0\}$. Hence, the dimension of the Stiefel manifold is $n p-\frac{1}{2}p(p+1)$. Note that the dimension of the manifold could be significantly less than the dimension of the ambient Euclidean space. Yet another example is the manifold of low-rank matrices~\cite{vandereycken2013low}. We now introduce the concept of a Riemannian gradient.
\begin{definition}[Riemannian Gradient]\label{def_riemann_grad}
    Suppose $f$ is a smooth function on $\M$. The Riemannian gradient $\grad f(x)$ is a vector in $T_x\M$ satisfying $\left.\frac{d(f(\gamma(t)))}{d t}\right|_{t=0}=\langle v, \grad f(x)\rangle_{x}$ for any $v\in T_x\M$, where $\gamma(t)$ is a curve as described in (\ref{eq_tangent_space}).
\end{definition}

Recall that in the Euclidean setting, a function $f:\mathbb{R}^n \to\mathbb{R}$ is $L$-smooth, if it satisfies
$|f(y)-f(x)-\langle \nabla f(x),y-x \rangle|\leq \frac{L}{2} \|x-y\|^2$, for all $x,y\in\mathbb{R}^n$. We now present the Riemannian counterpart of $L$-smooth functions. To do so, first we need the definition of retraction for a given $x\in\M$.
\begin{definition}[Retraction]\label{def_retraction}
    A retraction mapping $R_x$ is a smooth mapping from $T_x\M$ to $\M$ such that: $R_x(0)=x$, where $0$ is the zero element of $T_x\M$, and the differential of $R_x$ at $0$ is an identity mapping, i.e., $\left.\frac{d R_x(t\eta)}{d t}\right|_{t=0}=\eta$, $\forall \eta\in T_x\M$. In particular, the exponential mapping $\Exp_x$ is a retraction that generates geodesics. 
\end{definition}
\begin{assumption}[$L$-retraction-smoothness]\label{L-manifold-smooth}
There exists $L_g\geq0$ such that the following inequality holds for function $f$ in \eqref{problem}:
\begin{equation}\label{eq:lgsmoothness}
    |f(R_x(\eta))-f(x)-\langle \grad f(x),\eta \rangle_x|\leq \frac{L_g}{2}\|\eta\|^2, 	\forall x\in\M, \eta\in T_x\M.
\end{equation}
\end{assumption}

\cref{L-manifold-smooth} is also known as the restricted Lipschitz-type gradient for pullback function $\hat{f}_x(\eta):=f(R_x(\eta))$ \cite{boumal2018global}. The condition required in \cite{boumal2018global} is weaker because it only requires \cref{eq:lgsmoothness} to hold for $\|\eta\|_x\leq \rho_x$, where constant $\rho_x>0$. In our convergence analysis, we need this assumption to be held for all $\eta\in T_x\M$, i.e., $\rho_x = \infty$. This assumption is satisfied when the manifold $\M$ is a compact submanifold of $\R^n$, the retraction $R_x$ is globally defined\footnote{If the manifold is compact, then the exponential mapping $\Exp_x$ is already globally defined. This is known as the Hopf-Rinow theorem \cite{carmo1992riemannian}.} and function $f$ is $L$-smooth in the Euclidean sense; we refer the reader to \cite{boumal2018global} for more details. 
We also emphasize that \cref{L-manifold-smooth} is weaker than the geodesic smoothness assumption defined in \cite{zhang2016first}. The geodesic smoothness states that, $\forall \eta\in\M$, $f(\Exp_x(\eta))\leq f(x)+\langle g_x, \eta \rangle_x + {L_g d^2(x,\Exp_x(\eta))}/{2}$, where $g_x$ is a subgradient of $f$, $d(\cdot,\cdot)$ represents the geodesic distance. Such a condition is stronger than our \cref{L-manifold-smooth}, in the sense that, if the retraction is the exponential mapping, then geodesic smoothness implies the $L$-retraction-smoothness with the same parameter $L_g$ \cite{bento2017iteration}. 

Throughout this paper, we consider the Riemannian metric on $\M$ that is induced from the Euclidean inner product; i.e. $\langle \cdot, \cdot \rangle _x=\langle \cdot, \cdot \rangle,\ \forall x\in\M$\footnote{If the manifold is not an embedded submanifold of some Euclidean space, then we cannot have an induced Riemannian metric. In this case, the convergence result is not affected, though it would cause implementation difficulties.}. Using this Riemannian metric, the Riemannian gradient of a function is simply the projection of its Euclidean gradient onto the tangent space:
\begin{equation}\label{eq:rgrad}
    \grad f(x)=\proj_{T_x\M}\left(\nabla f(x)\right).
\end{equation}
We also present the definition of Riemannian Hessian for embedded submanifolds, which will be used in Section \ref{sec_cubic} about cubic regularized Newton's method.
\begin{definition}[Riemannian Hessian \cite{zhang2018cubic}]
    Suppose $\M$ is an embedded submanifold of $\R^n$. The Riemannian Hessian is defined as
    \begin{equation}\label{Rhessian_embedded}
        \hess f(x)[\eta] = \proj_{T_{x}\M} (D\grad f(x)[\eta]), \forall x\in\M, \eta\in T_{x}\M, 
    \end{equation}
    where $D\grad f(x)[\eta]$ is the common differential, i.e., $D\grad f(x)[\eta] = (J\grad f(x))[\eta]$, where $J$ is the Jacobian of the gradient mapping.
\end{definition}

\subsection{The Zeroth-order Riemannian Gradient Estimator}\label{section_estimator}

Recall that in the Euclidean setting, Nesterov and Spokoiny \cite{nesterov2017random} analyzed the Gaussian smoothing based zeroth-order gradient estimator. However, as that estimator requires function evaluations outside of the manifold to be well-defined, it is not directly applicable for the Riemannian setting. To address this issue, we introduce our stochastic zeroth-order Riemannian gradient estimator below.
\begin{definition}[Zeroth-Order Riemannian Gradient]\label{def_zero_rgrad}
    Generate $u=Pu_0\in T_x\M$, where $u_0\sim\N(0,I_n)$ in $\R^n$, and $P\in\R^{n\times n}$ is the orthogonal projection matrix onto $T_x\M$. Therefore $u$ follows the standard normal distribution $\N(0,P P^\top)$ on the tangent space in the sense that, all the eigenvalues of the covariance matrix $P P^\top$ are either $0$ (eigenvectors orthogonal to the tangent plane) or $1$ (eigenvectors embedded in the tangent plane). The zeroth-order Riemannian gradient estimator is defined as
    \begin{equation}\label{Gauss_oracle}
        g_\mu(x)=\frac{f(R_x(\mu u))-f(x)}{\mu} u=\frac{f(R_x(\mu P u_0))-f(x)}{\mu}P u_0.
    \end{equation}
\end{definition}


Note that the projection $P$ is easy to compute for commonly used manifolds. For example, for the Stiefel manifold $\M$, the projection is given by $\proj_{T_X\M}(Y) = (I-X X^\top)Y + X\operatorname{skew}(X^\top Y)$, where $\operatorname{skew}(A):=(A-A^\top)/2$ (see \cite{absil2009optimization}). 
\begin{remark}
{In this work, we assume that the function $f$ is defined on submanifolds embedded in Euclidean space, so that it is efficient to sample from the associated tangent space, as discussed above; see also~\cite{diaconis2013sampling}. We remark that the above gradient estimation methodology is more generally applicable to other manifolds. However, the generality comes at the cost of practical applicability as it is not an easy task to efficiently sample Gaussian random objects on the tangent space of general manifolds; see~\cite{hsu2002stochastic} for more details.}
\end{remark}

We now discuss some differences between the zeroth-order gradient estimators in the Euclidean setting \cite{nesterov2017random} and the Riemannian setting \eqref{Gauss_oracle}. In the Euclidean case, the zeroth-order gradient estimator can be viewed as estimating the gradient of the Gaussian smoothed function, $f_\mu(x)=\frac{1}{\kappa}\smallint_{\R^n}f(x+\mu u)e^{-\frac{1}{2}\|u\|^2}d u$,
because $ \nabla f_\mu(x)=\E_u(g_\mu(x))= \frac{1}{\kappa}\smallint_{\R^n} \frac{f(x+\mu u)-f(x)}{\mu} u e^{-\frac{1}{2}\|u\|^2}d u$, where $\kappa$ is the normalization constant for Gaussian. This was also observed as an instantiation of Gaussian Stein's identity \cite{balasubramanian2019zeroth}. However, this observation is no longer true in the Riemannian setting, as we incorporate the retraction operator when evaluating $g_\mu$, and this forces us to seek for a direct evaluation of $\E_u(g_\mu(x))$, instead of utilizing properties of the smoothed function $f_\mu$.
We also remark that, $g_\mu(x)$ is a biased estimator of $\grad f(x)$. The difference between them can be bounded as in \cref{lemma123}. Some intermediate results for this purpose are as follows.  
\begin{lemma}\label{lemma0}
Suppose $\X$ is a $d$-dimensional subspace of $\R^n$, with orthogonal projection matrix $P\in\R^{n\times n}$. $u_0$ follows a standard norm distribution $\N(0,I_n)$, and $u=P u_0$ is the orthogonal projection of $u_0$ onto the subspace $\X$. Then $\forall x\in\X$, we have
\begin{equation}\label{lemma0_eq1}
    x = \frac{1}{\kappa} \int_{\R^n} \langle x, u \rangle u e^{-\frac{1}{2}\|u_0\|^2}d u_0, \qquad\text{and}\qquad      \|x\|^2 = \frac{1}{\kappa} \int_{\R^n} \langle x, u \rangle^2 e^{-\frac{1}{2}\|u_0\|^2}d u_0,
\end{equation}
where $\kappa$ is the constant for normal density function: $\kappa := \smallint_{\R^n}e^{-\frac{1}{2}\|u\|^2}d u=(2\pi)^{n/2}$.
\end{lemma}
\begin{proof}{Proof of \cref{lemma0}}
By the definition of covariance matrix, we have $\frac{1}{\kappa} \int_{\R^n} u_0u_0^\top e^{-\frac{1}{2}\|u_0\|^2}d u_0 = I_n$. Since $\langle x, u \rangle=\langle x, u_0 \rangle$, $\forall x\in\X$, we have
\begin{equation}\label{temp_1}
    \frac{1}{\kappa} \int_{\R^n} \langle x, u \rangle u_0 e^{-\frac{1}{2}\|u_0\|^2}d u_0 = x,
\end{equation}
which implies $\frac{1}{\kappa} \smallint_{\R^n} \langle x, u \rangle u e^{-\frac{1}{2}\|u_0\|^2}d u_0 = P x=x$. Similarly, taking inner product with $x$ on both sides of \cref{temp_1}, we have $\|x\|^2 = \frac{1}{\kappa} \int_{\R^n} \langle x, u \rangle^2 e^{-\frac{1}{2}\|u_0\|^2}d u_0$.
\end{proof}
The following bound for the moments of normal distribution is restated without proof.
\begin{lemma}\label{lem:gaussmoment}\cite{nesterov2017random}
Suppose $u\sim\N(0,I_n)$ is a standard normal distribution. Then for all integers $ p\geq 2$, we have $M_p:= \E_u(\|u\|^p)\leq (n+p)^{p/2}$.
\end{lemma}

\begin{corollary}\label{corr:rmoment}
    For $u_0\sim\N(0,I_n)$ and $u=P u_0$, where $P\in\R^{n\times n}$ is the orthogonal projection matrix onto a $d$ dimensional subspace $\X$ of $\R^n$, we have $\E_{u_0}(\|u\|^p)\leq (d+p)^{p/2}$.
\end{corollary}
\begin{proof}{Proof of \cref{corr:rmoment}}
    Assume the eigen-decomposition of $P$ is $P=Q^\top\Lambda Q$, where $Q$ is an unitary matrix and $\Lambda$ is a diagonal matrix with the leading $d$ diagonal entries being 1 and other diagonal entries being 0. 
Denote $\tilde{u}=Q u_0\sim \N(0,I_n)$, then $\Lambda \tilde{u}=(\tilde{u}_1,...,\tilde{u}_d,0,...,0)$. Since $u=Q^\top \Lambda \tilde{u}$ has the same distribution as $\Lambda \tilde{u}$, we have $\E\|u\|^p=\E\|(\tilde{u}_1,...,\tilde{u}_d,0,...,0)\|^p\leq (d+p)^{p/2}$, by~\cref{lem:gaussmoment}.
\end{proof}

Now we provide the bounds on the error of our gradient estimator $g_\mu(x)$ \eqref{Gauss_oracle}. Recall that $d$ denotes the dimension of the manifold $\M$.
\begin{proposition}\label{lemma123}
    Under \cref{L-manifold-smooth}, we have
    \begin{enumerate}[noitemsep]
    \item[(a)] $\|\E_{u_0} (g_\mu(x))- \grad f(x)\| \leq \frac{\mu L_g}{2}(d+3)^{3/2}$,
    \item[(b)] $\|\grad f(x)\|^2\leq 2\|\E_{u_0} (g_\mu(x))\|^2+\frac{\mu^2}{2}L_g(d+6)^3$,
    \item[(c)] $\E_{u_0}(\|g_\mu(x)\|^2)\leq \frac{\mu^2}{2}L_g^2(d+6)^3+2(d+4)\|\grad f(x)\|^2$.
    \end{enumerate}
\end{proposition}

\begin{proof}{Proof of \cref{lemma123}}
For part (a), since
    \begin{equation*}
    \begin{split}
        \E (g_\mu(x))- \grad f(x) =\frac{1}{\kappa}\int_{\R^n}\left(\frac{f(R_x(\mu u))-f(x)}{\mu} -\langle \grad f(x), u \rangle \right) u e^{-\frac{1}{2}\|u_0\|^2}d u_0,
    \end{split}
    \end{equation*}
    we have
    \begin{align*}
          & \quad \|\E (g_\mu(x))- \grad f(x)\| \\
        = & \quad \|\frac{1}{\mu\kappa}\int_{\R^n}\left(f(R_x(\mu u))-f(x) -\langle \grad f(x), \mu u \rangle \right) u e^{-\frac{1}{2}\|u_0\|^2}d u_0\| \\
        \leq & \quad \frac{1}{\mu\kappa} \int_{\R^n} \frac{L_g}{2}\|\mu u\|^2 \|u\| e^{-\frac{1}{2}\|u_0\|^2}d u_0 =  \frac{\mu L_g}{2\kappa} \int_{\R^n} \|u\|^3e^{-\frac{1}{2}\|u_0\|^2}d u_0 \leq \frac{\mu L_g}{2}(d+3)^{3/2},
    \end{align*}
    where the first inequality is by due to \eqref{eq:lgsmoothness}, and the last inequality is from \cref{corr:rmoment}. This completes the proof of part (a).

To prove part (b), note that 
\begin{align*}
        & \|\grad f(x)\|^2 = \left\|\frac{1}{\kappa}\int_{\R^n}\langle\grad f(x),u\rangle u e^{-\frac{1}{2}\|u_0\|^2}d u_0\right\|^2 \\
        = & \left\|\frac{1}{\mu\kappa}\int_{\R^n}( [f(R_x(\mu u))-f(x)] -[f(R_x(\mu u))-f(x)-\langle\grad f(x),\mu u\rangle] ) u e^{-\frac{1}{2}\|u_0\|^2}d u_0\right\|^2 \\
        \leq & 2\|\E (g_\mu(x))\|^2+\left\|\frac{2}{\mu^2}\int_{\R^n}(f(R_x(\mu u))-f(x)-\langle\grad f(x),\mu u\rangle ) u e^{-\frac{1}{2}\|u_0\|^2}d u_0\right\|^2 \\
        \leq & 2\|\E (g_\mu(x))\|^2+\frac{2}{\mu^2}\int_{\R^n}(f(R_x(\mu u))-f(x)-\langle\grad f(x),\mu u\rangle )^2\|u\|^2 e^{-\frac{1}{2}\|u_0\|^2}d u_0 \\
        \leq & 2\|\E (g_\mu(x))\|^2+\frac{\mu^2}{2}L_g (d+6)^3,
    \end{align*}
    where the last inequality is from the same trick as in part (a). This completes the proof of part (b).

Finally, we prove part (c). Since $\E(\|g_\mu(x)\|^2)=\frac{1}{\mu^2}\E_{u_0}\left[ (f(R_x(\mu u))-f(x))^2\|u\|^2 \right]$, and $(f(R_x(\mu u))-f(x))^2=(f(R_x(\mu u))-f(x)-\mu\langle\grad f(x),u\rangle+\mu\langle\grad f(x),u\rangle)^2 \leq 2(\frac{L_g}{2}\mu^2\|u\|^2)^2+2\mu^2 \langle\grad f(x),u\rangle^2$, we have
    \begin{equation}\label{trick_c_1}
        \E(\|g_\mu(x)\|^2)\leq \frac{\mu^2}{2}L_g^2\E(\|u\|^6)+2\E(\| \langle\grad f(x),u\rangle u\|^2) 
        \leq \frac{\mu^2}{2}L_g^2(d+6)^3+2\E(\| \langle\grad f(x),u\rangle u\|^2).
    \end{equation}
    Now we bound the term $\E(\| \langle\grad f(x),u\rangle u\|^2)$ using the same trick as in \cite{nesterov2017random}. Without loss of generality, assume $\X$ is the $d$-dimensional subspace generated by the first $d$ coordinates, i.e., $\forall x\in\X$, the last $n-d$ elements of $x$ are zeros. Also for brevity, denote $g=\grad f(x)$. We have that
    \begin{align*}
        \E(\| \langle\grad f(x),u\rangle u\|^2)&= \frac{1}{\kappa}\int_{\R^n}\langle\grad f(x),u\rangle^2 \|u\|^2 e^{-\frac{1}{2}\|u_0\|^2}d u_0 \\
        &=\frac{1}{\kappa(d)}\int_{\R^d}\left(\sum_{i=1}^{d}g_i x_i\right)^2 \left(\sum_{i=1}^{d}x_i^2\right) e^{-\frac{1}{2}\sum_{i=1}^{d}x_i^2}d x_1\cdots d x_d,
    \end{align*}
    where $x_i$ denotes the $i$-th coordinate of $u_0$, the last $n-d$ dimensions are integrated to be one, and $\kappa(d)$ is the normalization constant for $d$-dimensional Gaussian distribution. For simplicity, denote $x=(x_1,...,x_d)$, then
    \begin{equation}\label{trick_c_2}
        \begin{split}
            & \E(\| \langle\grad f(x),u\rangle u\|^2) = \frac{1}{\kappa(d)}\int_{\R^d}\langle g, x\rangle^2 \|x\|^2 e^{-\frac{1}{2}\|x\|^2}d x \\
        \leq & \frac{1}{\kappa(d)}\int_{\R^d}\|x\|^2 e^{-\frac{\tau}{2}\|x\|^2}\langle g, x\rangle^2 e^{-\frac{1-\tau}{2}\|x\|^2}d x \leq \frac{2}{\kappa(d)\tau e}\int_{\R^d}\langle g, x\rangle^2 e^{-\frac{1-\tau}{2}\|x\|^2}d x \\
        = &\frac{2}{\kappa(d)\tau(1-\tau)^{1+d/2} e}\int_{\R^d} \langle g, x\rangle^2 e^{-\frac{1}{2}\|x\|^2}d x =\frac{2}{\tau(1-\tau)^{1+d/2} e}\|g\|^2,
        \end{split}
    \end{equation}
    where the second inequality is due to the following fact: $x^p e^{-\frac{\tau}{2}x^2}\leq (\frac{2}{\tau e})^{p/2}$. Taking $\tau = \frac{2}{(d+4)}$ gives the desired result.
    \end{proof}


\subsection{The Zeroth-order Riemannian Hessian Estimator}\label{sec:hessianestimator}
We now extend the above methodology and propose estimators for the Riemannian Hessian in the stochastic zeroth-order setting. We restrict our discussion to compact submanifolds embedded in Euclidean space, so that the definition of Riemannian Hessian \eqref{Rhessian_embedded} is applied. We assume the following assumption of $F(x,\xi)$: 
\begin{assumption}\label{assumption_hessian}
    Given any point $x\in\M$ and $\eta\in T_x\M$, we have
    \begin{equation}\label{assumption_hessian_eq}
        \|P_{\eta}^{-1}\circ \hess F(R_x(\eta),\xi)\circ P_{\eta} - \hess F(x,\xi)\|_{\operatorname{op}}\leq L_H \|\eta\|,
    \end{equation}
    almost everywhere for $\xi$, where $P_{\eta}:T_x\M\rightarrow T_{R_x(\eta)}\M$ denotes the parallel transport \cite{agarwal2018adaptive}, an isometry from the tangent space of $x$ to the tangent space of $R_x(\eta)$, and $\circ$ is the function composition. Here $\|\cdot\|_{\operatorname{op}}$ is the operator norm in the ambient Euclidean space.
    \end{assumption}

\cref{assumption_hessian} is the analogue of the Lipschitz Hessian type assumption from the Euclidean setting, and induces the following equivalent conditions (see, also \cite{agarwal2018adaptive}):
    \begin{equation}\label{coro_hessian_assumption}
    \begin{split}
        \|P_{\eta}^{-1}\grad F(R_x(\eta),\xi) - \grad f(x)-\hess F(x,\xi)[\eta]\| & \leq \frac{L_H}{2}\|\eta\|^2 \\
        \left|F(R_{x}(\eta),\xi) - \left[ F(x,\xi) + \langle \eta, \grad F(x,\xi)\rangle+\frac{1}{2}\langle \eta,\hess F(x,\xi)[\eta] \rangle \right]\right| &\leq \frac{L_H}{6}\|\eta\|^3.
    \end{split}
    \end{equation}
In the Euclidean setting, $P_{\eta}$ reduces to the identity mapping. Throughout this section, we also assume that $F(\cdot,\xi)$ satisfies \cref{L-manifold-smooth} and the following assumption, which is used frequently in zeroth-order stochastic optimization \cite{ghadimi2013stochastic, balasubramanian2019zeroth,zhou2019faster}.
\begin{assumption}\label{sto_assumption}
We have (with $\E=\E_{\xi}$) that, $\E [F(x,\xi)]=f(x)$, $\E [\grad F(x,\xi)]=\grad f(x)$ and $\E\left[ \|\grad F(x,\xi) - \grad f(x)\|^2 \right]\leq \sigma^2$, $\forall x\in\mathcal{M}$.
\end{assumption}
We first introduce the following identity which follows immediately from the second-order Stein's identity for Gaussian distribution~\cite{stein1972bound}.

%

\begin{lemma}\label{lemma_second_order_stein}
Suppose $\X$ is a $d$-dimensional subspace of $\R^n$, with orthogonal projection matrix $P\in\R^{n\times n}$, $P=P^2=P^\top $, and  $u_0\sim\N(0,I_n)$ is a standard normal distribution and $u=P u_0$ is the orthogonal projection of $u_0$ onto the subspace. Then $\forall H\in\R^{n\times n}$, $H^\top =H$, and $H=P H P$ (which means that the eigenvectors of $H$ lies all in $\X$), we have
\begin{equation}\label{lemma_second_eq1}
    P H P = \frac{1}{2\kappa} \int_{\R^n} \langle u, H u \rangle (u u^\top -P) e^{-\frac{1}{2}\|u_0\|^2}d u_0 = \E \left[ \frac{1}{2}\langle u, H u \rangle (u u^\top -P) \right],
\end{equation}
where $\|\cdot\|$ here is the Euclidean norm on $\R^n$, and $\kappa$ is the constant for normal density function given by $\kappa := \smallint_{\R^n}e^{-\frac{1}{2}\|u\|^2}d u=(2\pi)^{n/2}$.
\end{lemma}

The identity in~\eqref{lemma_second_eq1} simply follows by applying the second-order Stein's identity, $\E[(x x^\top - I_n)g(x)]=\E[\nabla^2 g(x)]$, directly to the function $g(x)=\frac{1}{2}\langle x, H x \rangle$ and multiplying the resulting identity by $P$ on both sides.

\begin{lemma}\cite{balasubramanian2019zeroth}\label{lemma_moment_second_order}
Suppose $\X$ is a $d$-dimensional subspace of $\R^n$, with orthogonal projection matrix $P\in\R^{n\times n}$, $P=P^2=P^\top $, and $u_0\sim\N(0,I_n)$ is a standard normal distribution and $u=P u_0$ is the orthogonal projection of $u_0$ onto the subspace. Then
\begin{equation}\label{lemma_moment_second_order_eq}
    \E [\|u_0u_0^\top -I_n\|_F^8]\leq 2(n+16)^8~~~\text{and}~~~    \E [\|u u^\top -P\|_F^8]\leq 2(d+16)^8.
\end{equation}
\end{lemma}

\begin{proof}{Proof of \cref{lemma_moment_second_order}}
    See \cite{balasubramanian2019zeroth} for the proof of the first inequality in \cref{lemma_moment_second_order_eq}. We now show how to get the right part from the left. Similar to the proof of \cref{corr:rmoment}, we use an eigen-decomposition of $P=Q^T\Lambda Q$ and get (again $\tilde{u}=Q u$):
    $$
        \E \|u u^\top -P\|_F^8 = \E \|(\tilde{u}_1,...,\tilde{u}_d)^\top (\tilde{u}_1,...,\tilde{u}_d) - I_d\|_F^8\leq 2(d+16)^8,
    $$
    which completes the proof.
\end{proof}

We now propose our zeroth-order Riemannian Hessian estimator, motivated by the zeroth-order Hessian estimator in the Euclidean setting proposed by~\cite{balasubramanian2019zeroth}.
\begin{definition}[Zeroth-Order Riemannian Hessian]
    Generate $u\in T_x\M$ following a standard normal distribution on the tangent space $T_x\M$, by projection $u=P_x u_0$ as described in Section \ref{section_estimator}. Then, the zeroth-order Riemannian Hessian estimator of a function $f$ at the point $x$ is given by
    \begin{equation}\label{Gauss_oracle_second_order}
        H_\mu(x)=\frac{1}{2\mu^2}(u u^\top -P)[F(R_x(\mu u),\xi)+F(R_x(-\mu u),\xi)-2F(x,\xi)].
    \end{equation}
    Note that our Riemannian Hessian estimator is actually the Hessian estimator of the pullback function $\hat{F}_{x}(\eta,\xi)=F(R_{x}(\eta),\xi)$, $\forall x\in\M$ and $\eta\in T_x\M$ projected onto the tangent space $T_{x}\M$. 
\end{definition}
We immediately have the following bound on the variance of $H_\mu(x)$.
\begin{lemma}\label{lemma_bound_var_H}
Under \cref{L-manifold-smooth}, the Riemannian Hessian estimator given in \cref{Gauss_oracle_second_order} satisfies
    \begin{equation}\label{hessian_norm_bound}
        \E_{\mathcal{U},\Xi} \|H_{\mu}(x)\|_F^4 \leq \frac{(d+16)^8}{8}L_g^2.
    \end{equation}
\end{lemma}

\begin{proof}{Proof of \cref{lemma_bound_var_H}}
From \cref{L-manifold-smooth} and \cref{corr:rmoment} we have
    \begin{equation}\label{lem6.5-eq1}
    \begin{split}
        &\E|F(R_x(\mu u),\xi)+F(R_x(-\mu u),\xi)-2F(x,\xi)|^8 \\
        =& \E|F(R_x(\mu u),\xi)-F(x,\xi)-\langle\grad F(x,\xi),\mu u\rangle +F(R_x(-\mu u),\xi)-F(x,\xi) -\langle\grad F(x,\xi),-\mu u\rangle|^8 \\
        \leq& \E [\frac{\mu^2L_g}{2}\|u\|^2+\frac{\mu^2L_g}{2}\|u\|^2]^8 = \E [\mu^{16}L_g^8 \|u\|^{16}]\leq \mu^{16} L_g^8(d+16)^8.
    \end{split}
    \end{equation}
Moreover, we have 
    \begin{equation}\label{lem6.5-eq2}
    \begin{split}
        \E \|H_{\mu}(x)\|_F^4 = & \E \left\|\frac{1}{2\mu^2}(u u^\top -P)[F(R_x(\mu u),\xi)+F(R_x(-\mu u),\xi)-2F(x,\xi)]\right\|_F^4 \\
        \leq & \frac{1}{16\mu^8}\left( \E|F(R_x(\mu u),\xi)+F(R_x(-\mu u),\xi)-2F(x,\xi)|^8 \E \|u u^\top  - P\|^8 \right)^{1/2} \\
        \leq & \frac{(d+16)^4}{8\mu^8} \left( \E|F(R_x(\mu u),\xi)+F(R_x(-\mu u),\xi)-2F(x,\xi)|^8 \right)^{1/2},
    \end{split}
    \end{equation}
    where the first inequality is by H\"{o}lder's inequality and the second one is by \cref{lemma_moment_second_order}.  Combining \eqref{lem6.5-eq1} and \eqref{lem6.5-eq2} yields the desired result \eqref{hessian_norm_bound}.
\end{proof}

We will also use the mini-batch multi-sampling technique. For $i=1,...,b$, denote each Hessian estimator as
\begin{equation}
    H_{\mu,i}(x)=\frac{1}{2\mu^2}(u_i u_i^\top -P)[F(R_x(\mu u_i),\xi_i)+F(R_x(-\mu u_i),\xi_i)-2F(x,\xi_i)].
\end{equation}
The averaged Hessian estimator is given by
\begin{equation}\label{multi_sample_hessian}
    \bar{H}_{\mu,\xi}(x)=\frac{1}{b}\sum_{i=1}^{b}H_{\mu,i}(x).
\end{equation}
We now have the following bound of $\bar{H}_{\mu,\xi}(x)$ and $\hess f(x)$.

\begin{lemma}\label{lemma_var_hessian_approx}
Under \cref{L-manifold-smooth} and \cref{assumption_hessian}, let $\bar{H}_{\mu,\xi}(x)$ be calculated as in \cref{multi_sample_hessian}, then we have that: $\forall x\in\M$ and $\forall \eta\in T_x\M$,
    \begin{equation}\label{zohessian_approx_square_moment}
        \E_{\mathcal{U},\Xi} \|\bar{H}_{\mu,\xi}(x) - \hess f(x)\|_{\operatorname{op}}^2 \leq  \frac{(d+16)^4}{\sqrt{2}b}L_g + \frac{\mu^2 L_H^2}{18}(d+6)^5,
    \end{equation}
    \begin{equation}\label{zohessian_approx_third_moment}
        \E_{\mathcal{U},\Xi} \|\bar{H}_{\mu,\xi}(x) - \hess f(x)\|_{\operatorname{op}}^3 \leq \tilde{C}\frac{(d+16)^6}{b^{3/2}}L_g^{1.5} + \frac{1}{27}\mu^3L_H^3 (d+6)^{7.5} ,
    \end{equation}
    where $\|\cdot\|_{op}$ denotes the operator norm and $\tilde{C}$ is some absolute constant.
\end{lemma}
\begin{proof}{Proof of \cref{lemma_var_hessian_approx}}
    Denote $\E=\E_{\mathcal{U},\Xi}$ as the expectation with respect to all previous random variables. We first show \cref{zohessian_approx_square_moment}.  Denote $X_i=H_{\mu,i}-\E H_{\mu,i}$, then $X_i$'s are iid zero-mean random matrices. Since $\|\cdot\|_{\operatorname{op}}\leq\|\cdot\|_F$, we have 
    \begin{equation}\label{lem6.6-eq0}
        \begin{split}
           & \E \|\bar{H}_{\mu,\xi}(x) - \E\bar{H}_{\mu,\xi}(x)\|_{\operatorname{op}}^2 =\E \left\|\frac{1}{b}\sum_{i=1}^{b} X_i\right\|_{\operatorname{op}}^2 \leq \E \left\|\frac{1}{b}\sum_{i=1}^{b} X_i\right\|_{F}^2 \\
            =&\E\left[\frac{1}{b^2}\sum_{i=1}^{b}\| X_i\|_{F}^2 + \frac{1}{b^2}\sum_{i\not=j} \langle X_i,X_j\rangle \right] =\E\left[\frac{1}{b^2}\sum_{i=1}^{b}\| X_i\|_{F}^2\right] \\
           = & \E \frac{1}{b^2}b\| X_1\|_{F}^2=\E \frac{1}{b}\| H_{\mu,1}-\E H_{\mu,1}\|_{F}^2 =\frac{1}{b}\E \left[ \| H_{\mu,1}\|_{F}^2 - \|\E H_{\mu,1}\|_{F}^2 \right] \\ \leq & \frac{1}{b}\E \| H_{\mu,1}\|_{F}^2  \leq \frac{1}{b}\sqrt{\E\|H_{\mu,1}(x)\|_{F}^4} \leq  \frac{(d+16)^4}{2\sqrt{2}b}L_g,
        \end{split}
    \end{equation}
    where the third inequality is from the Jensen's inequality, and the last inequality is due to \cref{hessian_norm_bound}. Note that
    \eqref{lem6.6-eq0} immediately implies 
    \begin{equation}\label{lem6.6-eq1} 
    \begin{split}
     \E \|\bar{H}_{\mu,\xi}(x) - \hess f(x)\|_{\operatorname{op}}^2 \leq  & 2\E \|\bar{H}_{\mu,\xi}(x) - \E\bar{H}_{\mu,\xi}(x)\|_{\operatorname{op}}^2 + 2 \|\E\bar{H}_{\mu,\xi}(x) - \hess f(x)\|_{\operatorname{op}}^2  \\
       \leq &  \frac{(d+16)^4}{\sqrt{2}b}L_g + 2 \|\E\bar{H}_{\mu,\xi}(x) - \hess f(x)\|_{\operatorname{op}}^2.
       \end{split}
    \end{equation}
    
    Now we bound the term $\|\E\bar{H}_{\mu,\xi}(x) - \hess f(x)\|_{\operatorname{op}}^2$. Note that
    \begin{equation*}
    \begin{split}
        &|\langle \eta, (\E H_{\mu,i}(x) - \hess f(x))[\eta]\rangle|\\
        =& \left|\langle \eta, \left(\E\left[ \frac{1}{2\mu^2}(u u^\top -P)[f(R_x(\mu u))+f(R_x(-\mu u))-2f(x)]\right] - \hess f(x) \right)[\eta] \rangle\right| \\
        =&\left|\langle \eta,\left( \E\left[ \frac{1}{2\mu^2}(u u^\top -P)[f(R_x(\mu u)) +f(R_x(-\mu u))-2f(x) -\mu^2\langle u, \hess f(x)[u] \rangle]\right] \right)[\eta] \rangle\right| \\
        =& \frac{1}{2\mu^2}\left|\langle \eta,\left( \E\left[ [f(R_x(\mu u))-f(x)- \frac{\mu^2}{2}\langle u, \hess f(x)[u] \rangle \right.\right.\right.\\ 
        & \left.\left.\left. + f(R_x(-\mu u)) - f(x) -\frac{\mu^2}{2}\langle u, \hess f(x)[u] \rangle](u u^\top -P)\right] \right)[\eta] \rangle\right|,
    \end{split}
    \end{equation*}
    which together with \cref{assumption_hessian} yields 
    \begin{equation}\label{lem6.6-eq3}
    \begin{split}
    & |\langle \eta, (\E H_{\mu,i}(x) - \hess f(x))[\eta]\rangle| \leq  \frac{\mu L_H}{6}\E \left[\|u\|^3\|u u^\top -P\|_{\operatorname{op}}\right]\|\eta\|^2 \\ (\text{H\"{o}lder})\leq &  \frac{\mu L_H}{6}\sqrt{\E\|u\|^6 \E\|u u^\top -P\|_{F}^2} \|\eta\|^2\leq \frac{\mu L_H}{6}(d+6)^{5/2}\|\eta\|^2,
    \end{split}
    \end{equation}
    where the last inequality is by \cref{corr:rmoment} and \cref{lemma_moment_second_order}. \eqref{lem6.6-eq3} implies 
    \begin{equation}\label{lem6.6-eq3.5}
    \|\E\bar{H}_{\mu,\xi}(x) - \hess f(x)\|_{\operatorname{op}}\leq \frac{\mu L_H}{6}(d+6)^{5/2}.
    \end{equation}
    Combining \eqref{lem6.6-eq1} and \eqref{lem6.6-eq3.5} gives \cref{zohessian_approx_square_moment}.
    
    Now we show \cref{zohessian_approx_third_moment}. By a similar analysis we have
    \begin{equation}\label{lem6.6-eq4}
    \begin{split}
        &\E \|\bar{H}_{\mu,\xi}(x) - \hess f(x)\|_{\operatorname{op}}^3 \\
        \leq& \E (\|\bar{H}_{\mu,\xi}(x) - \E\bar{H}_{\mu,\xi}(x)\|_{\operatorname{op}} + \|\E\bar{H}_{\mu,\xi}(x) - \hess f(x)\|_{\operatorname{op}})^3 \\
        \leq & 8\E \|\bar{H}_{\mu,\xi}(x) - \E\bar{H}_{\mu,\xi}(x)\|_{\operatorname{op}}^3 + 8  \|\E\bar{H}_{\mu,\xi}(x) - \hess f(x)\|_{\operatorname{op}}^3  \\
        (\text{H\"{o}lder})\leq & 8\sqrt{\E \|\bar{H}_{\mu,\xi}(x) - \E\bar{H}_{\mu,\xi}(x)\|_{\operatorname{op}}^2\E \|\bar{H}_{\mu,\xi}(x) - \E\bar{H}_{\mu,\xi}(x)\|_{\operatorname{op}}^4} \\ & +8  \|\E\bar{H}_{\mu,\xi}(x) - \hess f(x)\|_{\operatorname{op}}^3,
    \end{split}
    \end{equation}
    where the second inequality is by the following fact: when $a,b\geq0$, $(a+b)^3\leq\max\{(2a)^3,(2b)^3\}\leq 8a^3+8b^3$. 
    Moreover, since $\|\cdot\|_{\operatorname{op}}\leq\|\cdot\|_F$, and $X_i=H_{\mu,i}-\E H_{\mu,i}$ are iid zero-mean random matrices, we have 
    \begin{equation}\label{lem6.6-eq5}
    \begin{split}
        &\E \|\bar{H}_{\mu,\xi}(x) - \E \bar{H}_{\mu,\xi}(x)\|_{\operatorname{op}}^4 = \E \|\frac{1}{b}\sum_{i=1}^{b}X_i\|_{\operatorname{op}}^4 \leq \frac{C}{b^4}\left( \E \|\sum_{i=1}^{b}X_i\|_{\operatorname{op}} + (b \E \|X_i\|_{\operatorname{op}}^4)^{1/4} \right)^4 \\
       \leq & \frac{C}{b^4}\left( \sqrt{ \E \|\sum_{i=1}^{b} X_i\|_{F}^2} + (b \E \|X_i\|_{F}^4)^{1/4} \right)^4=\frac{C}{b^4}\left( \sqrt{ \sum_{i=1}^{b}\E \| X_i\|_{F}^2} + (b \E \|X_i\|_{F}^4)^{1/4} \right)^4\\
       =&\frac{C}{b^4}\left( \sqrt{b}\sqrt{\E \| X_1\|_{F}^2} + (b \E \|X_1\|_{F}^4)^{1/4} \right)^4 \leq\frac{C}{b^4}\left( \sqrt{b}\sqrt[4]{\E \| X_1\|_{F}^4} + (b \E \|X_1\|_{F}^4)^{1/4} \right)^4 \\
        =&\frac{C}{b^4}(\sqrt{b}+\sqrt[4]{b})^4 \E \|H_{\mu,1}-\E H_{\mu,1}\|_{F}^4 \leq \frac{16C}{b^2} \E \|H_{\mu,1}-\E H_{\mu,1}\|_{F}^4 \\
        =& \frac{16C}{b^2} \E (\|H_{\mu,1}\|_{F}^2-2\langle H_{\mu,1}, \E H_{\mu,1} \rangle+\|\E H_{\mu,1}\|_{F}^2)^2 \\
        \leq &\frac{16C}{b^2} \E (\|H_{\mu,1}\|_{F}^2+2\|H_{\mu,1}\|_{F} \|\E H_{\mu,1}\|_{F}+\|\E H_{\mu,1}\|_{F}^2)^2 \\
       \leq & \frac{16C}{b^2} \E (2\|H_{\mu,1}\|_{F}^2+2\|\E H_{\mu,1}\|_{F}^2)^2 \leq\frac{16C}{b^2} \E (2\|H_{\mu,1}\|_{F}^2+2\E\| H_{\mu,1}\|_{F}^2)^2 \\
     \leq &\frac{64C}{b^2} (\E \|H_{\mu,1}\|_{F}^4 + \E \|H_{\mu,1}\|_{F}^4) \leq \frac{128C}{b^2}(d+16)^8 L_g^2,
    \end{split}
    \end{equation}
    where the first inequality is due to the Rosenthal inequality~\cite{rio2009moment}, $C$ is an absolute constant, the fourth inequality is due to the fact $1\leq \sqrt[4]{b}\leq\sqrt{b}$. Plugging \cref{lem6.6-eq0}, \cref{lem6.6-eq3.5} and \cref{lem6.6-eq5} back to \cref{lem6.6-eq4} gives the desired result \eqref{zohessian_approx_third_moment}.
\end{proof}

\section{Stochastic Zeroth-order Riemannian Optimization Algorithms}\label{sec:algorithms}
We now demonstrate the applicability of the developed Riemannian derivative estimation methodology in Section~\ref{sec:prelim}, for various classes of stochastic zeroth-order Riemannian optimization algorithms. 

\subsection{Zeroth-order Smooth Riemannian Optimization}

In this section, we focus on the smooth optimization problem with $h\equiv 0$ and $f$ satisfying \cref{L-manifold-smooth}. We propose \texttt{ZO-RGD}, the zeroth-order Riemannian gradient descent method and provide its complexity analysis. The algorithm is formally presented in \cref{algorithm1}. 

\begin{algorithm}[ht]
   \caption{Zeroth-Order Riemannian Gradient Descent (\texttt{ZO-RGD})}
   \label{algorithm1}
\begin{algorithmic}[1]
   \STATE {\bfseries Input:} Initial point $x_0\in\M$, smoothing parameter $\mu$, step size $\eta_k$, fixed number of iteration $N$.
   \FOR{$k=0$ {\bfseries to} $N-1$}
   \STATE Sample a standard Gaussian random vector $u_k\in T_{x_k}\M$ by orthogonal projection in \cref{def_zero_rgrad}.
   \STATE Compute the zeroth-order gradient $g_\mu(x_k)$ by \cref{Gauss_oracle}.
   \STATE Update $x_{k+1}=R_{x_k}(-\eta_k g_\mu (x_k))$.
   \ENDFOR
\end{algorithmic}
\end{algorithm}

The following theorem gives the iteration and oracle complexities of \cref{algorithm1} for obtaining an $\epsilon$-stationary point of \eqref{problem} when $h\equiv 0$. 

\begin{theorem}\label{thm1}
    Let $f$ satisfy \cref{L-manifold-smooth} and suppose $\{x_k\}$ is the sequence generated by \cref{algorithm1} with the stepsize $\eta_k=\hat{\eta}=\frac{1}{2(d+4)L_g}$. Then, we have
    \begin{equation}\label{thm1-eq1}
        \begin{split}
            \hspace{-0.2in}\frac{1}{N+1}\sum_{k=0}^{N}&\E_{\mathcal{U}_k}\|\grad f(x_k)\|^2\leq  \frac{4}{\hat{\eta}}\left(\frac{f(x_0)-f(x^*)}{N+1}+C(\mu)\right),
        \end{split}
    \end{equation}
    where $\mathcal{U}_k$ denotes the set of all Gaussian random vectors we drew for the first $k$ iterations \footnote{The notation of taking the expectation w.r.t. a set, is to take the expectation for each of the elements in the set.}, and $C(\mu)=\frac{\mu^2 L_g}{16}\frac{(d+3)^3}{(d+4)}+\frac{\mu^2 }{16}\frac{(d+6)^3}{(d+4)}+\frac{\mu^2 L_g}{16}\frac{(d+6)^3}{(d+4)^2}$. In order to have
    \begin{equation}\label{thm1-eq2}
        \frac{1}{N+1}\sum_{k=0}^{N}\E_{\mathcal{U}_k}\|\grad f(x_k)\|^2\leq \epsilon^2,
    \end{equation}
    we need the smoothing parameter $\mu$ and number of iteration $N$ (which is also the number of calls to the zeroth-order oracle) to be set as
        $\mu=\mathcal{O}\left({\epsilon}/{d^{3/2}}\right),\ N = \mathcal{O}\left({d}/{\epsilon^2}\right).$
\end{theorem}

\begin{proof}{Proof of \cref{thm1}}
    From \cref{L-manifold-smooth} we have
    \[f(x_{k+1})\leq f(x_k)-\eta_k\langle g_\mu (x_k),\grad f(x_k) \rangle+\frac{\eta_k^2 L_g}{2}\|g_\mu (x_k)\|^2.\]
    Taking the expectation w.r.t. $u_k$ on both sides, we have
    \begin{align*}
        & \E_{u_k}\left[f(x_{k+1})\right] \leq f(x_k)-\eta_k\langle\E_{u_k}(g_\mu (x_k)),\grad f(x_k) \rangle+\frac{\eta_k^2 L_g}{2}\E_{u_k}(\|g_\mu (x_k)\|^2) \\
        \leq & f(x_k)-\eta_k\langle\E_{u_k}(g_\mu (x_k)),\grad f(x_k) \rangle +\frac{\eta_k^2 L_g}{2}\left(\frac{\mu^2}{2}L_g^2(d+6)^3+2(d+4)\|\grad f(x_k)\|^2\right),
    \end{align*}
    where the last inequality is by \cref{lemma123}. Now Take $\eta_k=\hat{\eta}=\frac{1}{2(d+4)L_g}$, we have
    \begin{align*}
        &\E_{u_k}\left[f(x_{k+1})\right]\\
        \leq & f(x_k)+\frac{\hat{\eta}}{2}( \|\grad f(x_k)\|^2-2\langle\E_{u_k}(g_\mu (x_k)),\grad f(x_k) \rangle )+\frac{\mu^2 L_g}{16}\frac{(d+6)^3}{(d+4)^2} \\
        =&f(x_k)+\frac{\hat{\eta}}{2}(\|\grad f(x_k)-\E_{u_k}(g_\mu (x_k))\|^2-\|\E_{u_k}(g_\mu (x_k))\|^2)+\frac{\mu^2 L_g}{16}\frac{(d+6)^3}{(d+4)^2} \\
        \leq & f(x_k)+\frac{\hat{\eta}}{2}\left(\frac{\mu^2 L_g^2}{4}(d+3)^{3}-\frac{1}{2}\|\grad f(x_k)\|^2+\frac{\mu^2}{4}L_g(d+6)^3\right)+\frac{\mu^2 L_g}{16}\frac{(d+6)^3}{(d+4)^2} \\
        =&f(x_k)-\frac{\hat{\eta}}{4}\|\grad f(x_k)\|^2+C(\mu),
    \end{align*}
    where the second inequality is from \cref{lemma123}. Define $\phi_k:=f(x_k)-f(x^*)$. Now take the expectation w.r.t. $\mathcal{U}_k=\{u_0,u_1,\ldots,u_{k-1}\}$,  we have
    $$
        \phi_{k+1}\leq \phi_k- \frac{\hat{\eta}}{4}\E_{\mathcal{U}_k}\|\grad f(x_k)\|^2+C(\mu).
    $$
    Summing the above inequality over $k=0,\ldots,N$ yields \eqref{thm1-eq1}.

    Therefore with $\mu=\mathcal{O}({\epsilon}/{d^{3/2}})$ we have $C(\mu)\leq {\hat{\eta}\epsilon^2}/{4}$. Taking $N\geq 8(d+4)L_g{(f(x_0)-f(x^*))}/{\epsilon^2}$ yields \eqref{thm1-eq2}. In summary, the number of iterations for obtaining an $\epsilon$-stationary solution is $\mathcal{O}({d}/{\epsilon^2})$, and hence the total zeroth-order oracle complexity is also $\mathcal{O}({d}/{\epsilon^2})$.
\end{proof}


\begin{remark}\label{rmk1}
Note that in \cref{algorithm1}, we only sample one Gaussian vector in each iteration of the algorithm. In practice, one can also sample multiple Gaussian random vectors in each iteration and obtain an averaged gradient estimator. Suppose we sample $m$ i.i.d. Gaussian random vectors in each iteration and use the average $\bar{g}_\mu(x)=\frac{1}{m}\sum_{i=1}^{m}g_{\mu,i}(x)$, then the bound for our zeroth-order estimator becomes
    \begin{equation}\label{multi_sample_bound}
            \E(\|\bar{g}_\mu(x)-\grad f(x)\|^2)\leq \mu^2L_g^2(d+6)^3+\frac{2(d+4)}{m}\|\grad f(x)\|^2.
    \end{equation}
Hence, the final result in \cref{thm1} can be improved to
\begin{equation}\label{multi_sample_bound-2}
        \begin{split}
            \frac{1}{N+1}\sum_{k=0}^{N}\E_{\mathcal{U}_k}\|\grad f(x_k)\|^2\leq 4 L_g \frac{f(x_0)-f(x^*)}{N+1}+\mu^2 L_g^2(d+6)^3,
        \end{split}
    \end{equation}
with $\hat{\eta}={1}/{L_g}$ and $C(\mu)={\mu^2 L_g}(d+6)^3/2$. Therefore the number of iterations required is improved to
$N = \mathcal{O}(1/\epsilon^2)$ when we set $\mu=\mathcal{O}(\epsilon/d^{3/2})$ and $ m=\mathcal{O}(d)$. However, the zeroth-order oracle complexity is still $\mathcal{O}({d}/{\epsilon^2})$. The proof of \eqref{multi_sample_bound} and \eqref{multi_sample_bound-2} is given in the appendix. This multi-sampling technique will play a key role in our stochastic and non-smooth case analyses.
\end{remark}

\subsection{Zeroth-Order Stochastic Riemannian Optimization for Nonconvex Problem}

In this section, we focus on the following nonconvex smooth problem:
\begin{equation}\label{prob:nonconvex-smooth}
    \min_{x\in\M} f(x):= \int_{\xi} F(x,\xi)d P(\xi),
\end{equation}
where $P$ is a random distribution, $F$ is a function satisfying \cref{L-manifold-smooth}, in variable $x$, almost surely. Note that $f$ automatically satisfies \cref{L-manifold-smooth} 
by the Jensen's inequality.

In the stochastic case, sampling multiple times in every iteration can improve the convergence rate. Our zeroth-order Riemannian gradient estimator is given by
\begin{equation}\label{stocashtic_oracle}
    \bar{g}_{\mu,\xi}(x) = \frac{1}{m}\sum_{i=1}^{m}g_{\mu,\xi_i}(x),~\text{where}~~g_{\mu,\xi_i}(x) = \frac{F(R_x(\mu u_i),\xi_i)-F(x,\xi_i)}{\mu} u_i,
\end{equation}
and $u_i$ is a standard normal random vector on $T_x\M$. We also immediately have that
\begin{equation}
    \E_{\xi_i} g_{\mu,\xi_i}(x) = \frac{f(R_x(\mu u))-f(x)}{\mu} u = g_{\mu}(x).
\end{equation}

The multi-sampling technique enables us to obtain the following bound on $\E\|\bar{g}_{\mu,\xi}(x)-\grad f(x)\|^2$, the proof of which is given in the Appendix \ref{sec:lemma4.2}.

\begin{lemma}\label{lemma:stocgradzor}
{For the Riemannian gradient estimator in~\eqref{stocashtic_oracle}, under Assumptions \ref{L-manifold-smooth} and \ref{sto_assumption}, we have}
\begin{equation}\label{stochastic_ineq}
\begin{split}
    &\E\|\bar{g}_{\mu,\xi}(x)-\grad f(x)\|^2\leq \mu^2 L_g^2(d+6)^3+\frac{8(d+4)}{m}\sigma^2+\frac{8(d+4)}{m}\|\grad f(x)\|^2,
\end{split}
\end{equation}
where the expectation $\E$ is taken for both Gaussian vectors $\mathcal{U}=\{u_1,...,u_m\}$ and $\xi$.
\end{lemma}

Our zeroth-order Riemannian stochastic gradient descent algorithm (\texttt{ZO-RSGD}) for solving \eqref{prob:nonconvex-smooth}, is presented in \cref{algorithm2}.

\begin{algorithm}[ht]
   \caption{Zeroth-order Riemannian Stochastic Gradient Descent (\texttt{ZO-RSGD})}
   \label{algorithm2}
\begin{algorithmic}[1]
   \STATE {\bfseries Input:} Initial point $x_0\in\M$, smoothing parameter $\mu$, multi-sample constant $m$, step size $\eta_k$, fixed number of iteration $N$.
   \FOR{$k=0$ {\bfseries to} $N-1$}
   \STATE Sample the standard Gaussian random vectors $u^k_i$ on $T_{x_k}\M$ by orthogonal projection in \cref{def_zero_rgrad}, and sample $\xi^k_i$, $i=1,...,m$.
   \STATE Compute the zeroth-order gradient $\bar{g}_{\mu,\xi}(x_k)$ by \cref{stocashtic_oracle}.
   \STATE Update $x_{k+1}=R_{x_k}(-\eta_k \bar{g}_{\mu,\xi}(x_k))$.
   \ENDFOR
\end{algorithmic}
\end{algorithm}


Now we present convergence analysis for obtaining an $\epsilon$-stationary point of \eqref{prob:nonconvex-smooth}.

\begin{theorem}\label{thm2}
    Let  $F$ satisfy \cref{L-manifold-smooth}, w.r.t. variable $x$ almost surely. Suppose $\{x_k\}$ is the sequence generated by \cref{algorithm2} with the stepsize $\eta_k=\hat{\eta}=\frac{1}{L_g}$. Under~\cref{sto_assumption}, we have
    \begin{align}\label{thm2-eq1}
        \frac{1}{N+1}\sum_{k=0}^{N}\E_{\mathcal{U}_k,\Xi_k}\|\grad f(x_k)\|^2\leq  4 L_g \frac{f(x_0)-f(x^*)}{N+1}+C(\mu),
    \end{align}
    where $C(\mu)=2\mu^2 L_g^2(d+6)^3+\frac{16(d+4)}{m}\sigma^2$, $\mathcal{U}_k$ denotes the set of all Gaussian random vectors and $\Xi_k$ denotes the set of all random variable $\xi_k$ in the first $k$ iterations. In order to have $\frac{1}{N+1}\sum_{k=0}^{N}\E_{\mathcal{U}_k,\Xi_k}\|\grad f(x_k)\|^2\leq \epsilon^2$, we need the smoothing parameter $\mu$, number of sampling $m$ in each iteration and number of iterations $N$ to be
    \begin{equation}
        \mu=\mathcal{O}\left({\epsilon}/{d^{3/2}}\right),\ m=\mathcal{O}\left({d\sigma^2}/{\epsilon^2}\right),\ N = \mathcal{O}\left({1}/{\epsilon^2}\right).
    \end{equation}
    Hence, the number of calls to the zeroth-order oracle is $m N=\mathcal{O}({d}/{\epsilon^4})$.
\end{theorem}

\begin{proof}{Proof of \cref{thm2}}
    From \cref{L-manifold-smooth}, we have:
    $$
        f(x_{k+1})\leq f(x_k)-\eta_k\langle \bar{g}_{\mu,\xi}(x),\grad f(x_k) \rangle+\frac{\eta_k^2 L_g}{2}\|\bar{g}_{\mu,\xi}(x)\|^2
    $$
    Take $\eta_k=\hat{\eta}=\frac{1}{L_g}$, we have
    \begin{equation*}
    \begin{split}
        f(x_{k+1})&\leq f(x_k)-\eta_k\langle \bar{g}_{\mu,\xi}(x),\grad f(x_k) \rangle+\frac{\eta_k^2 L_g}{2}\|\bar{g}_{\mu,\xi}(x)\|^2 \\
        &=f(x_k)+\frac{1}{2 L_g}\left(\|\bar{g}_{\mu,\xi}(x)-\grad f(x)\|^2- \|\grad f(x)\|^2\right).
    \end{split}
    \end{equation*}
    Take the expectation for the random variables at iteration $k$ on both sides, we have
    \begin{equation*}
    \begin{split}
        &\E_{k} f(x_{k+1})\leq f(x_k)+\frac{1}{2 L_g}\left(\E_{k}\|\bar{g}_{\mu,\xi}(x)-\grad f(x)\|^2- \|\grad f(x)\|^2\right) \\
        \cref{multi_sample_bound}&\leq f(x_k)+\frac{1}{2 L_g}\left(\mu^2 L_g^2(d+6)^3+\frac{8(d+4)}{m}\sigma^2+\left(\frac{8(d+4)}{m}-1\right)\|\grad f(x)\|^2\right).
    \end{split}
    \end{equation*}
    Summing up over $k=0,...,N$ (assuming that $m\geq 16(d+4)$) yields \eqref{thm2-eq1}.
    In summary, the total number of iterations for obtaining an $\epsilon$-stationary solution of \eqref{prob:nonconvex-smooth} is $\mathcal{O}({1}/{\epsilon^2})$, and the stochastic zeroth-order oracle complexity is $\mathcal{O}({d}/{\epsilon^4})$.
\end{proof}

{In Appendix~\ref{sec:geoconvex}, we present the oracle complexity of Algorithm~\ref{algorithm2} when $f$ is geodesically convex and $\mathcal{M}$ is the Hadamard manifold.}

\subsection{Zeroth-order Stochastic Riemannian Proximal Gradient Method}
We now consider the general optimization problem of the form in \cref{problem}. For the sake of notation, we denote $p(x):= f(x) +h(x)$. We assume that $\M$ is a compact submanifold, $h$ is convex in the embedded space $\R^n$ and is also Lipschitz continuous with parameter $L_h$, and $f(x):= \int_{\xi} F(x,\xi)d P(\xi)$ satisfying \cref{sto_assumption}. 


The non-differentiability of $h$ prohibits Riemannian gradient methods to be applied directly. In \cite{chen2018proximal}, by assuming that the exact gradient of $f$ is available, a manifold proximal gradient method (ManPG) is proposed for solving \eqref{problem}. 
One typical iteration of ManPG is as follows:
\begin{equation}\label{eq:manpg}
\begin{split}
    v_k & :={\argmin}~\langle \grad f(x_k), v \rangle + \frac{1}{2t}\|v\|^2+h(x_k+v), \ \st, \ {v\in T_{x_k}\M}\\
    x_{k+1} & :=R_{x_k}(\eta_k v_k),
\end{split}
\end{equation}
where $t>0$ and $\eta_k>0$ are step sizes.
In this section, we develop a zeroth-order counterpart of ManPG (\texttt{ZO-ManPG}), where we assume that only noisy function evaluations of $f$ are available. The following lemma from~\cite{chen2018proximal} provides a notion of stationary point that is useful for our analysis. 
\begin{lemma}\label{lemma_opt}
Let $\bar{v}_k$ be the minimizer of the $v$-subproblem in \eqref{eq:manpg}.
    If $\bar{v}_k=0$, then $x_k$ is a stationary point of problem \eqref{problem}. We say $x_k$ is an $\epsilon$-stationary point of \eqref{problem} with $t=\frac{1}{L_g}$, if $\|\bar{v}_k\| \leq {\epsilon}/{L_g}$.
\end{lemma}

Our \texttt{ZO-ManPG} iterates as:
\begin{equation}\label{sub_prob}
\begin{split}
    v_k &:={\argmin}~\langle \bar{g}_{\mu,\xi}(x_k), v \rangle + \frac{1}{2t}\|v\|^2+h(x_k+v), \ \st, \ {v\in T_{x_k}\M},\\
    x_{k+1} & :=R_{x_k}(\eta_k v_k),
\end{split}
\end{equation}
where $\bar{g}_{\mu,\xi}(x_k)$ is defined in \cref{stocashtic_oracle}. Note that the only difference between \texttt{ZO-ManPG} \eqref{sub_prob} and ManPG \eqref{eq:manpg} is that in \eqref{sub_prob} we use $\bar{g}_{\mu,\xi}(x)$ to replace the Riemannian gradient $\grad f$ in \eqref{eq:manpg}. A more complete description of the algorithm is given in \cref{algorithm_manpg}. 
\begin{algorithm}[ht]
   \caption{Zeroth-Order Riemannian Proximal Gradient Descent (\texttt{ZO-ManPG})}
   \label{algorithm_manpg}
\begin{algorithmic}[1]
   \STATE {\bfseries Input:} Initial point $x_0$ on $\M$, smoothing parameter $\mu$, number of multi-sample $m$, step size $\eta_k$, fixed number of iteration $N$.
   \FOR{$k=0$ {\bfseries to} $N-1$}
   \STATE Sample $m$ standard Gaussian random vector $u_i$ on $T_{x_{k}}\M$ by by orthogonal projection in \cref{def_zero_rgrad}, $i=1,...,m$.
   \STATE Compute the zeroth-order gradient the random oracle $\bar{g}_\mu(x_{k})$ by \cref{stocashtic_oracle}.
   \STATE Solve $v_k$ from \cref{sub_prob}.
   \STATE Update $x_{k+1}=R_{x_k}(\eta_k v_k)$.
   \ENDFOR
\end{algorithmic}
\end{algorithm}
Now we provide some useful lemmas for analyzing the iteration complexity of \cref{algorithm_manpg}. 
\begin{lemma}(Non-expansiveness)\label{lemma_nonexpansive}
    Suppose $v:=\arg\min_{v\in T_{x}\M} \langle g_1, v \rangle + \frac{1}{2t}\|v\|^2+h(x+v)$ and $w:=\arg\min_{w\in T_{x}\M} \langle g_2, w \rangle + \frac{1}{2t}\|w\|^2+h(x+w)$. Then we have
    \begin{equation}\label{non-expansive}
        \|v-w\|\leq t \|g_1-g_2\|.
    \end{equation}
\end{lemma}

\begin{proof}{Proof of \cref{lemma_nonexpansive}}
    By the first order optimality condition \cite{yang2014optimality}, we have $0\in \frac{1}{t}v+g_1+\proj_{T_{x}\M}\partial h(x+v)$ and $0\in \frac{1}{t}w+g_2+\proj_{T_{x}\M}\partial h(x+w)$, i.e. $\exists p_1\in\partial h(x+v)$ and $p_2\in\partial h(x+w)$ such that $v = -t(g_1+\proj_{T_{x}\M}(p_1))$ and $ w = -t(g_2+\proj_{T_{x}\M}(p_2))$. Therefore we have
    \begin{equation}\label{c1-c2}
    \begin{split}
        \langle v, w-v \rangle &= t\langle g_1+\proj_{T_{x}\M}(p_1),v-w \rangle \\
        \langle w, v-w \rangle &= t\langle g_2+\proj_{T_{x}\M}(p_2),w-v \rangle.
	\end{split}
    \end{equation}
    Now since $v,w\in T_{x}\M$, and using the convxity of $h$, we have
    \begin{equation}\label{c3}
        \langle \proj_{T_{x}\M}(p_1), v-w \rangle =\langle p_1,v-w \rangle
        =\langle p_1,(v+x)-(w+x) \rangle \geq h(v+x)-h(w+x).
    \end{equation}
    Substituting \cref{c1-c2} and into \eqref{c3} yields,
    \begin{equation*}
    \begin{split}
        \langle v, w-v \rangle & \geq t\langle g_1,v-w \rangle + h(v+x)-h(w+x) \\
        \langle w, v-w \rangle & \geq t\langle g_2,w-v \rangle + h(w+x)-h(v+x).
    \end{split}
    \end{equation*}
    Summing these two inequalities gives $\langle v-w,v-w \rangle\leq t\langle g_2- g_1,v-w \rangle$, and \cref{non-expansive} follows by applying the Cauchy-Schwarz inequality.
\end{proof}

\begin{corollary}\label{coro_manpg}
    Suppose $v_k$ is given by \eqref{sub_prob}, and $\bar{v}_k$ is solution of the $v$-subproblem in \cref{eq:manpg}, then we have
    \begin{equation*}
    \begin{split}
        \E_{\mathcal{U}_k,\Xi_k} \|v_k-\bar{v}_k\|_F^2\leq t^2\left( \mu^2 L_g^2(d+6)^3+\frac{8(d+4)}{m}\sigma^2+\frac{8(d+4)}{m}\|\grad f(x_k)\|^2\right).
    \end{split}
    \end{equation*}
\end{corollary}
\begin{proof}{Proof of \cref{coro_manpg}}
By \cref{lemma_nonexpansive}, we have
    $$
        \E_{\mathcal{U}_k,\Xi_k} \|v_k-\bar{v}_k\|_F^2\leq t^2\E_{\mathcal{U}_k,\Xi_k} \|\bar{g}_{\mu,\xi}(x_k)-\grad f(x_k)\|_F^2.
    $$
   From \cref{lemma:stocgradzor},
    \begin{equation*}
    \begin{split}
      &  \E_{\mathcal{U}_k,\Xi_k} \|\bar{g}_{\mu,\xi}(x_k)-\grad f(x_k)\|_F^2 \\ \leq &  \mu^2 L_g^2(d+6)^3+\frac{8(d+4)}{m}\sigma^2+\frac{8(d+4)}{m}\|\grad f(x_k)\|^2.
    \end{split}
    \end{equation*}
   The desired result hence follows by combining these two inequalities.
\end{proof}

The following lemma shows the sufficient decrease property for one iteration of \texttt{ZO-ManPG}.

\begin{lemma}\label{lemma_manpg_update}
    For any $t>0$, there exists a constant $\bar{\eta}>0$ such that for any $0\leq \eta_k\leq \min\{1, \bar{\eta}\}$, the $(x_k,v_k)$ generated by \cref{algorithm_manpg} satisfies 
    \begin{equation}\label{lemma_manpg_update-eq1}
        p(x_{k+1})-p(x_k)\leq -\left(\frac{\eta_k}{2t}-\tilde{C}\right)\|v_k\|^2,
    \end{equation}
    where $\tilde{C}=\mu^2 L_g^2(d+6)^3+\frac{8(d+4)}{m}\sigma^2+\frac{8(d+4)}{m}G^2$ and $G$ is the upper bound of the Riemannian gradient $\grad f(x)$ (existence by the compactness of $\M$).
\end{lemma}

\begin{proof}{Proof of \cref{lemma_manpg_update}}
    Notice that
    \begin{equation*}
    \begin{split}
        &f(x_{k+1})-f(x_k)\leq \langle\grad f(x_k),R_{x_k}(\eta_k v_k) - x_k\rangle + \frac{L_g}{2}\|R_{x_k}(\eta_k v_k) - x_k\|^2 \\
        &=\langle\grad f(x_k) - \bar{g}_{\mu,\xi}(x), R_{x_k}(\eta_k v_k) - x_k\rangle + \langle\bar{g}_{\mu,\xi}(x), R_{x_k}(\eta_k v_k) - x_k\rangle + \frac{L_g}{2}\|R_{x_k}(\eta_k v_k) - x_k\|^2,
    \end{split}
    \end{equation*}
    where the inequality follows from Assumption \ref{L-manifold-smooth}. Moreover, by Lemma \ref{lemma:stocgradzor} and the Fact 3.6 of \cite{chen2018proximal}, we have
    \begin{equation*}
    \begin{split}
        &\langle\grad f(x_k) - \bar{g}_{\mu,\xi}(x), R_{x_k}(\eta_k v_k) - x_k\rangle\leq \|\grad f(x_k) - \bar{g}_{\mu,\xi}(x)\|\|R_{x_k}(\eta_k v_k) - x_k\|\\
        &\leq M_1^2\eta_k^2\left[\mu^2 L_g^2(d+6)^3+\frac{8(d+4)}{m}\sigma^2+\frac{8(d+4)}{m}\|\grad f(x)\|^2\right]\|v_k\|^2.
    \end{split}
    \end{equation*}
    The rest of the proof of bounding $\langle\bar{g}_{\mu,\xi}(x), R_{x_k}(\eta_k v_k) - x_k\rangle + \frac{L_g}{2}\|R_{x_k}(\eta_k v_k) - x_k\|^2$ follows from exactly the same process as in (\cite{chen2018proximal}, Lemma 5.2). We omit the details for brevity. 
\end{proof}

\begin{theorem}\label{thm3}
    Under~\cref{sto_assumption} and~\cref{L-manifold-smooth}, the sequence generated by \cref{algorithm_manpg}, with $\eta_k=\hat{\eta} < \min\{1, \bar{\eta}\}$ and $t={1}/{L_g}$, satisfies:
    \begin{equation}\label{thm3-eq1}
    \begin{split}
        \frac{1}{N}\sum_{k=0}^{N-1} \E_{\mathcal{U}_k,\Xi_k}\|\bar{v}_k\|^2 &\leq \frac{4t(p(x_0)-p(x^*))}{(\hat{\eta}-8\tilde{C}) t N}+\frac{\hat{\eta}N t^2}{\hat{\eta}-8\tilde{C}t}\tilde{C}+\frac{8 t^3}{\hat{\eta}-8\tilde{C}t}\tilde{C}^2,
    \end{split}
    \end{equation}
    where $\tilde{C}=\mu^2 L_g^2(d+6)^3+\frac{8(d+4)}{m}\sigma^2+\frac{8(d+4)}{m}G^2$ and $G$ is the upper bound of the Riemannian gradient $\grad f(x)$ over the manifold $\M$. To guarantee 
    \[\min_{k=0,...,N-1} \E_{\mathcal{U}_k,\Xi_k}\|\bar{v}_k\|_F^2\leq{\epsilon^2}/{L_g^2},\] 
    the parameters need to be set as:
    $\mu=\mathcal{O}\left({\epsilon}/{d^{3/2}}\right)$, $m=\mathcal{O}\left({d G^2}/{\epsilon^2}\right)$, $N=\mathcal{O}\left({1}/{\epsilon^2}\right).$
    Hence, the number of calls to the stochastic zeroth-order oracle is $\mathcal{O}({d}/{\epsilon^4})$.
\end{theorem}

\begin{proof}{Proof of \cref{thm3}}
    Summing up \eqref{lemma_manpg_update-eq1} over $k=0,\ldots,N-1$ and using  \cref{coro_manpg}, we have:
    \begin{equation*}
    \begin{split}
        &p(x_0)-\E_{\mathcal{U}_k,\Xi_k} p(x_k)\geq\sum_{k=0}^{N-1}[\frac{\eta_k}{2 t}-\tilde{C}]\E_{\mathcal{U}_k}\|v_k\|_F^2\geq [\frac{\hat{\eta}}{4 t}-2\tilde{C}] \sum_{k=0}^{N-1}2\E_{\mathcal{U}_k,\Xi_k}\|v_k\|_F^2\\
        &\geq [\frac{\hat{\eta}}{4 t}-2\tilde{C}]\sum_{k=0}^{N-1}\left[ \E_{\mathcal{U}_k,\Xi_k}\|\bar{v}_k\|_F^2- t^2\left( \mu^2 L_g^2(d+6)^3+\frac{8(d+4)}{m}\sigma^2\right.\right.\\ & \left.\left.+\frac{8(d+4)}{m}\|\grad f(x_k)\|^2\right) \right] \\
%
        &\geq [\frac{\hat{\eta}}{4 t}-2\tilde{C}]\sum_{k=0}^{N-1} \E_{\mathcal{U}_k,\Xi_k}\|\bar{v}_k\|_F^2 - \frac{\hat{\eta} N t}{4}\left(\mu^2 L_g^2(d+6)^3+\frac{8(d+4)}{m}\sigma^2+\frac{8(d+4)}{m}G^2\right)\\ &
        +2t^2\left(\mu^2 L_g^2(d+6)^3+\frac{8(d+4)}{m}\sigma^2+\frac{8(d+4)}{m}G^2\right)^2,
    \end{split}
    \end{equation*}
    which immediately implies the desired result \eqref{thm3-eq1}.
\end{proof}


\begin{remark}
    The subproblem \cref{sub_prob} is the main computational effort in \cref{algorithm_manpg}. 
    Fortunately, this subproblem can be efficiently solved by a regularized semi-smooth Newton's method when $\M$ takes certain forms. We refer the reader to   \cite{xiao2018regularized,chen2018proximal} for more details. 
\end{remark}

\subsection{Escaping saddle points: Zeroth-order stochastic cubic regularized Newton's method over Riemannian manifolds}\label{sec_cubic}

In this section, we consider the problem of escaping saddle-points and converging to local minimizers in a stochastic zeroth-order Riemannian setting. Towards that, we leverage the Hessian estimator methodology developed in Section~\ref{sec:hessianestimator} and analyze a zeroth-order Riemannian stochastic cubic regularized Newton's method (\texttt{ZO-RSCRN}) for solving \eqref{prob:nonconvex-smooth}, which provably escapes the saddle points. 
Our approach is motivated by \cite{zhang2018cubic}, where the authors proposed the minimization of function $m_{x,\sigma}(\eta) = f(x)+\langle \grad f(x), \eta \rangle+\frac{1}{2}\langle P_x\circ \hess f(x)\circ P_x[\eta], \eta \rangle+\frac{\alpha_k}{6}\|\eta\|^3$ at each iteration. The zeroth-order counterpart replaces the Riemannian gradient and Hessian with the corresponding zeroth-order estimators. The proposed \texttt{ZO-RSCRN} algorithm is described in \cref{algorithm_cubic}. In \texttt{ZO-RSCRN}, the function in the cubic regularized subproblem is
\begin{equation}\label{cubic_subproblem}
    \hat{m}_{x,\alpha}(\eta) = f(x)+\langle \bar{g}_{\mu,\xi}(x), \eta \rangle+\frac{1}{2}\langle \bar{H}_{\mu,\xi}(x)[\eta], \eta \rangle+\frac{\alpha}{6}\|\eta\|^3.
\end{equation}
Note that if $\hat{\eta}=\operatorname{argmin}_{\eta} \hat{m}_{x,\alpha}(\eta)$, then the projection $P_x(\hat{\eta})$ is also a minimizer, because $\bar{g}_{\mu,\xi}(x)$ and $\bar{H}_{\mu,\xi}(x)$ only take effect on the component that is in $T_x\M$. 

\begin{algorithm}[tb]
   \caption{ Zeroth-Order Riemannian Stochastic Cubic Regularized Newton's Method \texttt{(ZO-RSCRN)}}
   \label{algorithm_cubic}
\begin{algorithmic}[1]
   \STATE {\bfseries Input:} Initial point $x_0$ on $\M$, smoothing parameter $\mu$, multi-sample parameter $m$ and $b$, cubic regularization parameter $\alpha$, number of iteration $N$.
   \FOR{$k=0$ {\bfseries to} $N-1$}
   \STATE Compute $\bar{g}_{\mu,\xi}(x_k)$ and $\bar{H}_{\mu,\xi}(x_k)$ based on~\eqref{stocashtic_oracle} and~\eqref{multi_sample_hessian} respectively.
   \STATE Solve $\eta_k=\operatorname{argmin}_{\eta} \hat{m}_{x_k,\alpha}(\eta)$, where $\hat{m}_{x,\alpha}(\eta)$ is defined in~\eqref{cubic_subproblem}.
   \STATE Update $x_{k+1}=R_{x_k}(P_x(\eta_k))$.
   \ENDFOR
\end{algorithmic}
\end{algorithm}

\begin{theorem}\label{thm:scrn}
    For manifold $\M$ and function $f:\M\rightarrow\R$ under Assumptions \ref{L-manifold-smooth}, \ref{assumption_hessian} and \ref{sto_assumption}, define $k_{\min}:=\argmin_{k}\E_{\mathcal{U}_k,\Xi_k}\|\eta_k\|$, then the update in \cref{algorithm_cubic} with $\alpha\geq L_H$ satisfies:
    \begin{equation}\label{thm:scrn-optcond}
    \E\|g_{k_{\min}+1}\|\leq \mathcal{O}(\epsilon), \mbox{ and } \E [\lambda_{\min}(\hess f_{k_{\min}+1})]\geq -\mathcal{O}(\sqrt{\epsilon}),
    \end{equation}
    given that the parameters satisfy:
    \begin{equation}\label{thm:scrn-param}
        N=\mathcal{O}\left({1}/{\epsilon^{3/2}}\right),\  \mu=\mathcal{O}\left(\min\left\{\frac{\epsilon}{d^{3/2}},\sqrt{\frac{\epsilon}{d^5}}\right\}\right),\ m=\mathcal{O}({d}/{\epsilon^2}),\ b=\mathcal{O}({d^4}/{\epsilon}),
    \end{equation}
    where $\lambda_{\min}$ denotes the smallest eigenvalue. 
    Hence, the zeroth-order oracle complexity is $\mathcal{O}({d}/{\epsilon^{7/2}}+{d^4}/{\epsilon^{5/2}})$.
\end{theorem}
\begin{proof}{Proof of \cref{thm:scrn}}
    Denote $f_k=f(x_k)$, $g_k=\grad f(x_k)$ and $\E=\E_{\mathcal{U}_k,\Xi_k}$ for ease of notation. We first provide the global optimality conditions of subproblem \cref{cubic_subproblem} following \cite{nesterov2006cubic}:
    \begin{equation}\label{cubic_optimality}
        (\bar{H}_{\mu,\xi}(x)+\lambda^*I)\eta + \bar{g}_{\mu,\xi}(x)=0,\ \lambda^*=\frac{\alpha}{2}\|\eta\|,\ \bar{H}_{\mu,\xi}(x)+\lambda^*I\succeq 0.
    \end{equation}
    Since the parallel transport $P_{\eta}$ is an isometry, we have 
    \begin{equation*}
    \begin{split}
        &\|g_{k+1}\|= \|P_{\eta_k}^{-1}g_{k+1}\| \\
        =&\|(P_{\eta_k}^{-1}g_{k+1}-g_{k}-\hess f_{k}[\eta_k]) + (g_{k}-\bar{g}_{\mu,\xi}(x_k))\\&+(\hess f_{k}[\eta_k] - \bar{H}_{\mu,\xi}(x_k)[\eta_k]) + (\bar{g}_{\mu,\xi}(x_k) + \bar{H}_{\mu,\xi}(x_k)[\eta_k])\| \\
        \leq & \|P_{\eta_k}^{-1}g_{k+1}-g_{k}-\hess f_{k}[\eta_k]\| + \|g_{k}-\bar{g}_{\mu,\xi}(x_k)\| \\&+ \|\hess f_{k}[\eta_k]-\bar{H}_{\mu,\xi}(x_k)[\eta_k]\| + \|\bar{g}_{\mu,\xi}(x_k)+\bar{H}_{\mu,\xi}(x_k)[\eta_k]\| \\
        \cref{coro_hessian_assumption}\leq & \frac{L_H}{2}\|\eta_k\|^2 + \|g_{k}-\bar{g}_{\mu,\xi}(x_k)\| \\&+ \|\hess f_{k}[\eta_k]-\bar{H}_{\mu,\xi}(x_k)[\eta_k]\| + \|\bar{g}_{\mu,\xi}(x_k)+\bar{H}_{\mu,\xi}(x_k)[\eta_k]\| \\
        \cref{cubic_optimality} = & \frac{L_H}{2}\|\eta_k\|^2 + \|g_{k}-\bar{g}_{\mu,\xi}(x_k)\|+ \|\hess f_{k}[\eta_k]-\bar{H}_{\mu,\xi}(x_k)[\eta_k]\| + \lambda^*\|\eta_k\|\\
        \cref{cubic_optimality} \leq & \frac{L_H}{2}\|\eta_k\|^2 + \|g_{k}-\bar{g}_{\mu,\xi}(x_k)\|+ \|\hess f_{k}[\eta_k]-\bar{H}_{\mu,\xi}(x_k)\|_{\operatorname{op}}\|\eta_k\| + \frac{\alpha}{2}\|\eta_k\|^2 \\
       \leq & \frac{L_H}{2}\|\eta_k\|^2+\|g_{k}-\bar{g}_{\mu,\xi}(x_k)\|+\frac{1}{2}\|\hess f_{k}-\bar{H}_{\mu,\xi}(x_k)\|_{\operatorname{op}}^2+\frac{1}{2}\|\eta_k\|^2+\frac{\alpha}{2}\|\eta_k\|^2.
    \end{split}
    \end{equation*}
    Taking expectation on both sides of the above inequality gives (by \cref{stochastic_ineq} and \cref{zohessian_approx_square_moment})
    \begin{equation}\label{temp_1_cubic}
        \E\|g_{k+1}\|-\sqrt{\delta_g}-\delta_H\leq\frac{1}{2}(L_{H}+\alpha+1+2L_2 \|g_k\|)\E\|\eta_k\|^2,
    \end{equation}
    where $\delta_g = \mu^2L_g^2(d+6)^3+\frac{8(d+4)}{m}(G^2+\sigma^2)$, 
  $G$ is the upper bound of $\|\grad f\|$ over $\M$, and $\delta_H = \frac{(d+16)^4}{b}L_g + \frac{\mu^2 L_H^2}{18}(d+6)^5$. 
Since $P_{\eta_k}^{-1}$ is an isometry, we have: 
    \begin{equation*}
    \begin{split}
        &\lambda_{\min}(\hess f_{k+1}) = \lambda_{\min}(P_{\eta_k}^{-1}\circ \hess f_{k+1} \circ P_{\eta_k}) \\
        \geq & \lambda_{\min}(P_{\eta_k}^{-1}\circ \hess f_{k+1} \circ P_{\eta_k} - \hess f_{k}) \\ &+\lambda_{\min}(\hess f_{k} - \bar{H}_{\mu,\xi}(x_k)) + \lambda_{\min}(\bar{H}_{\mu,\xi}(x_k)) \\
        \cref{assumption_hessian_eq}\geq & -L_H\|\eta_k\|+\lambda_{\min}(\hess f_{k} - \bar{H}_{\mu,\xi}(x_k)) + \lambda_{\min}(\bar{H}_{\mu,\xi}(x_k)) \\
        = & \lambda_{\min}(\hess f_{k} - \bar{H}_{\mu,\xi}(x_k)) + \lambda_{\min}(\bar{H}_{\mu,\xi}(x_k)-L_H\|\eta_k\| I) \\
        \cref{cubic_optimality} \geq & \lambda_{\min}(\hess f_{k} - \bar{H}_{\mu,\xi}(x_k)) - \frac{\alpha+2L_H}{2}\|\eta_k\|.
    \end{split}
    \end{equation*}    
    Taking expectation, we obtain (by \cref{zohessian_approx_square_moment})
    \begin{equation}\label{temp_2_cubic}
        \frac{\alpha+2L_{H}}{2}\E\|\eta_k\|\geq -(\sqrt{\delta_{H}}+\E\lambda_{\min}(\hess f_{k+1})).
    \end{equation}
    Now we will upper bound $\E\|\eta_k\|$. From \cref{assumption_hessian}, we have
    \begin{equation}\label{cubic_proof_main}
    \begin{split}
        \hat{f}_{x_k}(\eta_k)&\leq f(x_k) + g_k^\top  \eta_k  + \frac{1}{2}\eta_k^\top  H_k \eta_k + \frac{L_H}{6}\|\eta_k\|^3 \\
        &= \left(f(x_k)+\bar{g}_\mu(x_k)^\top  \eta_k +\frac{1}{2} \eta_k^\top  \bar{H}_\mu(x_k)\eta_k +\frac{L_H}{6}\|\eta_k\|^3 \right) \\ &+ \left( (g_k - \bar{g}_\mu(x_k))^\top  \eta_k + \frac{1}{2}\eta_k^\top  (H_k-\bar{H}_\mu(x_k)) \eta_k\right).
    \end{split}
    \end{equation}
    Using \cref{cubic_optimality} we have
    \begin{equation}\label{cubic_proof_eq1}
    \begin{split}
        &f(x_k)+\bar{g}_\mu(x_k)^\top  \eta_k +\frac{1}{2} \eta_k^\top  \bar{H}_\mu(x_k)\eta_k +\frac{L_H}{6}\|\eta_k\|^3 \\
       =& f(x_k) -\frac{1}{2}\eta_k^\top  \bar{H}_\mu(x_k)\eta_k + (\frac{L_H}{6} - \frac{\alpha}{2})\|\eta_k\|^3 \\
      =& f(x_k) -\frac{1}{2}\eta_k^\top  (\bar{H}_\mu(x_k)+\frac{\alpha}{2}\|\eta_k\|I)\eta_k - (\frac{\alpha}{4} - \frac{L_H}{6})\|\eta_k\|^3 \\
       \leq & f(x_k) - (\frac{\alpha}{4} - \frac{L_H}{6})\|\eta_k\|^3
        \leq f(x_k)-\frac{\alpha}{12}\|\eta_k\|^3,
    \end{split}
    \end{equation}
    where the last inequality is due to $\alpha\geq L_H$. Moreover, by Cauchy-Schwarz inequality and Young's inequality, we have 
    \begin{equation}\label{cubic_proof_eq2}
    \begin{split}
        &\E\left[ (g_k - \bar{g}_\mu(x_k))^\top  \eta_k + \frac{1}{2}\eta_k^\top  (H_k-\bar{H}_\mu(x_k)) \eta_k \right] \\
       \leq & \E\|g_k - \bar{g}_\mu(x_k)\|\|\eta_k\| + \frac{1}{2}\E\| H_k-\bar{H}_\mu(x_k)\|_{\operatorname{op}}\|\eta_k\|^2 \\
       \leq & \frac{32}{3\alpha}\E\|g_k-\bar{g}_\mu(x_k)\|^{3/2}+\frac{12}{\alpha}\E\| H_k-\bar{H}_\mu(x_k)\|_{\operatorname{op}}^3+\frac{\alpha}{24}\E\|\eta_k\|^3.
    \end{split}
    \end{equation}
    Plugging \eqref{cubic_proof_eq1} and \eqref{cubic_proof_eq2} to \cref{cubic_proof_main}, we have 
    \begin{equation}\label{cubic_proof_eq3}
    \E f_{k+1}\leq f_k - \frac{\alpha}{24}\E\|\eta_k\|^3 + \frac{32}{3 L_H}\delta_g^{3/4}+\frac{12}{L_H}\tilde{\delta}_H,
    \end{equation} 
    where $\tilde{\delta}_H=\tilde{C}\frac{(d+16)^6}{b^{3/2}}L_g^{1.5} + \frac{1}{27}\mu^3L_H^3 (d+6)^{7.5}$. Taking the sum for \eqref{cubic_proof_eq3} over $k=0,\ldots,N-1$, we have
    \begin{equation*}
    \begin{split}
        \frac{1}{N}\sum_{k=0}^{N}\E\|\eta_k\|^3\leq \frac{24}{L_H}\left(\frac{f_{0}-f^*}{N}+\frac{32}{3 L_H}\delta_g^{3/4}+\frac{12}{L_H}\tilde{\delta}_H\right),
    \end{split}
    \end{equation*}
    which together with \eqref{thm:scrn-param} yields 
    \begin{equation}\label{cubic_proof_eq4}
        \E\|\eta_{k_{\min}}\|^3\leq \mathcal{O}(\epsilon^{3/2}), \mbox{ and } \E \|\eta_{k_{\min}}\|^2\leq \mathcal{O}(\epsilon).
    \end{equation}
   Combining \cref{cubic_proof_eq4}, \cref{temp_1_cubic} and \cref{temp_2_cubic} yields \eqref{thm:scrn-optcond}. 
\end{proof}

\begin{remark}
    To solve the subproblem, we implement the same Krylov subspace method as in \cite{agarwal2018adaptive}, where the Riemannian Hessian and vector multiplication is approximated by Lanczos iterations. Note also that in our setting, we only require vector-vector multiplications due to the structure of our Hessian estimator in~\cref{Gauss_oracle_second_order}. For the purpose of brevity, we refer to \cite{carmon2018analysis,agarwal2018adaptive} for a comprehensive study of this method.
\end{remark}

\section{Numerical Experiments and Applications}\label{experiments}
We now explore the performance of the proposed algorithms on various simulation experiments. Finally, we demonstrate the applicability of stochastic zeroth-order Riemannian optimization for the problems of zeroth-order attacks on deep neural networks and controlling stiffness matrix in robotics. 
We conducted our experiments on a desktop with Intel Core 9600K CPU and NVIDIA GeForce RTX 2070 GPU. 

\subsection{Simulation Experiments} For all the simulation experiments listed below, we plot the average result over 100 runs.

{\bf Experiment 1: Procrustes problem \cite{absil2009optimization}.} This is a matrix linear regression problem on a given manifold: $\min_{X\in\M} \|AX-B\|_F^2$, where $X\in\R^{n\times p}$, $A\in\R^{l\times n}$ and $B\in\R^{l\times p}$. The manifold we use is the Stiefel manifold $\M = \St(n, p)$. 
In our experiment, we pick up different dimension $n\times p$ and record the time cost to achieve prescribed precision $\epsilon$. The entries of matrix $A$ are generated by standard Gaussian distribution. We compare our \texttt{ZO-RGD} (\cref{algorithm1}) with the first-order Riemannian gradient method (RGD) on this problem. The results are shown in Table \ref{tab:table_1}. For each run, we sample $m=n\times p$ Gaussian samples for each iteration. The multi-sample version of \texttt{ZO-RGD} closely resembles the convergence rate of RGD, as shown in \cref{fig_proscrutes}. These results 
indicate our zeroth-order method \texttt{ZO-RGD} is comparable with its first-order counterpart RGD, though the former one only uses zeroth-order information. 

\begin{table}[t!]
\begin{center}
\begin{small}
\begin{sc}
\begin{tabular}{|c|c|c|c|c|}
\hline
Dimension & $\epsilon$ & Stepsize & No. iter. \texttt{ZO-RGD} & Aver. No. iter. RGD \\
\hline
$15\times 5$ & $10^{-3}$ & $10^{-2}$ & $460\pm 137$ & 442 \\
$25\times 15$ & $10^{-3}$ & $10^{-2}$ & $892\pm 99$ & 852 \\
$50\times 20$ & $10^{-2}$ & $5\times10^{-3}$ & $255\pm 26$ & 236 \\
\hline
\end{tabular}
\end{sc}
\end{small}
\end{center}
\caption{Comparison of \texttt{ZO-RGD} and RGD on the Procrustes problem.}\label{tab:table_1}
\end{table}

\begin{figure*}[t!]
\begin{center}
\subfigure[$(n,p)=(15,5)$]{\includegraphics[clip, trim=4cm 8cm 4cm 8cm,width=0.32\columnwidth]{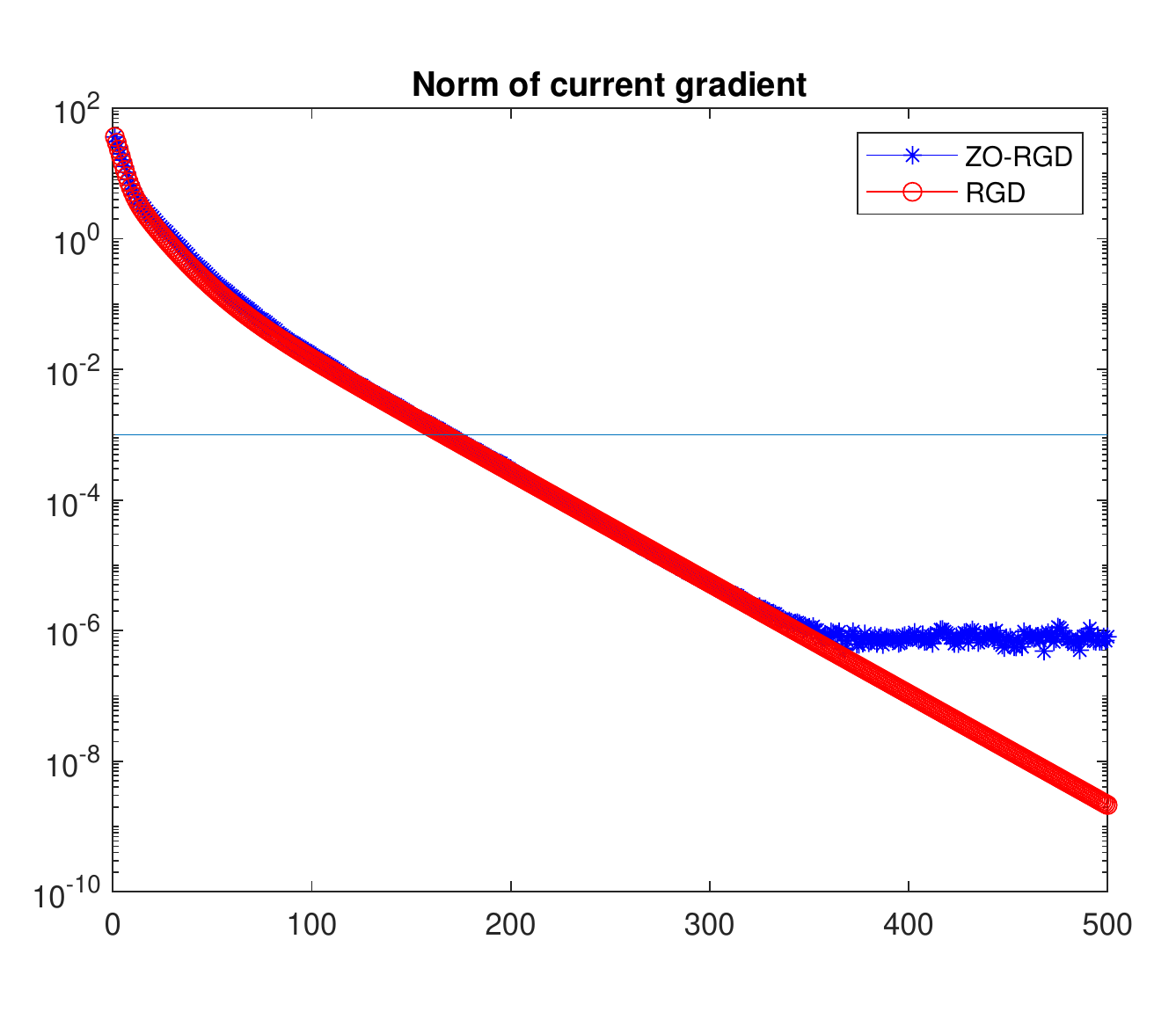}}
\subfigure[$(n,p)=(25,15)$]{\includegraphics[clip, trim=4cm 8cm 4cm 8cm,width=0.32\columnwidth]{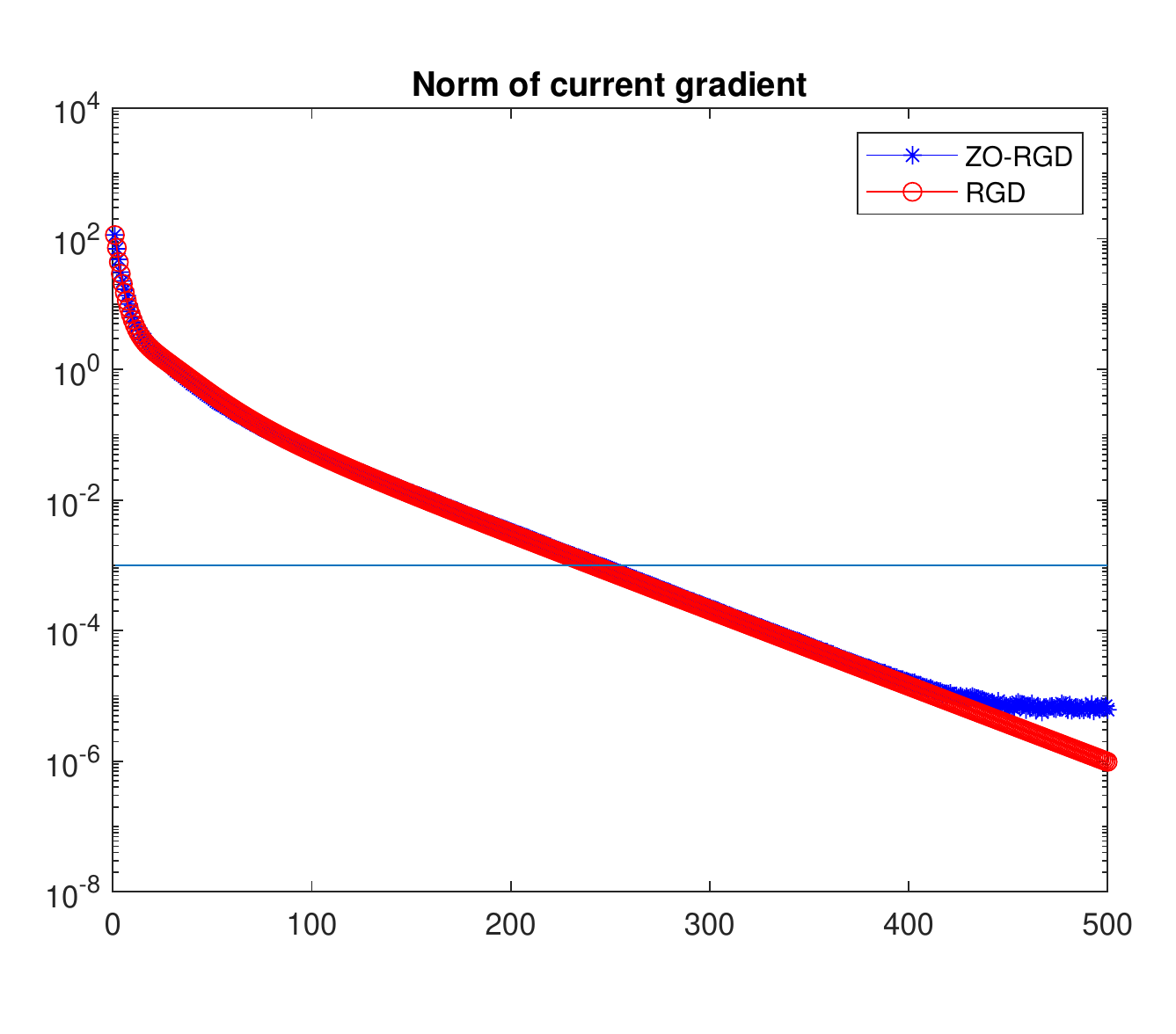}}
\subfigure[$(n,p)=(50,20)$]{\includegraphics[clip, trim=4cm 8cm 4cm 8cm,width=0.32\columnwidth]{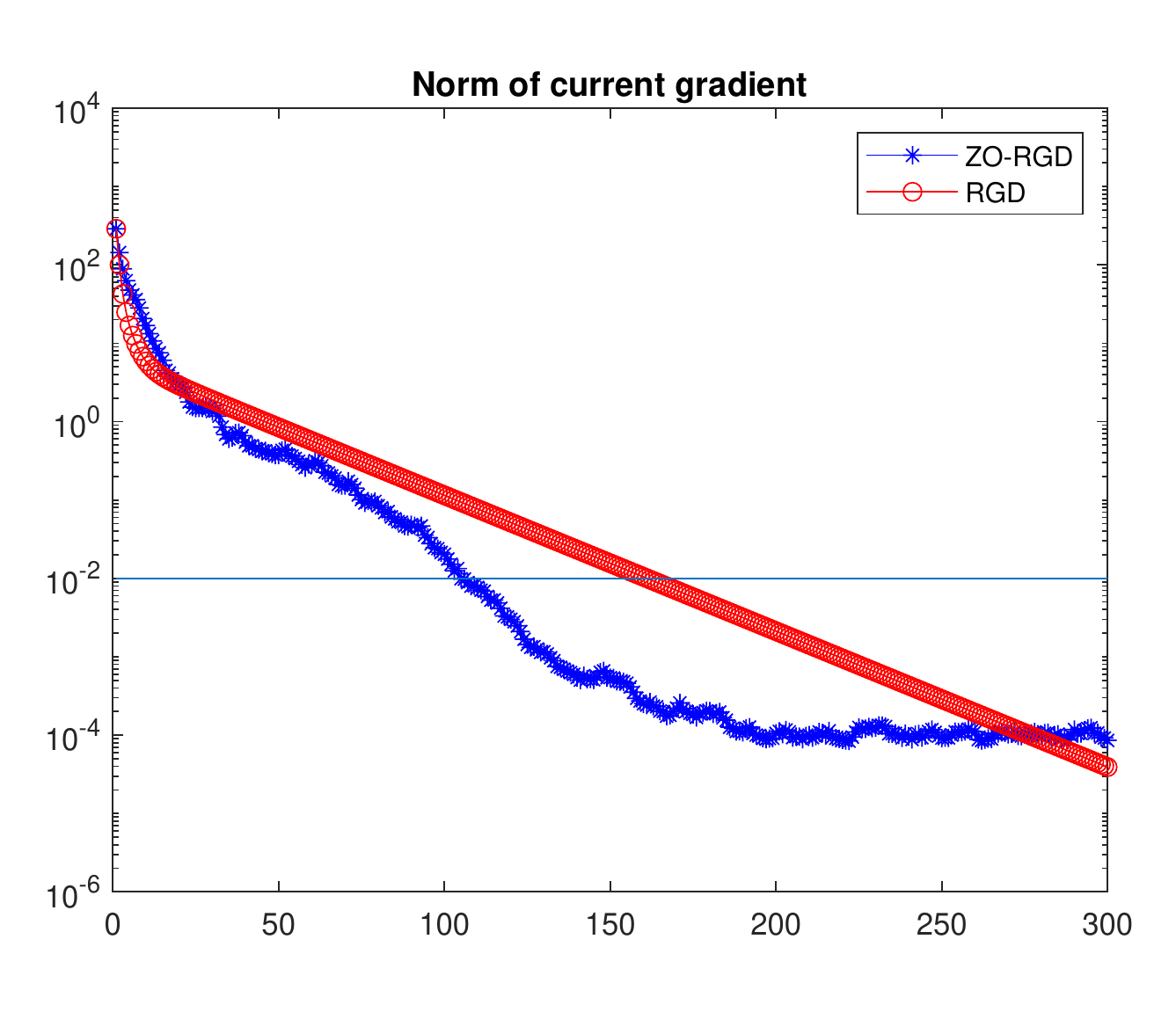}}
\caption{The convergence curve of ZO-RGD v.s. RGD. x-axis is the number of iterations and y-axis is the norm of Riemannian gradient at corresponding points. Note that our zeroth-order algorithm does not use gradient information in updates, while the graph still shows the norm of gradient to show the effectiveness of our method. The horizontal lines are the prescribed precisions. }
\label{fig_proscrutes}
\end{center}
\end{figure*}

{\bf Experiment 2: k-PCA \cite{zhang2016fast,tripuraneni2018averaging,zhou2019faster}.} k-PCA on Grassmann manifold is a Rayleigh quotient minimization problem. Given a symmetric positive definite matrix $H\in\R^{n\times n}$, we need to solve $\min_{X\in\operatorname{Gr}(n,p)}-\frac{1}{2}\Tr (X^\top  H X)$. The Grassmann manifold $\operatorname{Gr}(n,p)$ is the set of $p$-dimensional subspaces in $\mathbb{R}^n$. We refer the reader to \cite{absil2009optimization} for more details about the Grassmann quotient manifold. This problem can be written as a finite sum problem: $\min_{X\in\operatorname{Gr}(n,p)}\sum_{i=1}^{n}-\frac{1}{2}\Tr (X^\top  h_i h_i^\top  X)$, where $h_i\in\R^{n}$ and $H=\sum_{i=1}^{n} h_i h_i^\top$. We compare our \texttt{ZO-RSGD} algorithm (\cref{algorithm2}) and its first-order counterpart RSGD on this problem. The results are shown in \cref{fig_kpca} (a) and (d). 
In our experiment, we set $n=100$, $p=50$, and the matrix $H$ is generated by $H=AA^\top $, where $A\in\R^{n\times p}$ is a normalized randomly generated data matrix. From \cref{fig_kpca} (a) and (d), we see that the performance of \texttt{ZO-RSGD} is similar to its first-order counterpart RSGD. 

\begin{figure}[t!]
\begin{center}
\subfigure[Experiment 2: $m=20$]{\includegraphics[clip, trim=3.5cm 8cm 4cm 8cm,width=0.3\columnwidth]{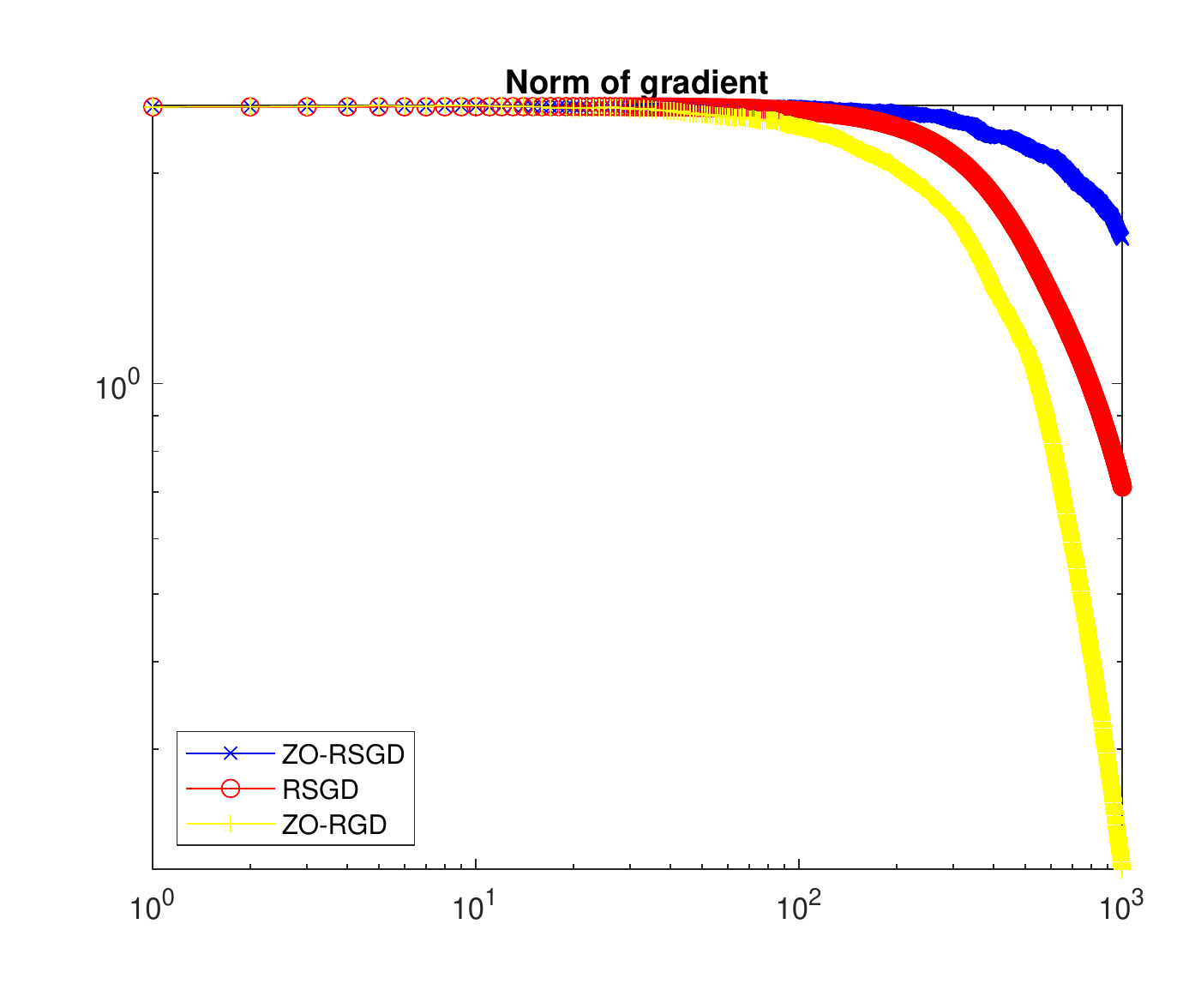}}
\subfigure[Experiment 3]{\includegraphics[clip, trim=3.5cm 8cm 4cm 8cm,width=0.3\columnwidth]{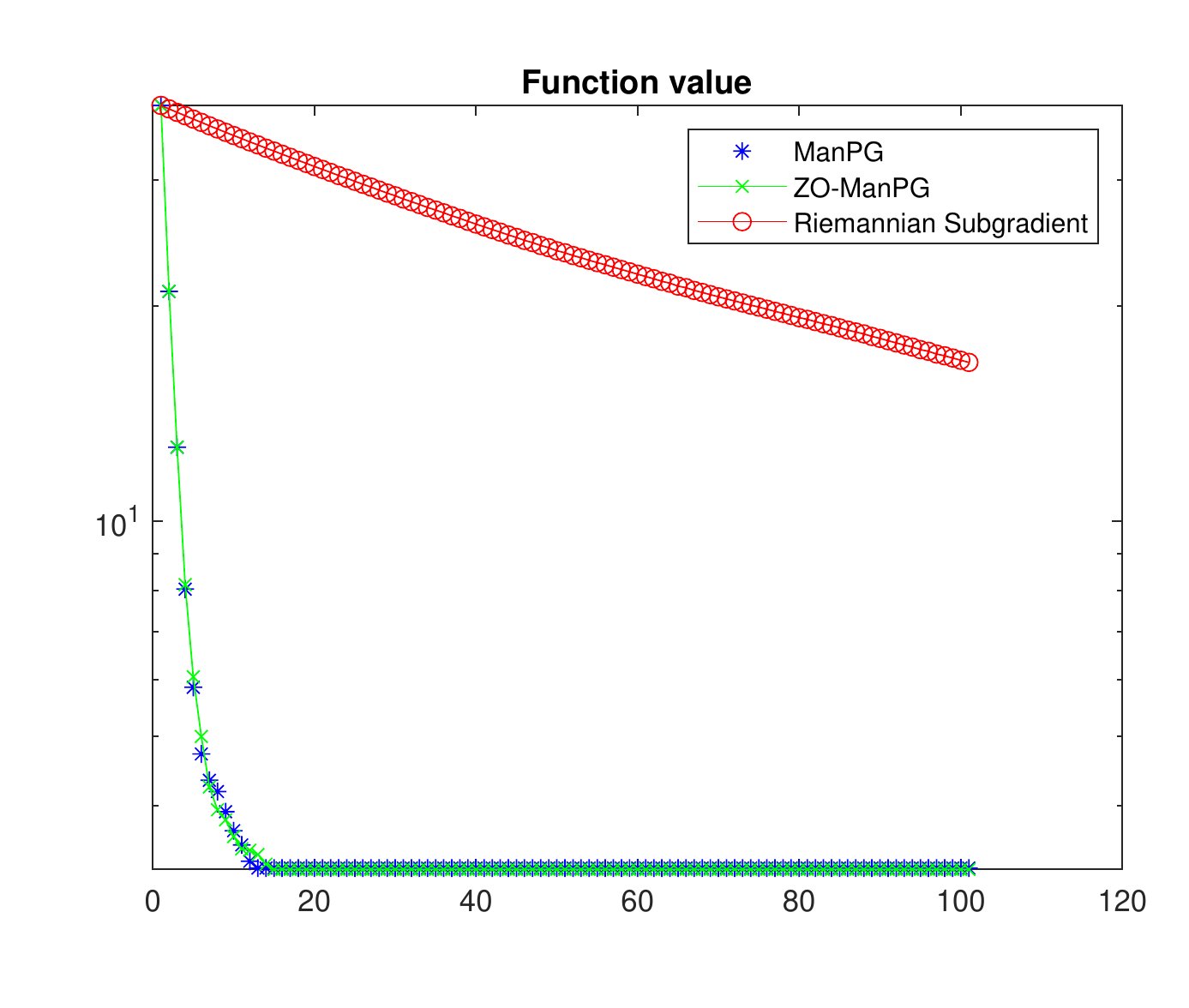}}
\subfigure[Experiment 4]{\includegraphics[clip, trim=3.5cm 8cm 4cm 8cm,width=0.3\columnwidth]{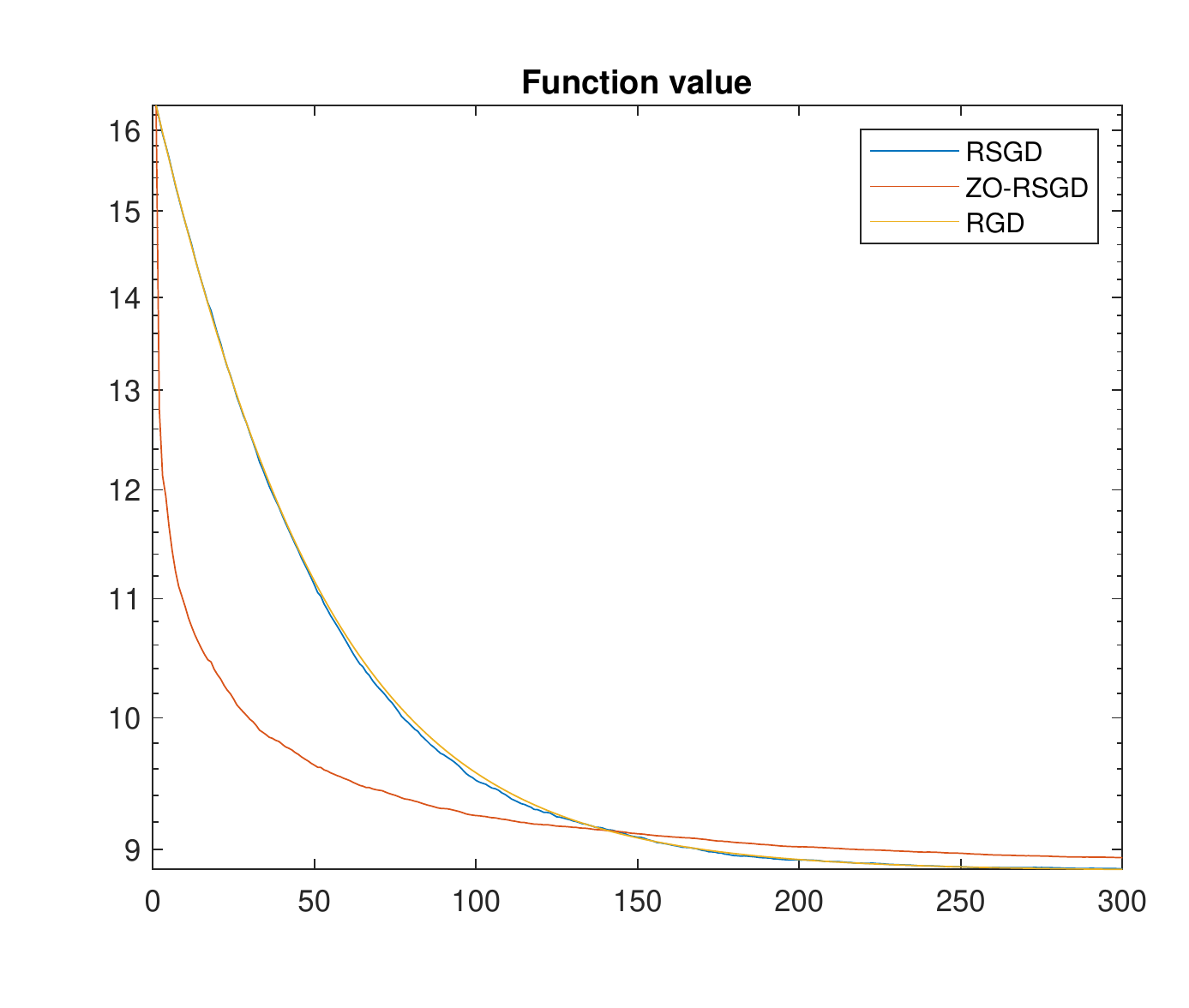}}\\
\subfigure[Experiment 2: $m=40$]{\includegraphics[clip, trim=3.5cm 8cm 4cm 8cm,width=0.3\columnwidth]{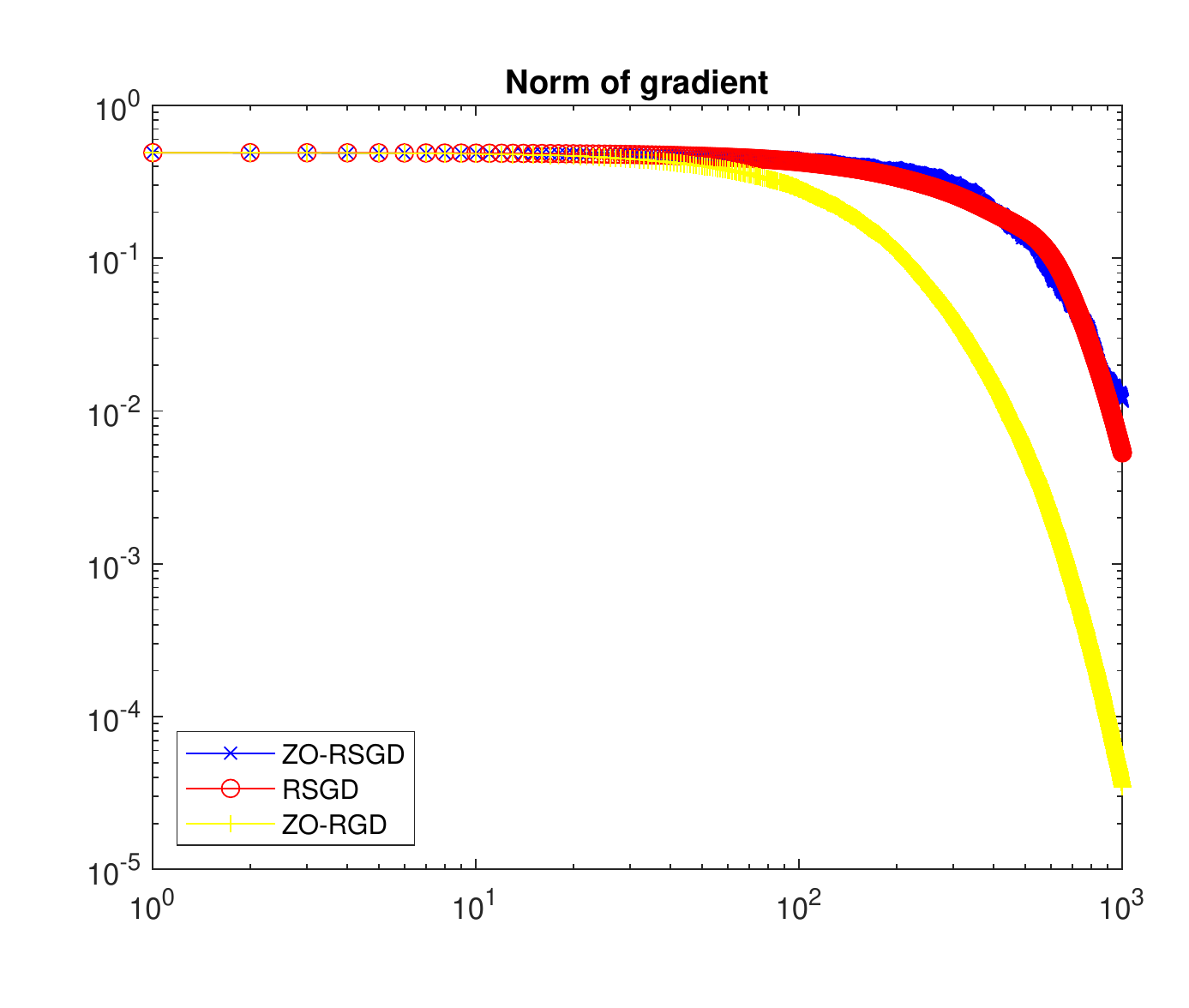}}
\subfigure[Experiment 3]{\includegraphics[clip, trim=3.5cm 8cm 4cm 8cm,width=0.3\columnwidth]{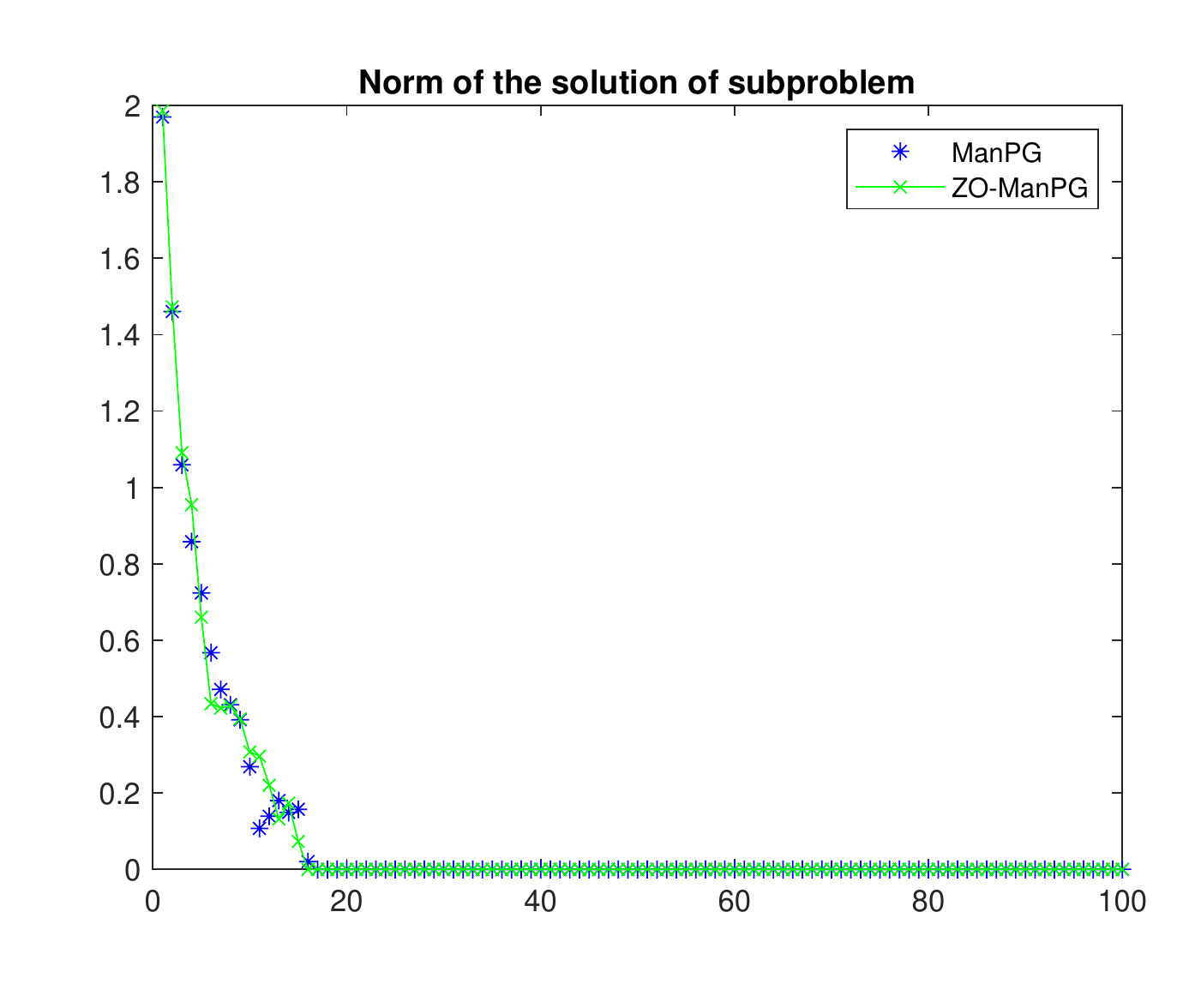}}
\subfigure[Experiment 4]{\includegraphics[clip, trim=3.5cm 8cm 4cm 8cm,width=0.3\columnwidth]{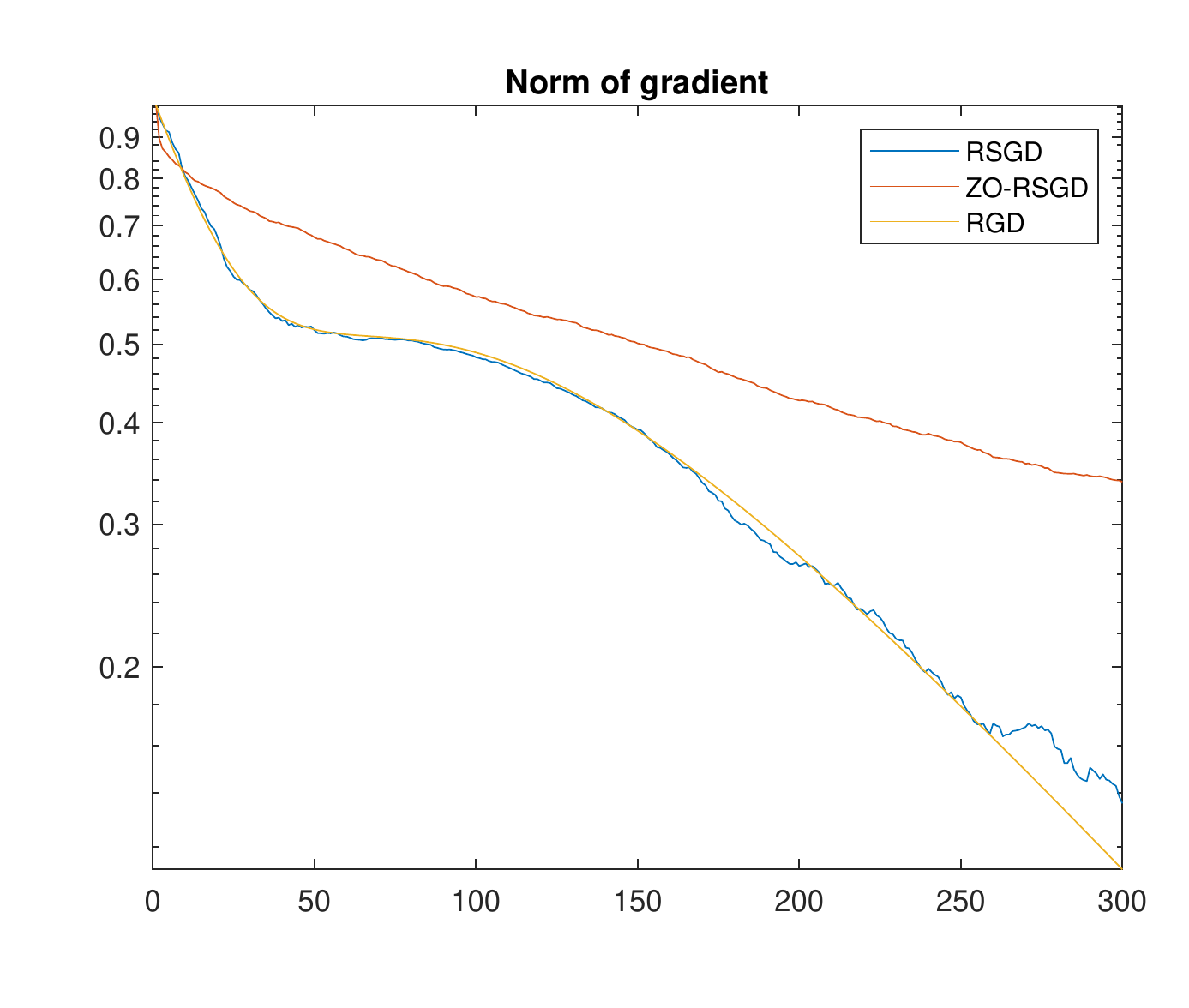}}
\caption{The convergence of three numerical experiments. The $x$-axis always denotes the number of iterations. Figures (a) and (d) are results for k-PCA (Experiment 2). Here three algorithms are compared: \texttt{ZO-RSGD} (\cref{algorithm2}), RSGD, and \texttt{ZO-RGD} (\cref{algorithm1}). Figures (b) and (e) are results for sparse PCA (Experiment 3) in which the $y$-axis of Figure (e) denotes the norm of $v_k$ in  \eqref{eq:manpg} (for ManPG) and \eqref{sub_prob} (for \texttt{ZO-ManPG}), which actually measures the optimality of the problem. Here three algorithms are compared: \texttt{ZO-ManPG} (\cref{algorithm_manpg}), ManPG and Riemannian subgradient method. Figures (c) and (f) are results for Karcher mean of PSD matrices problem (Experiment 4). Here three algorithms are compared: RSGD, \texttt{ZO-RSGD} (\cref{algorithm2}), and RGD.} 
\label{fig_kpca}
\end{center}
\end{figure}

{\bf Experiment 3: Sparse PCA \cite{Jolliffe2003,Zou-spca-2006,zou2018selective}.} The sparse PCA problem, arising in statistics, is a Riemannian optimization problem over the Stiefel manifold with nonsmooth objective: $\min_{X\in \St(n, p)}-\frac{1}{2}\Tr (X^\top  A^\top A X)+\lambda \|X\|_1$. Here, $A\in \R^{m\times n}$ is the normalized data matrix. We compare our \texttt{ZO-ManPG} (\cref{algorithm_manpg}) with ManPG \cite{chen2018proximal} and Riemannian subgradient method \cite{li2019nonsmooth}. In our numerical experiments, we chose $(m,n,p)=(50,100,10)$, and entries of $A$ are drawn from Gaussian distribution and rows of $A$ are then normalized. The comparison results are shown in \cref{fig_kpca} (b) and (e). These results show that our \texttt{ZO-ManPG} is comparable to its first-order counterpart ManPG and they are both much better than the Riemannian subgradient method. 


{\bf Experiment 4: Karcher mean of given PSD matrices \cite{bini2013computing,zhang2016first,kasai2018riemannian}.} Given a set of positive semidefinite (PSD) matrices $\{A_i\}_{i=1}^{n}$ where $A_i\in\R^{d\times d}$ and $A_i\succeq 0$, we want to calculate their Karcher mean: $ \min _{X \in \mathcal{S}_{++}^{d}} \frac{1}{2 n} \sum_{i=1}^{n}\left(\operatorname{dist}\left(X, A_{i}\right)\right)^{2}$, where $\operatorname{dist}\left(X, Y\right) = \|\operatorname{logm}(X^{-1/2} Y X^{-1/2})\|_{F}$ ($\operatorname{logm}$ stands for matrix logarithm) represents the distance along the corresponding geodesic between the two points $X, Y \in \mathcal{S}_{++}^{d}$. 
This experiment serves as an example of optimizing geodesically convex functions over Hadamard manifolds, with \texttt{ZO-RSGD} (\cref{algorithm2}).
In our numerical experiment, we take $d=3$ and $n=500$. 
We compare our \texttt{ZO-RSGD} algorithm with its first-order counterpart RSGD and RGD. The results are shown in \cref{fig_kpca} (c) and (f), and from these results we see that ZO-RSGD is comparable to its first-order counterpart RSGD in terms of function value, though it is inferior to RSGD and RGD in terms of the size of the gradient. 

{\bf Experiment 5: Procrustes problem with \texttt{ZO-RSCRN}.} Here, we consider the Procrustes problem in Experiment 1 and use the \texttt{ZO-RSCRN} with both estimated gradients and Hessians. Following~\cite{agarwal2018adaptive}, we use the gradient norm as a performance measure (although the algorithm converges to local-minimzers). We use the Lanczos method (specifically Algorithm 2 from~\cite{agarwal2018adaptive}) for solving the sub-problem in Step 4.  Furthermore, as we are estimating the second order information, we set $n=6$ and $p=4$ and consider $\epsilon=10^{-3}$. In  Figure~\ref{fig:robotics}, (a), we plot the gradient norm versus iterations for Riemannian Stochastic Cubic-Regularized Newton method in the zeroth order and second-order setting. We notice that the zeroth-order method compares favourably to the second-order counterpart in terms of iteration complexity. Admittedly, scaling up the~\texttt{ZO-RSCRN} method to work in higher-dimensions, based on variance reduction techniques, is an interesting problem that we plan to tackle as future work.


\subsection{Real world applications}\label{sec:revisitrealworld}


\noindent {\bf Black-box stiffness control for robotics.} We now study the first motivating example discussed in Section~\ref{sec:roboticsapp} on the control of robotics with the policy parameter being the stiffness matrix $\boldsymbol{K}^{\mathcal{P}} \in \mathcal{S}^d_{++}$, see \cite{jaquier2020bayesian} for more engineering details. Mathematically, given the current position of robot $\hat{\boldsymbol{p}}$ and current speed $\dot{\boldsymbol{p}}$, the task is to minimize
\begin{equation}\label{robotics_function}
f(\boldsymbol{K}^{\mathcal{P}})=w_{\boldsymbol{p}}\|\hat{\boldsymbol{p}}-\boldsymbol{p}\|^{2}+w_{d} \operatorname{det}(\boldsymbol{K}^{\mathcal{P}})+w_{c} \operatorname{cond}(\boldsymbol{K}^{\mathcal{P}})    
\end{equation}
with $\boldsymbol{p}$ being the new position, and $\operatorname{cond}$ is the condition number. With a constant external force $\boldsymbol{f}^{e}$ applied to the system, we have the following identity which solves $\boldsymbol{p}$ by $\boldsymbol{K}^{\mathcal{P}}$: $\boldsymbol{f}^{e}=\boldsymbol{K}^{\mathcal{P}}(\hat{\boldsymbol{p}}-\boldsymbol{p})-\boldsymbol{K}^{\mathcal{D}} \dot{\boldsymbol{p}}$, where the damping matrix $\boldsymbol{K}^{\mathcal{D}}=\boldsymbol{K}^{\mathcal{P}}$ for critical damped case. As the stiffness matrix is a positive definite matrix, the above optimization problem is a Riemannian optimization problem over the positive definite manifold (where the manifold structure is the same as the Karcher mean problem). The function $f$ is not known analytically and following~\cite{jaquier2020bayesian}, we use a simulated setting for a robot (7-DOF Franka Emika Panda robot) to evaluate the function $f$ for a given value of $\boldsymbol{K}^{\mathcal{P}}$, with the same parameters as in \cite{jaquier2020bayesian}. We compare our \texttt{ZO-RGD} method with Euclidean Zeroth-order gradient descent (ZO-GD) method \cite{balasubramanian2019zeroth}. We test the cases when $d=2$ and $d=3$ for minimizing function $f$ w.r.t $\boldsymbol{K}^{\mathcal{P}}$, and the results are shown in Figure \ref{fig:robotics}, (b) and (c). In our experiments, the stepsize of ZO-GD is $3\times10^{-4}$ and \texttt{ZO-RGD} is $10^{-3}$.  Note that for ZO-GD method, one has to project the matrix back to the positive definite set, whereas the \texttt{ZO-RGD} method intrinsically guarantees that the iterates are positive definite, thus is much more stable. Also, due to the fact that \texttt{ZO-RGD} is more stable, the stepsize of \texttt{ZO-RGD} can be larger than ZO-GD, which results in faster convergence. 

\begin{figure}[t!]
\begin{center}
\subfigure[SCRN]{\includegraphics[width=5.2cm, height=4.17cm]{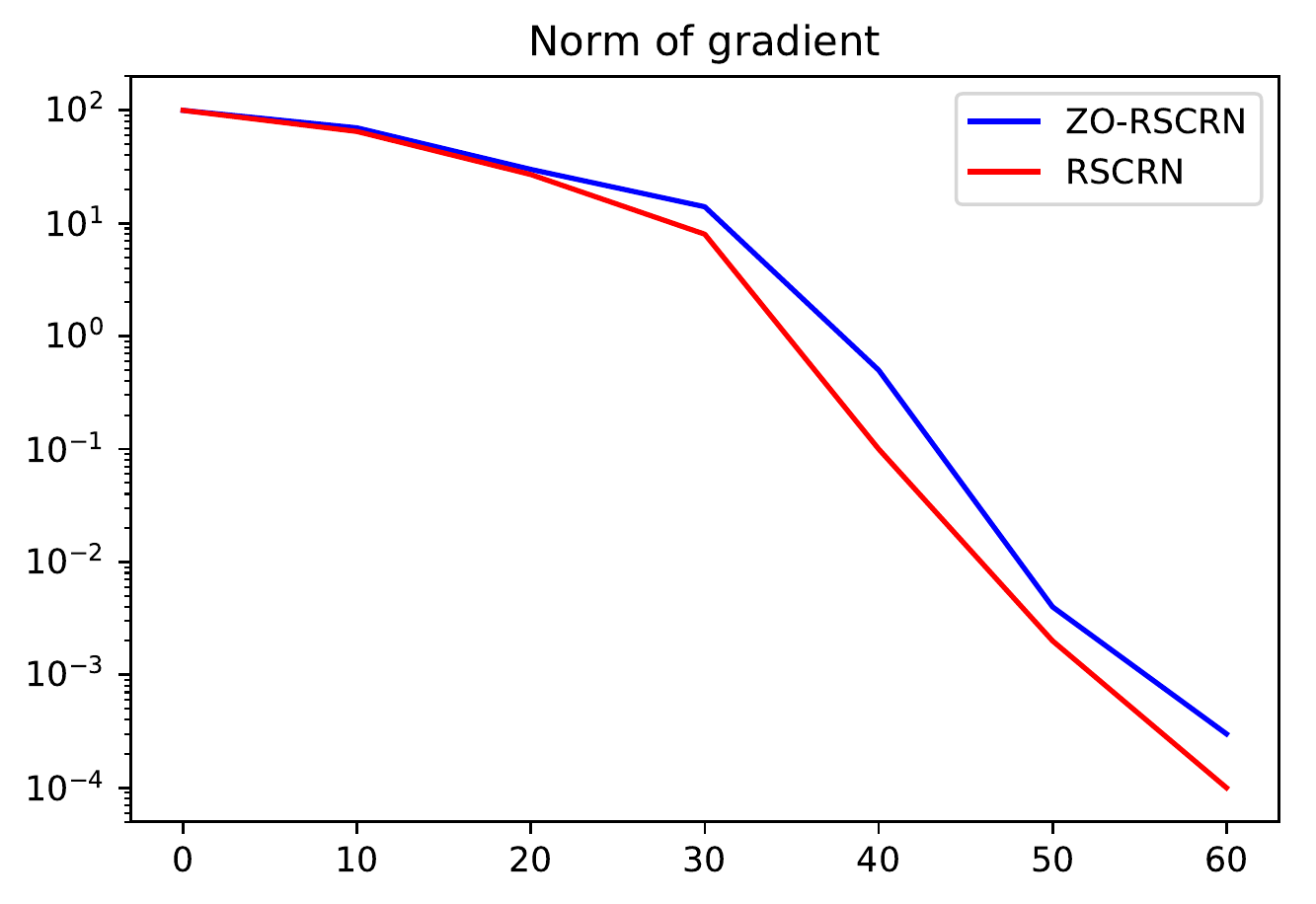}}
\subfigure[2d case]{\includegraphics[clip, trim=4.5cm 9cm 4.5cm 8cm,width=0.30\columnwidth]{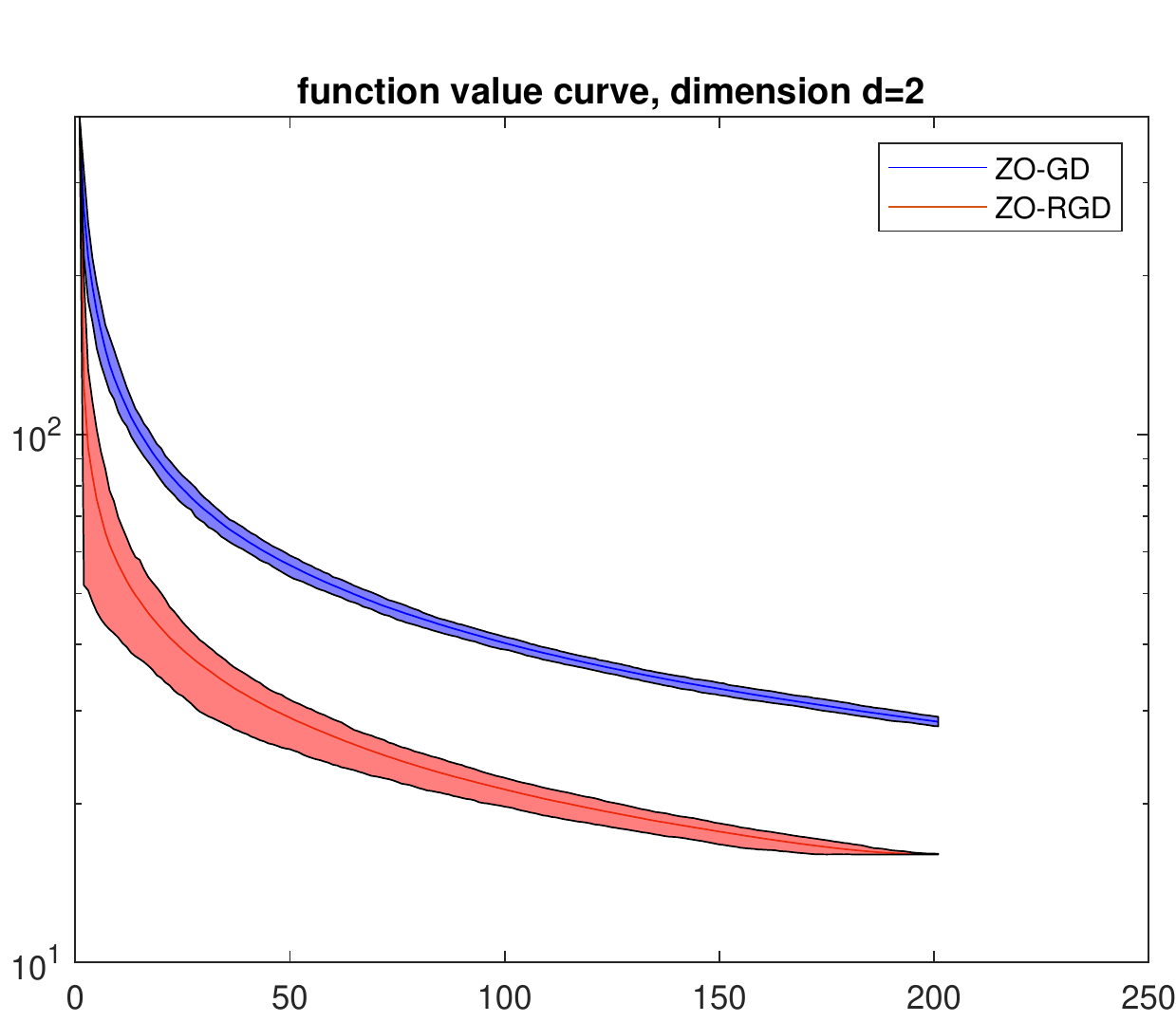}}
\subfigure[3d case]{\includegraphics[clip, trim=4.5cm 9cm 4.5cm 8cm,width=0.30\columnwidth]{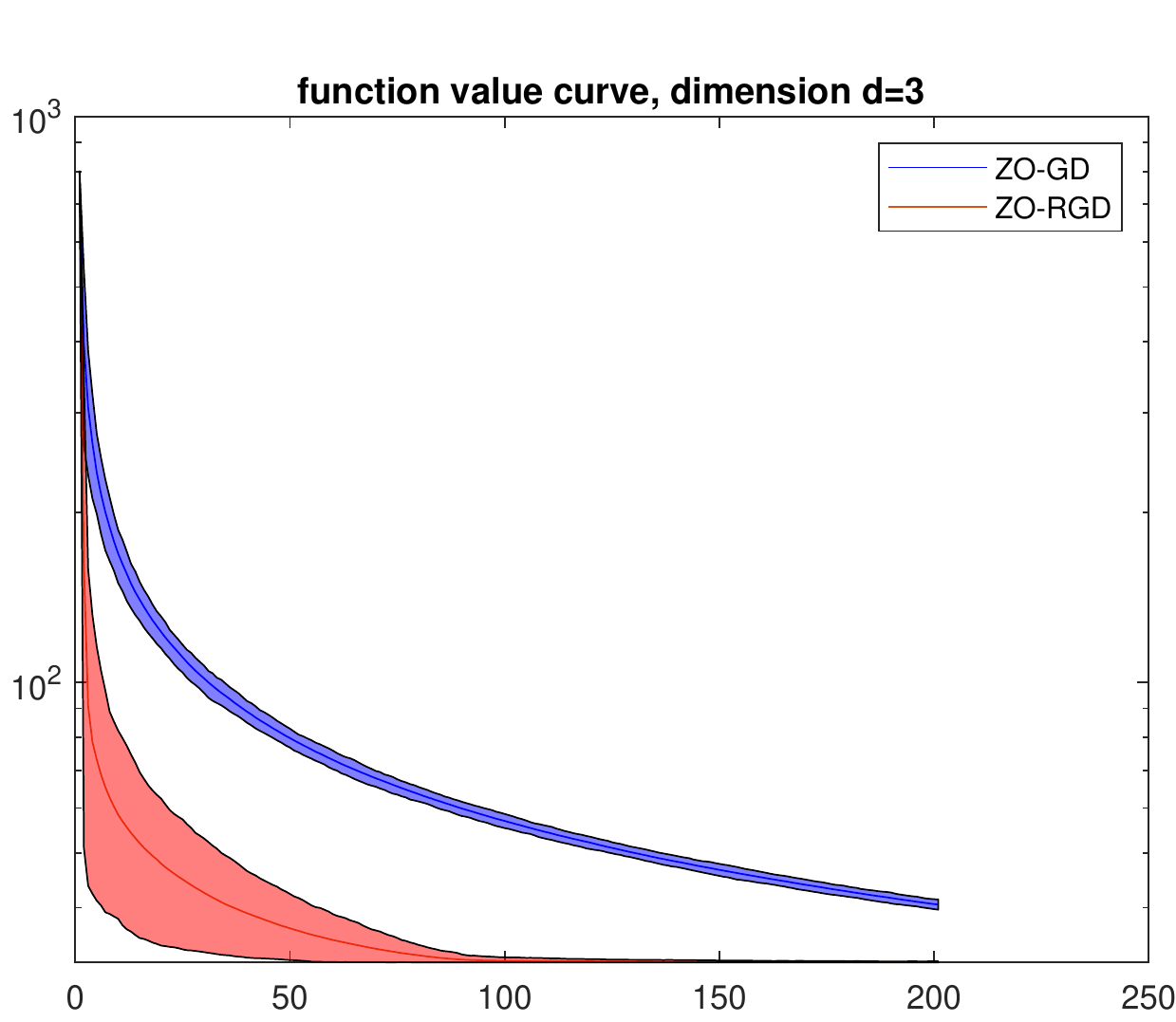}}
\caption{Figure (a) corresponds to Experiment 5. Figures (b) and (c) correspond to the experiments on the robotic minimization function in \eqref{robotics_function}. The $x$-axis in all figures correspond to iteration number. }
\label{fig:robotics}
\end{center}
\end{figure}

\noindent {\bf Zeroth-order black-box attack on Deep Neural Networks (DNNs).}  
We now return to the motivating example described in Section~\ref{sec:zoattackapp} and propose our black-box attack algorithm, as stated in Algorithm~\ref{algorithm_black_attack} (in Appendix \ref{sec:blackboxdetails}). For the sake of comparison, we also assume the architecture of the DNN is known and use ``white-box" attacks based on first-order Riemannian optimization methods (Algorithm \ref{algorithm_white_attack}) and compare against the PGD attack~\cite{madry2017towards}, which does not explicitly enforce any constraints on the perturbed training data. For simplicity, we assume the manifold is a sphere. That is, we assume that the perturbation set $S$ is given by $S(R)=\{\delta:\|\delta\|_2=R\}$, where $R$ is the radius of the sphere. This is consistent with the optimal $\ell_2$-norm attack studied in the literature \cite{lyu2015unified}. Furthermore, the sphere constraint guarantees that the perturbed image is always in a certain distance from the original image. We start our zeroth-order attack from a perturbation and maximize the loss function on the sphere. For the black-box method, to accelerate the convergence, we use Euclidean zeroth-order optimization to find an appropriate initial perturbation (Algorithm \ref{algorithm_black_attack_pre}). It is worth noting that the zeroth-order attack in \cite{chen2017zoo,tu2019autozoom} has a non-smooth objective function, which has $\mathcal{O}({n^3}/{\epsilon^3})$ complexity to guarantee convergence \cite{nesterov2017random}, whereas the complexity needed for our method is  $\mathcal{O}({d}/{\epsilon^2})$. 

We first tested our method on the giant panda picture in the Imagenet data set \cite{imagenet_cvpr09}, with the network structure the Inception v3 network structure \cite{szegedy2016rethinking}. The attack radius in our algorithm is proportional to the $\ell_2$ norm of the original image. Both white-box and black-box Riemannian attacks are successful, which means that they both converge to images that lie in a different image class (i.e. with a different label), see Figure \ref{fig:panda_attack}. We also tested Algorithms \ref{algorithm_white_attack} and \ref{algorithm_black_attack} on the CIFAR10 dataset, and the network structure we used is the VGG net \cite{simonyan2014very}. The corresponding results are provided in Appendix \ref{sec:blackboxdetails}.

\begin{figure}[t!]
\begin{center}
\subfigure[Original image]{\includegraphics[width=0.24\columnwidth]{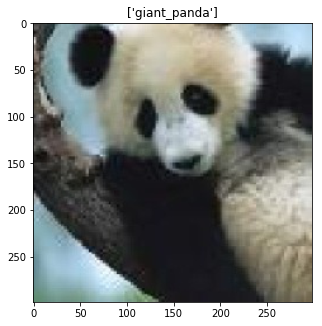}}
\subfigure[PGD attack]{\includegraphics[width=0.24\columnwidth]{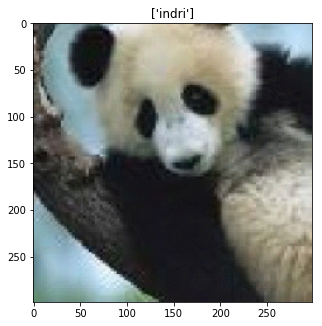}}
\subfigure[First-order attack on the sphere]{\includegraphics[width=0.24\columnwidth]{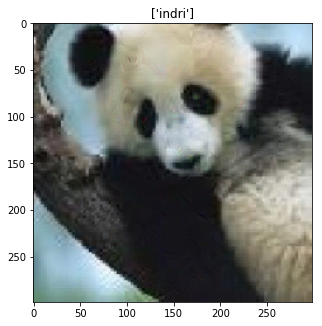}}
\subfigure[Zeroth-order attack on the sphere]{\includegraphics[width=0.24\columnwidth]{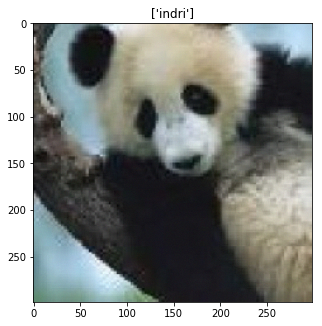}}

\caption{The attack on giant panda picture \cite{imagenet_cvpr09}. (a): the original image; (b): the PGD attack with a small diameter; (c) Riemannian attack (Algorithm \ref{algorithm_white_attack}) on the sphere with the same diameter; (d): Riemannian zeroth-order attack (Algorithm \ref{algorithm_black_attack}). 'Indri' refers to the class to which the original image is misclassified to.}
\label{fig:panda_attack}
\end{center}
\end{figure}

\section{Conclusions}

In this paper, we proposed zeroth-order algorithms for solving Riemannian optimization over submanifolds embedded in Euclidean space in which only noisy function evaluations are available for the objective. These algorithms adopt new estimators of the Riemannian gradient and Hessian from noisy objective function evaluations, based on a Riemannian version of the Gaussian smoothing technique.
The proposed estimators overcome the difficulty of the non-linearity of the manifold constraint and the issues that arise in using Euclidean Gaussian smoothing techniques when the function is defined only over the manifold. The iteration complexity and oracle complexity of the proposed algorithms are analyzed for obtaining an appropriately defined $\epsilon$-stationary point or $\epsilon$-approximate local minimum. The established complexities are independent of the dimension of the ambient Euclidean space and only depend on the intrinsic dimension of the manifold. Numerical experiments demonstrated that the proposed zeroth-order algorithms are comparable to their first-order counterparts.


%

\appendix
\section{Geodesically Convex Problem}\label{sec:geoconvex}
In this section we consider the smooth problem \eqref{prob:nonconvex-smooth} where $f$ is geodesically convex. The definition of geodesic convexity is given below (see, e.g., \cite{zhang2016first}).
\begin{definition}
    A function $f:\M\rightarrow\R$ is geodesically convex if for all $x,y\in\M$, there exists a geodesic $\gamma$ such that $\gamma(0)=x$, $\gamma(1)=y$ and $\forall t\in[0,1]$ we have $f(\gamma(t))\leq (1-t)f(x)+t f(y)$. 
\end{definition}
It can be shown that this definition is equivalent to, $f(\Exp_x(\eta))\geq f(x)+\langle g_x, \eta \rangle_x,\ \forall \eta\in T_x\M$, where $g_x$ is a subgradient of $f$ at $x$, $\Exp$ is the exponential mapping, and $\langle\cdot,\cdot\rangle_x$ is the inner product in $T_x\M$ induced by Riemannian metric $d(\cdot,\cdot)$. When $f$ is smooth, we have $g_x=\grad f(x)$, the Riemannian gradient at $x$. It is known that geodesically convex function is a constant on compact manifolds. Therefore, in this subsection, we assume that $\M$ is an Hadamard manifold \cite{bishop1969manifolds,gromov1978mianiifolds}, and $\X$ is a bounded and geodesically convex subset of $\M$.

\begin{assumption}\label{geo_cvx}
    The subset $\X$ of Hadamard manifold $\M$ is bounded by diameter $D$, and the sectional curvature is lower bounded by $\varrho$. The function $F(x,\xi)$ is geodesically convex w.r.t. $x\in\M$, almost everywhere for $\xi$ (and hence $f$ is geodesically convex).
\end{assumption}

The following lemma from \cite{zhang2016first} is useful for our subsequent analysis. Here $\mathcal{P}_{\X}$ denotes the projection onto $\X$, i.e., $\mathcal{P}_{\mathcal{X}}(x):=\left\{y \in \mathcal{X}: d(x, y)=\inf _{z \in \mathcal{X}} d(x, z)\right\}$.
\begin{lemma}[\cite{zhang2016first}]\label{lemma_geo_cvx}
    For any Riemannian manifold $\M$ where the sectional curvature is lower bounded by $\varrho$ and any points $x$, $x_s\in\M$, the update 
    \[x_{s+1}=\mathcal{P}_{\X}(\Exp_{x_s}(-\eta_s g_s))\] 
    satisfies: $\langle -g_s, x-x_s \rangle\leq \frac{1}{2 \eta_s}(d^2(x_s,x)-d^2(x_{s+1},x))+\frac{\zeta(\varrho,d(x_s,x)) \eta_s}{2}\|g_s\|^2$,
    where $d(\cdot,\cdot)$ is the Riemannian metric defined globally on $\M$, and $\zeta(\varrho,c):= {c\sqrt{|\varrho|}}/{\tanh(c\sqrt{|\varrho|})}$.
\end{lemma}

In this subsection, we consider the following algorithm, which is a special case of \cref{algorithm2}. 
\begin{equation}\label{algorithm_geo_cvx}
    x_{k+1}=\mathcal{P}_{\X}(\Exp_{x_k}(-\eta_k \bar{g}_{\mu,\xi} (x_k))).
\end{equation}

We now present our result for obtaining an $\epsilon$-optimal solution of \eqref{prob:nonconvex-smooth}.

\begin{theorem}\label{thm:geo_cvx}
    Let the manifold $\M$ and the function $f:\M\rightarrow\R$ satisfy Assumptions \ref{L-manifold-smooth}, \ref{sto_assumption}, and \ref{geo_cvx}. Suppose \cref{algorithm2} is run with the update in \cref{algorithm_geo_cvx} and with $\eta_k={1}/{L_g}$. Denote $\Delta_k=\E_{\mathcal{U}_k,\Xi_k}(f(x_k)-f^*$). To have $\min_{1\leq k\leq t}\Delta_k\leq \epsilon$, we need the smoothing parameter $\mu$, number of sampling at each iteration $m$ and the number of iteration $t$ to  be respectively of order:
    \begin{equation}\label{thm:geo_cvx-param}
        \mu = \mathcal{O}({\sqrt{\epsilon}}/{d^{3/2}}),\ m = \mathcal{O}({d}/{\epsilon}),\ t=\mathcal{O}({1}/{\epsilon}).
    \end{equation}
    Hence, the zeroth-order oracle complexity is $N = mt= \mathcal{O}({d}/{\epsilon^2})$.
\end{theorem}
\begin{proof}{Proof of \cref{thm:geo_cvx}}
    From \cref{L-manifold-smooth} we have that:
    \begin{equation*}
        f(x_{k+1})-f(x_k)\leq -\eta_k\langle \grad f(x_k), \bar{g}_{\mu,\xi} (x_k) \rangle+\frac{L_g}{2}\eta_k^2 \|\bar{g}_{\mu,\xi} (x_k)\|^2.
    \end{equation*}
    Taking $\eta_k=\frac{1}{L_g}$, we have
    \begin{equation*}
    \begin{split}
        f(x_{k+1})-f(x_k)&\leq \frac{1}{2L_g}\left(-2\langle \grad f(x_k), \bar{g}_{\mu,\xi} (x_k) \rangle+ \|\bar{g}_{\mu,\xi} (x_k)\|^2\right) \\
        &=\frac{1}{2L_g}\left(\|\bar{g}_{\mu,\xi} (x_k)-\grad f(x_k)\|^2-\|\grad f(x_k)\|^2\right).
    \end{split}
    \end{equation*}
    Taking expectation with respect to $u_k$ on both sides of the inequality above and taking $m\geq 16(d+4)$, we have (by \cref{stochastic_ineq})
    \begin{equation}\label{geo_cvx_proof_1}
    \begin{split}
        & \E_{u_k}f(x_{k+1})-f(x_k) \\
        \leq & \frac{1}{2L_g}\left(\mu^2 L_g^2(d+6)^3+\frac{8(d+4)}{m}\sigma^2+\left(\frac{8(d+4)}{m}-1\right)\|\grad f(x_k)\|^2\right)\\
        \leq & \frac{\mu^2 L_g(d+6)^3}{2}+\frac{4(d+4)}{m L_g}\sigma^2 -\frac{1}{4 L_g}\|\grad f(x_k)\|^2.
    \end{split}
    \end{equation}
     Now considering the geodesic convexity and \cref{lemma_geo_cvx}, we have
    \begin{align}\label{thm:geo_cvx-eq1}
        f(x_{k+1})-f^*  \leq \langle -\bar{g}_{\mu,\xi} (x_k), \Exp_{x_k}^{-1}(x^*) \rangle \leq \frac{L_g}{2}(d^2(x_k,x^*)-d^2(x_{k+1},x^*))+\frac{\zeta(\varrho, D)\|\bar{g}_{\mu,\xi} (x_k)\|^2}{2 L_g}.
    \end{align}
    From \cref{lemma:stocgradzor} we have 
    \begin{align}\label{thm:geo_cvx-eq2}
    & \E \|\bar{g}_{\mu,\xi} (x_k)\|^2  \leq  2\E \|\bar{g}_{\mu,\xi} (x_k)-\grad f(x_k)\|^2+2\E\|\grad f(x_k)\|^2 \\  \leq & 2\mu^2L_g^2(d+6)^3+\frac{16(d+4)}{m}\sigma^2+\left(\frac{16(d+4)}{m}+2\right)\|\grad f(x_k)\|^2.\nonumber
    \end{align}
    Now take the expectation w.r.t. $u_k$ for both sides of \eqref{thm:geo_cvx-eq1}, and combine with \eqref{thm:geo_cvx-eq2}, we have 
    \begin{align}\label{geo_cvx_proof_2}
        \Delta_{k+1}\leq \frac{L_g}{2}(d^2(x_k,x^*)-d^2(x_{k+1},x^*)) +\frac{\zeta(\varrho, D)}{2 L_g}\left(2\mu^2L_g^2(d+6)^3+\frac{16(d+4)}{m}\sigma^2+3\|\grad f(x_k)\|^2\right). 
    \end{align}
    Multiplying \eqref{geo_cvx_proof_2} with $\frac{1}{6\zeta(\varrho, D)}$, and sum up with \cref{geo_cvx_proof_1}, we have
    \begin{equation*}
        \left(1+\frac{1}{6\zeta}\right)\Delta_{k+1} - \Delta_{k} \leq \frac{L_g}{12\zeta}(d^2(x_k,x^*)-d^2(x_{k+1},x^*)) + \mu^2L_g(d+6)^3+\frac{16(d+4)}{3m L_g}\sigma^2.
    \end{equation*}
    Summing it over $k=0,\ldots,t-1$ we have
    \begin{equation*}
        \Delta_{t} - \Delta_{0}+\frac{1}{6\zeta}\sum_{k=1}^{t}\Delta_{k}\leq \frac{L_g}{12\zeta}d^2(x_0,x^*) + (\mu^2L_g(d+6)^3+\frac{16(d+4)}{3m L_g}\sigma^2)t.
    \end{equation*}
    Equivalently, we have
    \begin{equation*}
        \frac{1}{t}\sum_{k=1}^{t}\Delta_{k}\leq \frac{L_g}{2 t}d^2(x_0,x^*) + 6\zeta (\mu^2L_g(d+6)^3+\frac{16(d+4)}{3m L_g}\sigma^2) + \frac{6\zeta}{t}\Delta_{0},
    \end{equation*}
    which together with \eqref{thm:geo_cvx-param} yields $\min_{1\leq k\leq t}\Delta_k\leq\epsilon$. 
\end{proof}

\section{Proof of \cref{rmk1}}
\begin{proof} {Proof of the improved bound \cref{multi_sample_bound}}

Since $\bar{g}_\mu(x)=\frac{1}{m}\sum_{i=1}^{m}g_{\mu,i}(x)$, we have (denote $\mathcal{U}=\{u_1,...,u_m\}$):
\begin{align*}
    &\E_{\mathcal{U}}\|\bar{g}_\mu(x)-\grad f(x)\|^2\\
    \leq & 2\E_{\mathcal{U}}\|\bar{g}_\mu(x)-\E_{\mathcal{U}}\bar{g}_\mu(x)\|^2+2\|\E_{\mathcal{U}}\bar{g}_\mu(x)-\grad f(x)\|^2 \\
    = & 2\E_{\mathcal{U}}\left\|\frac{1}{m}\sum_{i=1}^{m}\left[g_{\mu,i}(x)-\E_{\mathcal{U}}g_{\mu,i}(x)\right]\right\|^2+2\left\|\frac{1}{m}\sum_{i=1}^{m}\left[\E_{\mathcal{U}}g_{\mu,i}(x)-\grad f(x)\right]\right\|^2 \\
    = & \frac{2}{m^2}\E_{\mathcal{U}}\sum_{i=1}^{m}\|g_{\mu,i}(x)-\E_{\mathcal{U}}g_{\mu,i}(x)\|^2+\frac{2}{m^2}\left\|\sum_{i=1}^{m}\left[\E_{\mathcal{U}}g_{\mu,i}(x)-\grad f(x)\right]\right\|^2 \\
    \leq & \frac{2}{m}\E_{u_1}\|g_{\mu,1}(x)-\E_{\mathcal{U}}g_{\mu,1}(x)\|^2+2\|\E_{u_1}g_{\mu,1}(x)-\grad f(x)\|^2 \\
    \leq & \frac{2}{m}\E_{u_1}\|g_{\mu,1}(x)\|^2+\frac{\mu^2 L_g^2}{2}(d+3)^3 \\
    \leq & \frac{\mu^2}{m}L_g^2(d+6)^3+\frac{4(d+4)}{m}\|\grad f(x)\|^2+\frac{\mu^2 L_g^2}{2}(d+3)^3 \\
    \leq & \mu^2L_g^2(d+6)^3+\frac{4(d+4)}{m}\|\grad f(x)\|^2,
\end{align*}
where the second equality is from the fact that $u_i$ and $u_j$ are independent when $i\not=j$.
\end{proof}

\begin{proof}{Proof of \eqref{multi_sample_bound-2}}
    Following the $L_g$-retraction-smooth, we have: $f(x_{k+1})\leq f(x_k)-\eta_k\langle \bar{g}_{\mu}(x),\grad f(x_k) \rangle+\frac{\eta_k^2 L_g}{2}\|\bar{g}_{\mu}(x)\|^2$. Taking $\eta_k=\hat{\eta}={1}/{L_g}$, we have
    \begin{align*}
        f(x_{k+1})&\leq f(x_k)-\eta_k\langle \bar{g}_{\mu}(x),\grad f(x_k) \rangle+\frac{\eta_k^2 L_g}{2}\|\bar{g}_{\mu}(x)\|^2 \\
        &=f(x_k)+\frac{1}{2 L_g}\left(\|\bar{g}_{\mu}(x)-\grad f(x)\|^2- \|\grad f(x)\|^2\right).
    \end{align*}
    Now take the expectation for the random variables at the iteration $k$ on both sides, we have
    \begin{align*}
        \E_{k} f(x_{k+1})&\leq f(x_k)+\frac{1}{2 L_g}\left(\E_{k}\|\bar{g}_{\mu}(x)-\grad f(x)\|^2- \|\grad f(x)\|^2\right) \\
        \cref{stochastic_ineq}&\leq f(x_k)+\frac{1}{2 L_g}\left(\mu^2 L_g^2(d+6)^3+\left(\frac{4(d+4)}{m}-1\right)\|\grad f(x)\|^2\right).
    \end{align*}
    By choosing $m\geq 8(d+4)$, summing the above inequality over $k=0,\ldots,N$ gives \eqref{multi_sample_bound-2}.    
\end{proof}

\section{Proof of \cref{lemma:stocgradzor}}\label{sec:lemma4.2}
\begin{proof}{}
For the sake of notation, here we denote $\E=\E_{u_0}$. 
From \eqref{stocashtic_oracle} we have
\begin{equation}\label{lemma:stocgradzor-eq1}
\E(\|g_{\mu,\xi}(x)\|^2)=\frac{1}{\mu^2}\E\left[ (F(R_x(\mu u,\xi))-F(x,\xi))^2\|u\|^2 \right].
\end{equation} 
From \cref{L-manifold-smooth} we have
\begin{equation}\label{lemma:stocgradzor-eq2}
\begin{split}
&(F(R_x(\mu u,\xi))-F(x,\xi))^2 \\
=& (F(R_x(\mu u,\xi))-F(x,\xi)-\mu\langle\grad F(x,\xi),u\rangle+\mu\langle\grad F(x,\xi),u\rangle)^2 \\ 
\leq & 2\left(\frac{L_g}{2}\mu^2\|u\|^2\right)^2+2\mu^2 \langle\grad F(x,\xi),u\rangle^2.
\end{split}
\end{equation}
Combining \eqref{lemma:stocgradzor-eq1} and \eqref{lemma:stocgradzor-eq2} yields 
    \begin{equation}\label{lemma:stocgradzor-eq3}
      \begin{split}
        \E(\|g_{\mu,\xi}(x)\|^2)&\leq \frac{\mu^2}{2}L_g^2\E(\|u\|^6)+2\E(\| \langle\grad F(x,\xi),u\rangle u\|^2) \\
        (\cref{corr:rmoment})&\leq \frac{\mu^2}{2}L_g^2(d+6)^3+2\E(\| \langle\grad F(x,\xi),u\rangle u\|^2).
    \end{split}  
    \end{equation}
    Denote our $d$-dimensional tangent space as $\X$. Without loss of generality, suppose $\X$ is the subspace generated by projecting onto the first $d$ coordinates, i.e., $\forall x\in\X$, the last $(n-d)$ elements of $x$ are zeros. Also for brevity, denote $g=\grad F(x,\xi)$. Use $x_i$ to denote the $i$-th coordinate of $u_0$, and $\kappa(d)$ denote the normalization constant for $d$-dimensional Gaussian distribution. For simplicity, denote $x=(x_1,...,x_d)$. We have
    \begin{equation}\label{lemma:stocgradzor-eq4}
        \begin{split}
        & \E(\| \langle\grad F(x,\xi),u\rangle u\|^2) =  \frac{1}{\kappa}\int_{\R^n}\langle\grad F(x,\xi),u\rangle^2 \|u\|^2 e^{-\frac{1}{2}\|u_0\|^2}d u_0 \\
        = &\frac{1}{\kappa(d)}\int_{\R^d}\left(\sum_{i=1}^{d}g_i x_i\right)^2 \left(\sum_{i=1}^{d}x_i^2\right) e^{-\frac{1}{2}\sum_{i=1}^{d}x_i^2}d x_1\cdots d x_d,\\
        = & \frac{1}{\kappa(d)}\int_{\R^d}\langle g, x\rangle^2 \|x\|^2 e^{-\frac{1}{2}\|x\|^2}d x = \frac{1}{\kappa(d)}\int_{\R^d}\|x\|^2 e^{-\frac{\tau}{2}\|x\|^2}\langle g, x\rangle^2 e^{-\frac{1-\tau}{2}\|x\|^2}d x \\
        \leq & \frac{2}{\kappa(d)\tau e}\int_{\R^d}\langle g, x\rangle^2 e^{-\frac{1-\tau}{2}\|x\|^2}d x =\frac{2}{\kappa(d)\tau(1-\tau)^{1+d/2} e}\int_{\R^d} \langle g, x\rangle^2 e^{-\frac{1}{2}\|x\|^2}d x \\ = & \frac{2}{\tau(1-\tau)^{1+d/2} e}\|g\|^2 \leq  (d+4)\|g\|^2,
        \end{split}
    \end{equation}
    where the last $n-d$ dimensions of $u_0$ are integrated to be one,  the first inequality is due to the following fact: $x^p e^{-\frac{\tau}{2}x^2}\leq (\frac{2}{\tau e})^{p/2}$, and the second inequality follows by setting $\tau = \frac{2}{(d+4)}$. 
From \cref{sto_assumption}, we have 
\begin{equation}\label{lemma:stocgradzor-eq5}
\E_{\xi}\|\grad F(x,\xi)\|^2 \leq 2\E_{\xi}\|\grad F(x,\xi) - \grad f(x)\|^2 + 2\|\grad f(x)\|^2  \leq 2\sigma^2+2\|\grad f(x)\|^2.
\end{equation}
Combining \eqref{lemma:stocgradzor-eq3}, \eqref{lemma:stocgradzor-eq4} and \eqref{lemma:stocgradzor-eq5} yields 
    \begin{equation}\label{temp_lemma}
    \begin{split}
        \E_{\xi}\left[\E_{u_0}(\|g_{\mu,\xi}(x)\|^2)\right]&\leq \E_{\xi}\left[\frac{\mu^2}{2}L_g^2(d+6)^3+2(d+4)\|\grad F(x,\xi)\|^2\right]\\
        &\leq \frac{\mu^2}{2}L_g^2(d+6)^3+4(d+4)(\sigma^2+\|\grad f(x)\|^2).
    \end{split}
    \end{equation}
Finally, we have     
    \begin{align*}
        &\E_{\mathcal{U},\xi}\|\bar{g}_{\mu,\xi}(x)-\grad f(x)\|^2\\
        \leq & 2\E_{\mathcal{U},\xi}\|\bar{g}_{\mu,\xi}(x)-\E_{\mathcal{U},\xi}g_{\mu}(x)\|^2+2\|\E_{\mathcal{U},\xi}g_{\mu}(x)-\grad f(x)\|^2 \\
        \leq & \frac{2}{m}\E_{u_1,\xi_1}\|g_{\mu_1,\xi_1}(x)-\E_{u_1,\xi_1}g_{\mu_1,\xi_1}(x)\|^2+\frac{\mu^2 L_g^2}{2}(d+3)^3 \\
        \leq & \frac{2}{m}\E_{u_1,\xi_1}\|g_{\mu_1,\xi_1}(x)\|^2+\frac{\mu^2 L_g^2}{2}(d+3)^3 \\
        \leq & \frac{2}{m}\left( \frac{\mu^2 L_g^2}{2}(d+6)^3 + 4(d+4)[\|\grad f(x)\|^2+\sigma^2] \right)+\frac{\mu^2 L_g^2}{2}(d+3)^3 \\
        \leq & \mu^2 L_g^2(d+6)^3+\frac{8(d+4)}{m}\sigma^2+\frac{8(d+4)}{m}\|\grad f(x)\|^2,
    \end{align*}
    where the second inequality is from \cref{lemma123} and the fourth inequality is from \eqref{temp_lemma}.
\end{proof}

\section{Implementation Details of Black-Box Attacks}\label{sec:blackboxdetails}

\begin{figure}[t]
\begin{center}
\subfigure[]{\includegraphics[width=0.3\columnwidth]{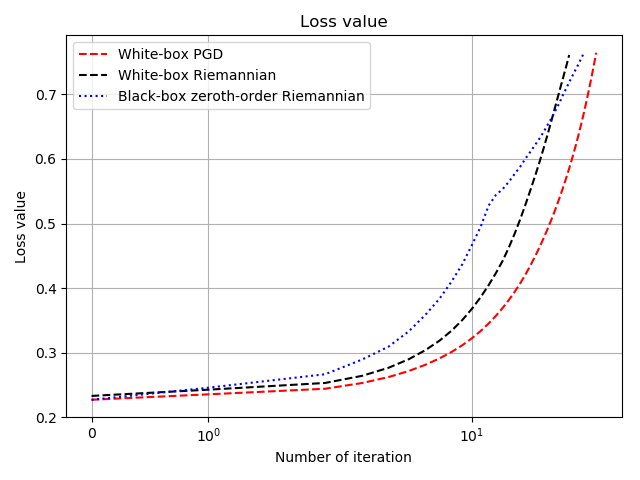}}
\subfigure[]{\includegraphics[width=0.3\columnwidth]{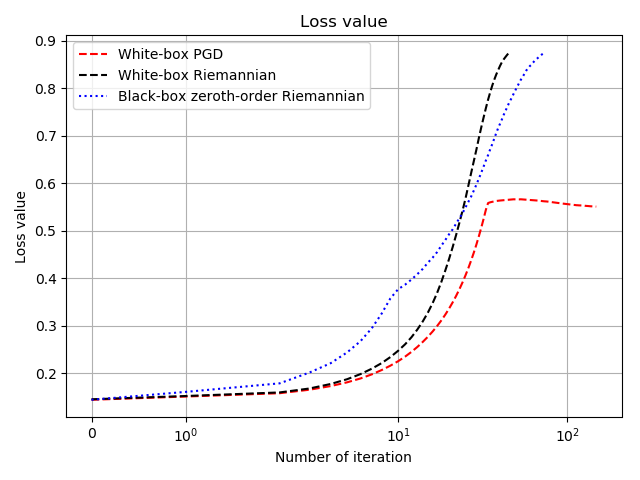}}
\subfigure[]{\includegraphics[width=0.3\columnwidth]{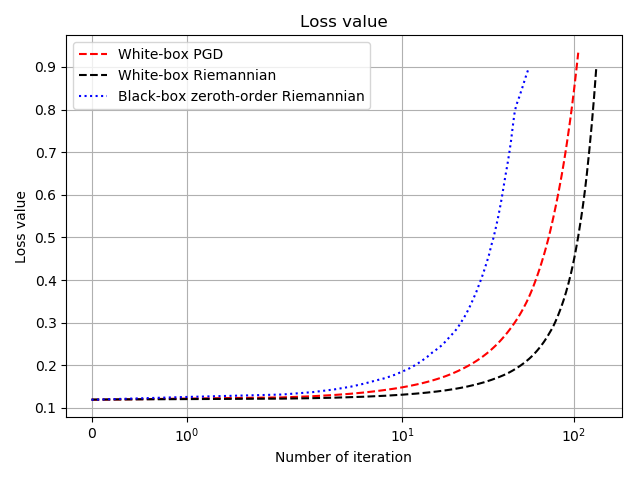}}

\caption{Loss function versus the iteration numbers. We observe that the loss function increases while performing our attacks. The three figures correspond to the last three rows in Figure \ref{fig:attack_result}. For the failed the PGD attack, we notice that the function value stuck in the middle and then decreased, while white and black-box Riemannian attack increased loss value successfully.}
\label{fig:loss_curve_attack}
\end{center}
\end{figure}

Here we provide our white and black-box Riemannian attack algorithm in Algorithms \ref{algorithm_white_attack} and \ref{algorithm_black_attack}, respectively. For the black-box attack, to accelerate convergence, we introduce a pre-attack step to search for a sufficiently large loss value on the prescribed sphere, still in a black-box manner (only use the function value). For further acceleration of the black-box attack, the hierarchical attack \cite{chen2017zoo} and the auto-encoder technique \cite{tu2019autozoom} might be applicable. The attack results on CIFAR-10 images are shown in Figure \ref{fig:attack_result}. We also provide the loss function curve in Figure \ref{fig:loss_curve_attack}. Again the network structure we used is the VGG net \cite{simonyan2014very}. It can be seen from Figure \ref{fig:attack_result} that our black-box attack yields similar attack result as the PGD attack, however PGD failed in one of the five images, and the loss function curve in Figure \ref{fig:loss_curve_attack} indicates that the loss function curve of PGD attack stagnated. This may be due to an inappropriate choice of parameters for the PGD attack. Nevertheless this experiment shows the ability of our Riemannian attack method. We also mention here that the Riemannian attack is in general slower than the PGD attack due to the multi-sampling technique discussed in \cref{rmk1}.

\begin{algorithm}[!ht]
   \caption{White-box attack via Riemannian optimization}
   \label{algorithm_white_attack}
\begin{algorithmic}[1]
   \STATE {\bfseries Input:} Original image $\tilde{x}$, original label $y_0$, radius of the attack region $R$, step size $\eta_k$, some convergence criterion.
   \STATE Randomly sample $\delta$ s.t. $\|\delta\|=R$.
   \STATE Set the initial point $x_0=\tilde{x}+\delta$, $k=0$.
   \REPEAT
   \STATE Update $x_{k+1}=R_{x_k}(\eta_k \grad_x L(\theta,x,y))$, $k = k+1$.
   \UNTIL{Convergence criterion is met.}
\end{algorithmic}
\end{algorithm}

\begin{algorithm}[!ht]
   \caption{Black-box attack via Riemannian zeroth-order optimization}
   \label{algorithm_black_attack}
\begin{algorithmic}[1]
   \STATE {\bfseries Input:} Original image $\tilde{x}$, original label $y_0$, radius of the attack region $R$, step size $\eta_k$, smoothing parameter $\mu$, number of multi-sample $m$, some convergence criterion.

   \STATE Obtain $\delta$ (initial perturbation) by pre-training steps (Algorithm \ref{algorithm_black_attack_pre}).
   \STATE Set the initial point $x_0=\tilde{x}+\delta$, $k=0$.
   \REPEAT
   \STATE Sample $m$ standard Gaussian random matrix $u_i$ on $T_{x_{k}}\M$.
   \STATE Set the random oracle $\bar{g}_\mu(x_{k})$ by (\ref{stocashtic_oracle}).
   \STATE Update $x_{k+1}=R_{x_k}(\eta_k \bar{g}_\mu(x_{k}))$, $k = k+1$.
   \UNTIL{Convergence criterion is met.}
\end{algorithmic}
\end{algorithm}

\begin{algorithm}[!ht]
   \caption{Pre-training step for black-box attack}
   \label{algorithm_black_attack_pre}
\begin{algorithmic}[1]
   \STATE {\bfseries Input:} Original image $\tilde{x}$, original label $y_0$, radius of the attack region $R$, step size $\eta_k$, smoothing parameter $\mu$, number of multi-sample $m$.

   \STATE $x_0=\tilde{x}$.
   \REPEAT
   \STATE Sample $m$ standard Gaussian random matrix $u_i$.
   \STATE Set the random oracle $\bar{g}_\mu(x_{k})$ by (\ref{stocashtic_oracle}), with
            $g_{\mu_i}(x)=\frac{f(x+\mu u_i)-f(x)}{\mu} u_i$
   \STATE Update $x_{k+1}=x_k+\eta_k \bar{g}_\mu(x_{k})$.
   \STATE Update $\delta = x_{k+1}-\tilde{x}$, $k = k+1$.
   \UNTIL{$\|\delta\|\geq R$}

   \STATE $\delta=\frac{\delta}{\|\delta\|} R$
\end{algorithmic}
\end{algorithm}

\begin{figure}[t]
\small
\begin{center}
\subfigure[]{\includegraphics[width=0.20\columnwidth]{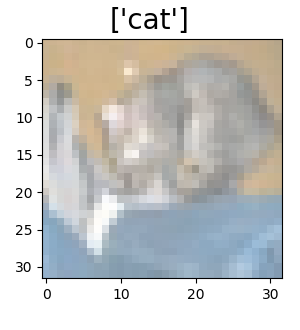}}
\subfigure[]{\includegraphics[width=0.20\columnwidth]{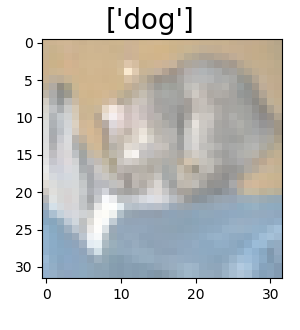}}
\subfigure[]{\includegraphics[width=0.20\columnwidth]{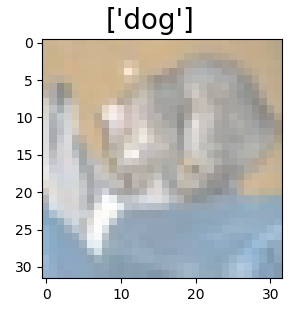}}
\subfigure[]{\includegraphics[width=0.20\columnwidth]{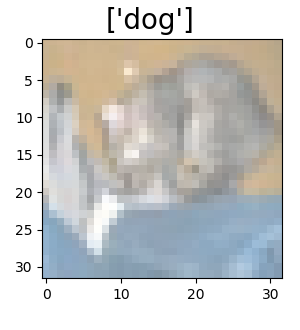}}

\subfigure[]{\includegraphics[width=0.20\columnwidth]{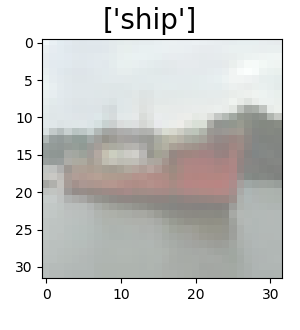}}
\subfigure[]{\includegraphics[width=0.20\columnwidth]{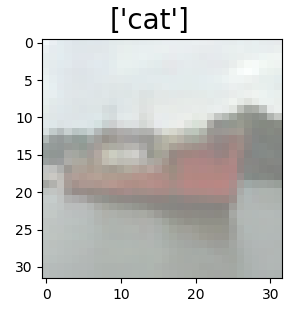}}
\subfigure[]{\includegraphics[width=0.20\columnwidth]{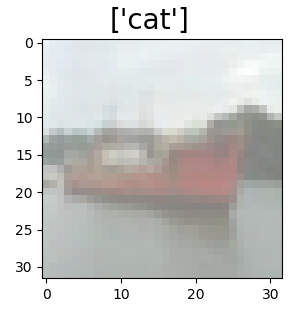}}
\subfigure[]{\includegraphics[width=0.20\columnwidth]{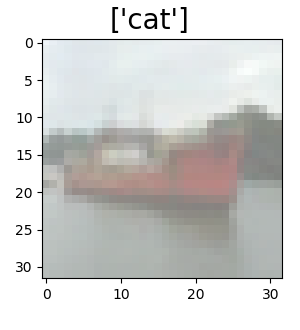}}

\subfigure[]{\includegraphics[width=0.20\columnwidth]{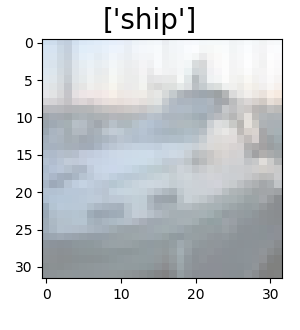}}
\subfigure[]{\includegraphics[width=0.20\columnwidth]{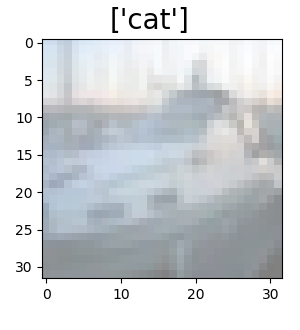}}
\subfigure[]{\includegraphics[width=0.20\columnwidth]{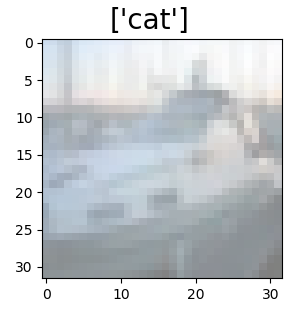}}
\subfigure[]{\includegraphics[width=0.20\columnwidth]{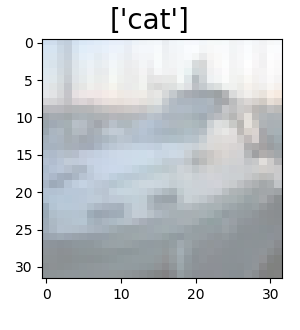}}

\subfigure[]{\includegraphics[width=0.20\columnwidth]{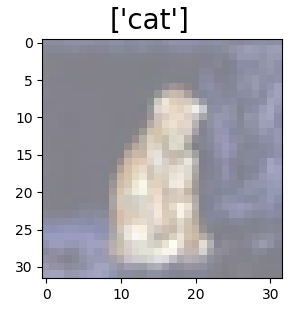}}
\subfigure[]{\includegraphics[width=0.20\columnwidth]{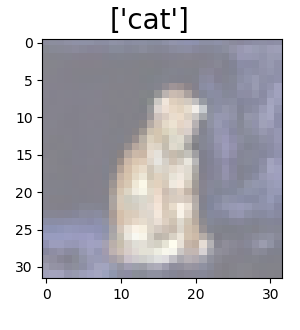}}
\subfigure[]{\includegraphics[width=0.20\columnwidth]{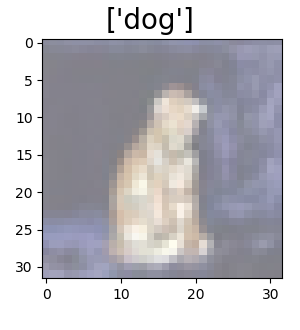}}
\subfigure[]{\includegraphics[width=0.20\columnwidth]{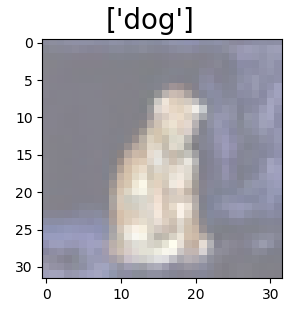}}

\subfigure[]{\includegraphics[width=0.20\columnwidth]{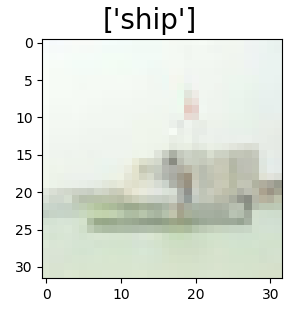}}
\subfigure[]{\includegraphics[width=0.20\columnwidth]{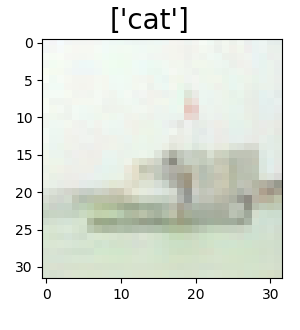}}
\subfigure[]{\includegraphics[width=0.20\columnwidth]{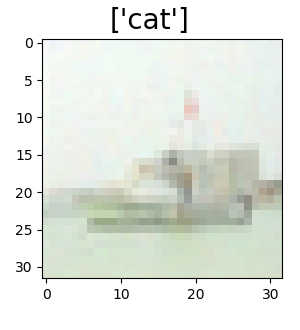}}
\subfigure[]{\includegraphics[width=0.20\columnwidth]{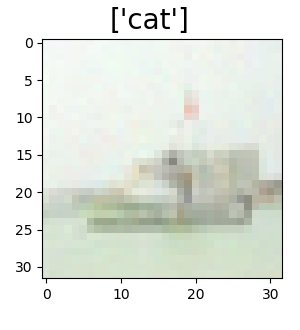}}
\caption{The attack on CIFAR10 picture. From left to right columns: the original image; the PGD attack with a small diameter; white box Riemannian attack on the sphere with the same diameter; black box Riemannian attack on the sphere with the same diameter. Notice that for the figure in the fourth row, the PGD attack failed while the Riemannian attacks succeeded. The diameter is set to be $0.01$ times the norm of the original images.}
\label{fig:attack_result}
\end{center}
\end{figure}

%
%

\section*{Acknowledgments.}
JL and SM acknowledge the support by NSF grants DMS-1953210 and CCF-2007797. KB and SM acknowledge the support by UC Davis CeDAR (Center for Data Science and Artificial Intelligence Research) Innovative Data Science Seed Funding Program. JL also wants to thank Tesi Xiao for helpful discussions on black-box attacks.\bibliographystyle{amsalpha}
\bibliography{references}
\end{document}